\documentclass[11pt,letterpaper,one side]{amsart}
\pdfoutput=1
\usepackage{mathmode}
\usepackage{bm,yhmath}
\usepackage{tikz,tikz-cd}
\usepackage{enumitem}
\usepackage{graphicx,array}
\usepackage{stackengine,scalerel}
\usepackage{amsthm}
\usepackage{multirow}
\usepackage{longtable}
\usepackage{ragged2e}
\graphicspath{ {./figures/} }
\def\PAA{profinitely almost auto-symmetric}

\def\XX{\mathscr{G}}
\def\Z{\mathbb{Z}}
\def\MCG{\mathrm{MCG}}
\def\[#1{[\![#1]\!]}
\def\abs#1{\left |#1 \right |}
\def\im{\mathrm{im}}
\def\Tor{\mathrm{tor}}
\def\Aut{\mathrm{Aut}}

\def\Fp{\mathbb{F}_p}
\def\O{\mathcal{O}}
\def\Stab{\mathrm{Stab}}
\def\tto{\longrightarrow}
\def\ttt{\xrightarrow{\tiny \cong}}
\def\T{\hat{T}}
 \DeclareRobustCommand{\longleftmapsto}{\text{\reflectbox{$\longmapsto$}}}
\def\quotient#1#2{%
    \raise1ex\hbox{$#1$}\Big/\lower1ex\hbox{$#2$}%
}
\def\quo#1#2{%
    \raise0.6ex\hbox{$#1$}\big/\lower0.6ex\hbox{$#2$}%
}
\def\Hom{\mathrm{Hom}}
\def\F{{f^\dagger}}
\def\R{\mathbb{R}}

\def\limi{\varprojlim}

\def\Zx{\widehat{\Z}^{\times}}

\def\hfree#1{H_1({#1},\Z)_{\mathrm{free}}}
\def\ab#1{{#1}^{\text{ab}}}
\def\Ab#1{{#1}^{\text{Ab}}}

\def\tensor{\widehat{\mathbb{Z}}\otimes_{\mathbb{Z}}}

\newcommand\ttimes{\mathbin{\ThisStyle{\ensurestackMath{%
  \stackengine{-1\LMpt}{\SavedStyle\times}
  {\SavedStyle_{\hstretch{.9}{\mkern1mu\sim}}}{O}{c}{F}{T}{S}}}}}
\tikzset{%
  symbol/.style={
    draw=none,
    every to/.append style={
      edge node={node [sloped, allow upside down, auto=false]{$#1$}}
    },
  },
}
\def\CircleArrowright{\ensuremath{%
\reflectbox{\rotatebox[origin=c]{180}{$\circlearrowleft$}}}}
\newtheorem*{THMA}{\autoref{Mainthm: SF rigid}}
\newtheorem*{THMB}{\autoref{Mainthm: SF non-rigid}}
\newtheorem*{THMC}{\autoref{Mainthm: Almost rigid}}
\newtheorem*{CORM}{\autoref{Maincor: Almost rigid}}

\def\dt{(DT)}

\makeatletter
\DeclareRobustCommand\bigop[1]{%
  \mathop{\vphantom{\sum}\mathpalette\bigop@{#1}}\slimits@
}
\newcommand{\bigop@}[2]{%
  \vcenter{%
    \sbox\z@{$#1\sum$}%
    \hbox{\resizebox{\ifx#1\displaystyle.9\fi\dimexpr\ht\z@+\dp\z@}{!}{$\m@th#2$}}%
  }%
}

\newcommand{\addresseshere}{%
  \enddoc@text\let\enddoc@text\relax
}

\makeatother

\newcommand{\bigK}{\DOTSB\bigop{\#}}

\title{Profinite almost rigidity in 3-manifolds}
\author{Xiaoyu Xu}
\address{Beijing International Center for Mathematical Research\\
Peking University\\
 Beijing 100871, P.R. China}
\email{xuxiaoyu@stu.pku.edu.cn}

\begin{document}
\begin{sloppypar}
\maketitle

\begin{abstract}
We prove that any compact, orientable 3-manifold with empty or  toral boundary is profinitely almost rigid among all compact, orientable  3-manifolds. 
     In other words, 
the profinite completion of its fundamental group determines its homeomorphism type to finitely many possibilities.  
%
    Moreover, the profinite completion of the fundamental group of a mixed 3-manifold together with the peripheral structure uniquely determines the homeomorphism type of its Seifert part, i.e. the maximal graph manifold components in the JSJ-decomposition.  
    On the other hand,  
without assigning the peripheral structure, the profinite completion of a mixed 3-manifold group may not uniquely determine the fundamental group of its Seifert part. 
    The proof is based on JSJ-decomposition.
\end{abstract}

\setcounter{tocdepth}{1}
\tableofcontents 

\section{Introduction}\label{SEC: Intro}

It is a useful idea to distinguish the properties of a finitely generated group $G$ through  the collection of its finite quotient groups, denoted by $\mathcal{C}(G)$. When $G=\pi_1(M^3)$ is a 3-manifold group, $\mathcal{C}(G)$ corresponds to the lattice of finite regular coverings of $M$. It is known that among finitely generated groups, the profinite completion encodes full data of finite quotients.
\begin{definition}\label{DEF: Profinite completion}
Let $G$ be a group, the \textit{profinite completion} of $G$ is defined as $$\widehat{G}=\limi\limits_{N\lhd_f G} \quo{G}{N}\;,$$ where $N$ ranges over all finite-index normal subgroups of $G$.
\end{definition}
\begin{fact}[{\cite{DFPR82}}]\label{Finite quotient}
    For two finitely generated groups $G_1$ and $G_2$, $\mathcal{C}(G_1)=\mathcal{C}(G_2)$ if and only if $\widehat{G_1}\cong \widehat{G_2}$.
\end{fact}

The geometrization theorem implies that a 3-manifold is largely determined by its fundamental group. For instance, the fundamental group of a closed, orientable, prime 3-manifold determines its homeomorphism type except for the special case of lens spaces, see \cite[Theorem 2.1.2]{AFW15}. 
As we are studying the fundamental group $\pi_1M^3$ through its finite quotients $\mathcal{C}(\pi_1M)$, it is natural to ask whether $\mathcal{C}(\pi_1M)$ (or equivalently, $\widehat{\pi_1M}$) determines the homeomorphism type of $M$. 
This question has been formulated as the profinite rigidity within 3-manifolds.

\begin{definition}\label{indef: almost rigidity}
Let $\mathscr{M}$ be a class of  3-manifolds. For any $M\in \mathscr{M}$, we denote $$\Delta_{\mathscr{M}}(M)=\{ N\in \mathscr{M}\mid \widehat{\pi_1N}\cong \widehat{\pi_1M}\}.$$ 
We say that a manifold $M$ is \textit{profinitely rigid} in $\mathscr{M}$ if $\Delta_{\mathscr{M}}(M)=\{M\}$, and $M$ is \textit{profinitely almost rigid} in $\mathscr{M}$ if $\Delta_{\mathscr{M}}(M)$ is finite.  
The class $\mathscr{M}$ itself is called \textit{profinitely (almost) rigid} if every $M\in \mathscr{M}$ is profinitely (almost)  rigid in $\mathscr{M}$. 
\end{definition}

To avoid the trivial ambiguities, we always include the following convention.
\begin{convention}\label{conv1}
In our context, we always assume that a compact 3-manifold has no boundary spheres. Indeed, the boundary spheres can be excluded through capping 3-balls while preserving the fundamental group.
\end{convention}



Within 3-manifolds, the question of profinite rigidity is closely related with geometrization. 
In a series of works, Wilton-Zalesskii \cite{WZ17,WZ17b,WZ19} showed that among compact, orientable, irreducible 3-manifolds with empty or incompressible toral boundary, the profinite completion determines whether the 3-manifold is geometric; it further determines the geometry type within the geometric case, and determines the prime decomposition and the JSJ-decomposition in the non-geometric case. 

For geometric 3-manifolds, Grunewald-Pickel-Segal \cite{GPS80} showed the profinite almost rigidity of $Sol$-manifolds, while Stebe \cite{Stebe} and more examples by Funar \cite{Fun13} showed that the class of $Sol$-manifolds is not profinitely rigid.  Wilkes \cite{Wil17} showed the profinite almost rigidity of closed orientable Seifert fibered spaces, including the profinite rigidity in $\mathbb{S}^2\times\mathbb{E}^1$, $\mathbb{E}^3$, $Nil$ and ${ \widetilde{\mathrm{SL}(2,\mathbb{R})}}$  geometries, and the only non-rigid exceptions are those  constructed by Hempel \cite{Hem14} within $\mathbb{H}^2\times \mathbb{E}^1$ geometry.
Indeed, the non-rigid examples \cite{Fun13,Hem14,Stebe} are all given by commensurable  surface bundles over circle. Recently, Liu \cite{Liu23} proved the profinite almost rigidity of finite-volume hyperbolic 3-manifolds. In addition, Bridson, Reid, McReynolds and Spitler \cite{BMRS20,BR22} have found various examples of profinitely rigid closed hyperbolic 3-manifolds, including the Weeks manifold and the $\mathrm{Vol(3)}$ manifold. Indeed, they showed that the fundamental groups of these hyperbolic 3-manifolds are actually profinitely rigid among all finitely generated, residually finite groups. However, whether profinite rigidity holds in the hyperbolic case is still unknown. 

For non-geometric 3-manifolds, Wilkes \cite{Wil18} proved that closed graph manifolds are profinitely almost rigid but not profinitely rigid, 
and he presented a complete profinite classification for graph manifolds, including many profinitely rigid examples. Therefore, the  crucial topic remaining for profinite almost rigidity in 3-manifolds is the mixed manifolds (including purely hyperbolic ones).

%

Within 3-manifolds, two well-known theorems exhibit the relationship between the fundamental group and the homeomorphism type of a 3-manifold.
\begin{theorem}[{\cite[Corollary 6.5]{Wal68}}]\label{Wald}
The fundamental group of a compact, Haken, boundary-incompressible 3-manifold together with its peripheral structure (i.e. the conjugacy classes of the peripheral subgroups) uniquely determines the homeomorphism type of this 3-manifold.
\end{theorem}
\begin{theorem}
\label{Joh}
Among compact, orientable 3-manifolds, the  fundamental group determines the homeomorphism type  to finitely many possibilities, but possibly not unique.
\end{theorem}
\autoref{Joh} was proven by Johannson  {\cite[Corollary 29.3]{Joh79}}  for irreducible  3-manifolds with non-empty incompressible boundary, and independently by Swarup \cite{Swa80} for irreducible  3-manifolds with   boundary, without the boundary-incompressibility assumption. The full version of \autoref{Joh} then follows from the validity of the Kneser's conjecture \cite{Hei72}, together with the geometrization theorem for the closed case, see \cite[Theorem 2.1.2]{AFW15}.
%

The goal of this paper is to establish a partial analog of these two theorems in the profinite setting.

Let us first consider \autoref{Wald}. Accordingly, in dealing with 3-manifolds with boundary, it is worthwhile to investigative a specific class of profinite isomorphisms, i.e. those respecting the peripheral structure.


\begin{definition}
For compact 3-manifolds $M$ and $N$ with (possibly empty) incompressible boundary, an isomorphism $f:\widehat{\pi_1M}\ttt \widehat{\pi_1N}$ {\em respects the peripheral structure} if there is a one-to-one correspondence between the boundary components of $M$ and $N$, simply denoted as $\partial_iM\leftrightarrow\partial_{i}N$, such that $f$ sends $\overline{\pi_1\partial_iM}$, the closure of a peripheral subgroup, to a conjugate of $\overline{\pi_1\partial_{i}N}$ in $\widehat{\pi_1N}$.
\end{definition}

We also agree that any profinite isomorphism between closed 3-manifolds automatically respects the peripehral structure. 
 
Indeed, as is shown in \autoref{THM: finite quotients with marked subgroups}, the profinite completion of a compact 3-manifold group $\widehat{\pi_1M}$ together with its peripheral structure encodes the same data as the ``peripheral pair quotients'' of $\pi_1M$, i.e. the finite quotients of $\pi_1M$ marked with the conjugacy classes of peripheral subgroups, denoted by $\mathcal{C} ( \pi_1 M; \pi_1\partial_1M,\cdots,\pi_1\partial_mM)$ (see \autoref{DEF: Finite quotient marked with subgroups}). 

Our first result can be viewed as a profinite version of \autoref{Wald} in the Seifert part of mixed 3-manifolds.  
 
%
In the following context, a {\em mixed 3-manifold} denotes a compact,  orientable and irreducible 3-manifold with empty or incompressible toral boundary, whose JSJ-decomposition is non-trivial and contains at least one   hyperbolic piece. 
The {\em Seifert part} of a mixed 3-manifold is the union of its maximal graph manifold components in its JSJ-decomposition. By convention,   the Seifert part of a purely-hyperbolic manifold is defined as an empty space.

\begin{mainthm}\label{Mainthm: SF rigid}
Suppose $M$ and $N$ are mixed 3-manifolds, with empty or toral boundary. If $\widehat{\pi_1M}\cong \widehat{\pi_1N}$ through an isomorphism respecting the peripheral structure, then the Seifert parts of $M$ and $N$ are homeomorphic.
\end{mainthm}

Note that any profinite isomorphism between closed manifolds does respect the peripheral structure by our definition. Hence, we have the following corollary.

\begin{maincor}\label{maincor: Seifert}
Among closed mixed 3-manifolds, the profinite completion of the fundamental group determines the homeomorphism type of the Seifert part. 
\end{maincor}

Any profinite isomorphism between finite-volume hyperbolic 3-manifolds respects the peripheral structure (\autoref{LEM: Hyp preserving peripheral}). However, this is not always true for Seifert fibered spaces, as for isomorphisms between their original fundamental groups. Indeed, flexibility may occur  when the profinite isomorphism does not respect the peripheral structure.

Despite \autoref{Mainthm: SF rigid}, from which $\mathcal{C}(\pi_1M;\pi_1\partial_1M,\cdots,\pi_1\partial_nM)$ determines the homeomorphism type of the Seifert part within mixed manifolds, $\mathcal{C}(\pi_1M)$ itself may not even determine the fundamental group of the Seifert part as shown by the following theorem. This shows sharp contrast with the classical case \cite{SS01}, where $\pi_1M$, without assigning the peripheral structure, is sufficient to determine the fundamental group of the Seifert part, through the group theoretical JSJ-decomposition.

\begin{mainthm}\label{Mainthm: SF non-rigid}
There exists a pair of irreducible, mixed 3-manifolds $M$, $N$ with incompressible toral boundary, such that $\widehat{\pi_1M}\cong \widehat{\pi_1N}$ by an isomorphism not respecting the peripheral structure, while the fundamental groups of the  Seifert part of $M$ and $N$ are not isomorphic.
\end{mainthm}

In combination with \cite{SS01}, this implies $\pi_1M\not\cong \pi_1N$, which disproves the group-level profinite rigidity in the fundamental group of mixed 3-manifolds.

Our next result is a profinite version of \autoref{Joh}. 
Though the absolute profinite rigidity does not hold for the fundamental group of 3-manifolds with  toral boundary by \autoref{Mainthm: SF non-rigid}, the profinite completion of fundamental group, or equivalently $\mathcal{C}(\pi_1M)$, still determines the homeomorphism type of such a 3-manifold to  finitely many possibilities. We remind the readers that the following result does not require the peripheral structure.

\begin{mainthm}\label{Mainthm: Almost rigid}
Any compact, orientable 3-manifold with empty or toral boundary is profinitely almost rigid in the class of all compact, orientable 3-manifolds.
\end{mainthm}

The following two tables conclude current  results for profinite rigidity and almost rigidity in 3-manifolds based on geometric decomposition.

\begin{table}[H]
\centering
\renewcommand*{\arraystretch}{2}
{
\resizebox{\columnwidth}{!}{%
\begin{tabular}{|cccc|}
\hline
\multicolumn{4}{|c|}{\textbf{\scalebox{1.1}{Closed orientable 3-manifolds}}}                                                                                                                                                                                                                                                   \\ \hline
\multicolumn{1}{|c|}{\textbf{Class of manifolds}}             & \multicolumn{1}{c|}{\textbf{Profinite almost rigidity}} & \multicolumn{2}{c|}{\textbf{Profinite rigidity}}                                                                                                                                     \\ \hline
\multicolumn{1}{|c|}{$Sol$-manifolds}                         & \multicolumn{1}{c|}{\scalebox{0.95}{True  \cite{GPS80}}}                               & \multicolumn{2}{c|}{False \cite{Stebe,Fun13}}                                                                                                                                                           \\ \hline
\multicolumn{1}{|c|}{\multirow{3}{*}{Seifert fibered spaces}} & \multicolumn{1}{c|}{\multirow{3}{*}{True \cite{Wil17}}}              & \multicolumn{1}{c|}{$\mathbb{H}^2\times \mathbb{E}^1$ geometry}                                                                                                            & False \cite{Hem14}  \\ \cline{3-4} 
\multicolumn{1}{|c|}{}                                        & \multicolumn{1}{c|}{}                                   & \multicolumn{1}{c|}{\scalebox{0.95}{\renewcommand*{\arraystretch}{1.3}\begin{tabular}[c]{@{}c@{}}$\mathbb{E}^3$, $Nil$, $\mathbb{S}^2\times \mathbb{E}^1$, \\ $\widetilde{\mathrm{SL}(2,\mathbb{R})}$ geometry\end{tabular}}} & True \cite{Wil17}    \\ \cline{3-4} 
\multicolumn{1}{|c|}{}                                        & \multicolumn{1}{c|}{}                                   & \multicolumn{1}{c|}{$\mathbb{S}^3$ geometry}                                                                                                                               & Trivial \\ \hline
\multicolumn{1}{|c|}{Hyperbolic manifolds}                    & \multicolumn{1}{c|}{True \cite{Liu23}}                               & \multicolumn{2}{c|}{{Unknown}}                                                                                                                                                         \\ \hline
\multicolumn{1}{|c|}{{Graph manifolds}}        & \multicolumn{1}{c|}{{True \cite{Wil18}}}              & \multicolumn{2}{c|}{False$\,+\,$complete classification \cite{Wil18}}                                                                                                                                                          \\ \hline
\multicolumn{1}{|c|}{\multirow{2}{*}{Mixed manifolds}}        & \multicolumn{1}{c|}{\multirow{2}{*}{ {True  (\autoref{Mainthm: Almost rigid})} }}              & \multicolumn{1}{c|}{Entire manifold}                                                                                                                                       & {Unknown} \\ \cline{3-4} 
\multicolumn{1}{|c|}{}                                        & \multicolumn{1}{c|}{}                                   & \multicolumn{1}{c|}{Seifert part}                                                                                                                                          & {True (\autoref{maincor: Seifert})  }   \\ \hline
\end{tabular}%
}
}
\end{table}

As for bounded 3-manifolds, in order to exclude the ambiguity that  non-homeomorphic bounded 3-manifolds may have isomorphic fundamental groups, we say a class of manifolds $\mathscr{M}$ has {\em group-level profinite rigidity} if for any $M,N\in \mathscr{M}$, $\widehat{\pi_1M}\cong \widehat{\pi_1N}$ implies $\pi_1M\cong \pi_1N$.

\begin{table}[H]
\centering
\renewcommand*{\arraystretch}{2.3}
{
\resizebox{\columnwidth}{!}{%
\begin{tabular}{|cccc|}
\hline
\multicolumn{4}{|c|}{\textbf{\scalebox{1.1}{Compact orientable 3-manifolds with incompressible toral boundary}}}                                                                                                                                                                                                               \\ \hline
\multicolumn{1}{|c|}{\textbf{Class of manifolds}}             & \multicolumn{1}{c|}{\textbf{Profinite almost rigidity}} & \multicolumn{2}{c|}{\textbf{Group-level profinite rigidity}}                                                                                                                         \\ \hline
\multicolumn{1}{|c|}{Seifert fibered spaces}                  & \multicolumn{1}{c|}{True (\autoref{PROP: Seifert almost rigid})}                               & \multicolumn{2}{c|}{True (\autoref{COR: Seifert rigid})}                                                                                                                                                            \\ \hline
\multicolumn{1}{|c|}{Finite-volume hyperbolic manifolds}      & \multicolumn{1}{c|}{True \cite{Liu23}}                               & \multicolumn{2}{c|}{ {Unknown}}                                                                                                                                                         \\ \hline
\multicolumn{1}{|c|}{Graph manifolds}                         & \multicolumn{1}{c|}{True (\autoref{THM: Graph mfd almost rigid})}                               & \multicolumn{2}{c|}{False \cite{Wil18}}                                                                                                                                                           \\ \hline
\multicolumn{1}{|c|}{Mixed manifolds}                         & \multicolumn{1}{c|}{ {True (\autoref{Mainthm: Almost rigid})}}                               & \multicolumn{2}{c|}{{False  (\autoref{Mainthm: SF non-rigid})}}                                                                                                                                                           \\ \hline
\end{tabular}%
}
}
\end{table}

\subsection{Ingredients of proof}
\subsubsection{Congruent isomorphism}

Wilton-Zalesskii \cite{WZ19} (see also Wilkes \cite{WilkesJSJ,Wil19} for the bounded case) proved the profinite detection of geometric decomposition in 3-manifolds. In this paper, we reformulate their results using the language of profinite Bass-Serre theory. 

\begin{definition}[Congruent isomorphism]
\label{indef: congruent}
A \textit{congruent isomorphism} $f_{\bullet}$ between two finite graphs of (profinite) groups $(\mathcal{G},\Gamma)$ and $(\mathcal{H},\Gamma')$ consists of:
\begin{enumerate}[leftmargin=*]
\item an isomorphism of oriented graphs $\F: \Gamma\stackrel{\cong}{\longrightarrow} \Gamma'$;
\item  an isomorphism of (profinite) groups $f_z: \mathcal{G}_z\stackrel{\cong}{\longrightarrow} \mathcal{H}_{\F(z)}$ for each $z\in \Gamma$;
\item  a coherence relation at each edge $e\in E(\Gamma)$, i.e. commutative by composing with a conjugation 
\begin{equation*}
\centering
\begin{tikzcd}[column sep=large]
\mathcal{G}_e \arrow[dd, "\cong","f_e"'] \arrow[rr, "\varphi_i^e"] &                                                  & \mathcal{G}_{d_i(e)} \arrow[dd, "\cong"',"f_{d_i(e)}"] \\
 & \CircleArrowright &\\
\mathcal{H}_{\F(e)} \arrow[r, "(\varphi')_i^{\F(e)} "]     & \mathcal{H}_{\F(d_i(e))} \arrow[r, "C_{h_i^e}", "\cong"'] & \mathcal{H}_{\F(d_i(e))}               
\end{tikzcd}
\end{equation*}
where $C_{h_i^e}(x)=h_i^e\cdot x\cdot (h_i^e)^{-1}$ denotes the conjugation by $h_i^e\in \mathcal{H}_{\F(d_i(e))}$, and $i=0,1$.
\end{enumerate}
\end{definition} 

\begin{theorem}[cf. \autoref{THM: Profinite Isom up to Conj}]\label{inthm: JSJ}
Suppose $M,N$ are two orientable, irreducible 3-manifolds with empty or incompressible toral boundary, which are not closed $Sol$-manifolds. Then, any isomorphism $\widehat{\pi_1M}\cong  \widehat{\pi_1N}$ induces a congruent isomorphism between their JSJ-graphs of profinite groups $(\widehat{\mathcal{G}_M},\Gamma_M)$ and $(\widehat{\mathcal{G}_N},\Gamma_N)$, after reassigning orientation of the JSJ-graphs. In addition, the congruent isomorphism matches up the hyperbolic/Seifert fibered type of the vertices.
\end{theorem}

\subsubsection{Peripheral $\Zx$-regularity}

In light of \autoref{inthm: JSJ}, it is crucial to analyze the properties of a profinite isomorphism $f_v:\widehat{\pi_1M_v}\ttt\widehat{\pi_1N_v}$ between the JSJ-pieces when restricted to the closure of a peripheral subgroup $\overline{\pi_1\partial_iM_v}$. This implies information about the gluing map at this JSJ-torus, as shown by the following commutative diagram.
\begin{equation}\label{indig1}
    \begin{tikzcd}[column sep=tiny]
\widehat{\pi_1M_u} \arrow[r,symbol=\supseteq] \arrow[dd, "C_{h_0}\circ f_u"'] & \overline{\pi_1\partial_kM_u} \arrow[dd] & & & & \widehat{\pi_1T^2} \arrow[llll, "\varphi_0"'] \arrow[rrrr, "\varphi_1"]   & & & & \overline{\pi_1\partial_iM_v} \arrow[r,symbol=\subseteq] \arrow[dd] & \widehat{\pi_1M_v} \arrow[dd, "C_{h_1}\circ f_v"] \\
 & & & &  & & & & & &\\
\widehat{\pi_1N_u} \arrow[r,symbol=\supseteq ]                              & \overline{\pi_1\partial_lN_u}      & & &     & \widehat{\pi_1T^2} \arrow[llll, "\varphi_0'"'] \arrow[rrrr, "\varphi_1'"] & & & & \overline{\pi_1\partial_jN_v} \arrow[r,symbol=\subseteq]           & \widehat{\pi_1N_v}                              
\end{tikzcd}
\end{equation}

Note that the closure of a peripheral subgroup corresponding to an incompressible toral boundary component can be naturally identified as $\overline{\pi_1\partial_iM}\cong\widehat{\Z}\otimes_{\Z} \pi_1(\partial_iM) $ (cf. \autoref{COR: Peripheral group tensor}). 
We depict a special property of the map $\overline{\pi_1\partial_iM}\cong\widehat{\Z}\otimes_{\Z} \pi_1(\partial_iM)\to \overline{\pi_1\partial_jN}\cong\widehat{\Z}\otimes_{\Z} \pi_1(\partial_jN)$  called peripheral $\Zx$-regularity.

\begin{definition}[Peripheral $\Zx$-regularity]\label{DEF: Intro peripheral regular}
Let $M$ and $N$ be compact, orientable 3-manifolds with incompressible toral boundary. Let $\partial_i M$ be a boundary component of $M$, and let $\pi_1\partial_iM$ be a conjugacy representative of the corresponding peripheral subgroup. Suppose $f:\widehat{\pi_1M}\ttt\widehat{\pi_1N}$ is an isomorphism. For $\lambda\in \Zx$, we say that $f$ is \textit{peripheral $\lambda$-regular} at $\partial_iM$ if the following holds.
\begin{enumerate}[leftmargin=*]
\item There exists a boundary component $\partial_j N$ of $N$, a conjugacy representative of the peripheral subgroup $\pi_1\partial_jN$, and an element $g\in \widehat{\pi_1N}$ such that $f(\overline{\pi_1\partial_iM})=C_{g}^{-1}( \overline{\pi_1\partial_jN})$, where $C_g$ denotes the conjugation by $g$.
\item \label{indef: Zx-regular 2} 
There exists an isomorphism $\psi:\pi_1\partial_iM\ttt\pi_1\partial_jN $ such that the following diagram commutes
\begin{equation*}
\begin{tikzcd}[row sep=large, column sep=large]
\overline{\pi_1\partial_iM} \arrow[r, "C_{g}\circ f|_{\overline{\pi_1\partial_iM}}"] & \overline{\pi_1\partial_jN}                  \\
\tensor \pi_1\partial_iM \arrow[u, "\cong"] \arrow[r, "\lambda\otimes \psi"]         & \tensor \pi_1\partial_jN \arrow[u, "\cong"']
\end{tikzcd}
\end{equation*}
where $\lambda$ denotes the scalar multiplication by $\lambda$ in $\widehat{\Z}$.
\end{enumerate}
For simplicity, we say that $f$ is \textit{peripheral $\Zx$-regular} at $\partial_iM$ if it is peripheral $\lambda$-regular for some $\lambda\in \Zx$.
\end{definition}

\subsubsection{Profinite almost auto-symmetry}
When $f$ is peripheral $\Zx$-regular at $\partial_iM$, there are two main ingredients to be considered
in \autoref{DEF: Intro peripheral regular}~(\ref{indef: Zx-regular 2}), 
namely the coefficient $\lambda\in \Zx$ and the group isomorphism $\psi$. In proof of profinite (almost) rigidity, it is crucial to determine whether $\psi$ can be induced by an ambient homeomorphism $\Phi:M\ttt N$, for which we introduce an accurate definition as follows.

Let $M$ and $N$ be compact, orientable 3-manifolds with incompressible toral boundary. Let $\partial_iM$ be a boundary component of $M$, and let $\pi_1\partial_iM\subseteq \pi_1M$ be a conjugacy representative of the peripheral subgroup. 
Suppose $\Phi: M\to N$ is a homeomorphism, which sends $\partial_iM$ to a boundary component $\partial_jN$. Up to a choice of basepoints, let $\Phi_\ast: \pi_1M\ttt\pi_1N$ denote the induced isomorphism between the fundamental groups, and let $\pi_1\partial_jN=\Phi_\ast(\pi_1\partial_iM)$ be the  corresponding conjugacy representative of the peripheral subgroup.

\begin{definition}\label{inDef: induced by homeo}
Under the above assumptions, 
suppose  $f: \widehat{\pi_1M}\ttt\widehat{\pi_1N}$ is a profinite isomorphism which is peripheral $\Zx$-regular at $\partial_iM$. 
We say that {\em $f$ restricting on $\partial_iM$ is induced by a homeomorphism $\Phi$ with coefficient $\lambda\in \Zx$}, if there exists an element $g\in \widehat{\pi_1N}$ such that   $C_g\circ f$ sends $\overline{\pi_1\partial_iM}$ to $\overline{\pi_1\partial_jN}$, and the isomorphism
$$
C_g\circ f|_{\overline{\pi_1\partial_iM}}: \; \overline{\pi_1\partial_iM}\cong \tensor \pi_1\partial_iM\tto \tensor \pi_1\partial_jN \cong \overline{\pi_1\partial_jN}
$$
can be decomposed as $\lambda\otimes (\Phi_\ast|_{\pi_1\partial_iM})$.
\end{definition}

Let $M$ be a compact, orientable  3-manifold  with incompressible toral boundary, and let $\mathcal{P}=\{\partial_1M,\cdots,\partial_kM\}$ be a collection of some boundary components of $M$. 
We define two subgroups of $\Aut(\widehat{\pi_1M})$: 
{\small
\begin{equation*}
\begin{gathered}
\XX_{\Zx}(M,\mathcal{P}):=\left\{ f: \widehat{\pi_1M}\ttt\widehat{\pi_1M}\left|\;\begin{gathered}\text{for each } 1\le i\le k\text{,}\\ f(\overline{\pi_1\partial_iM})\text{ is a conjugate of } \overline{\pi_1\partial_iM},\\\text{ and }f\text{ is }   \text{peripheral }\Zx\text{-regular at }\partial_iM\end{gathered}\right.\;\right\},\\[1ex]
\XX_{h}(M,\mathcal{P}):=\left\{f\in \XX_{\Zx}(M,\mathcal{P})\left|\, \begin{gathered} \text{there is a homeomorphism }\Phi: M\to M, \text{ such}\\ \text{that }f \text{ restricting }  \text{on each }\partial_iM  \text { is induced by} \\\text{the homeomorphism } \Phi\text{ with coefficient }\lambda_i\in \Zx \end{gathered}\right.\right\}. 
\end{gathered}
\end{equation*}
}
 
We remind the readers that in the definition of $\XX_h(M,\mathcal{P})$, the ambient homeomorphism $\Phi$ is unified for all components $\partial_iM$, while the coefficients $\lambda_i$ might be different from each other.

\begin{definition}[Profinite almost auto-symmetry]\label{indef: PAA}
Under the above assumptions, we say that $M$ is {\em {\PAA} at $\mathcal{P}$} 
if $\XX_{h}(M,\mathcal{P})$ is a finite-index subgroup in $\XX_{\Zx}(M,\mathcal{P})$.
\end{definition}

\subsubsection{Gluing geometric pieces}

Based on \autoref{inthm: JSJ}, the proof for profinite almost rigidity in mixed 3-manifolds consists of two parts. 
Firstly, the profinite almost rigidity in the geometric case implies that 3-manifolds profinitely isomorphic to a given mixed 3-manifold can only be built up from a finite number of possible geometric pieces. Secondly, the peripheral $\Zx$-regularity and the profinite almost auto-symmetry ensure that the gluings between the geometric pieces also fall into finitely many possibilities. 

The key results within this process are listed as follows. 

\begin{theorem}\label{inthm: hyperbolic}
\begin{enumerate}[leftmargin=*]
\item (cf. \cite[Theorem 9.1]{Liu23}) The class of finite-volume hyperbolic 3-manifolds is profinitely almost rigid.
\item (cf. \autoref{THM: Hyperbolic peripheral regular}) Suppose $M$ and $N$ are finite-volume hyperbolic 3-manifolds with cusps. Then, any profinite isomorphism $f:\widehat{\pi_1M}\ttt\widehat{\pi_1N}$ is peripheral $\Zx$-regular at all boundary components of $M$. 
\item (cf. \autoref{COR: Hyperbolic PAA}) Any cusped finite-volume hyperbolic 3-manifold $M$ is {\PAA} at $\partial M$.
\end{enumerate}
\end{theorem}

\begin{theorem}\label{inthm: Graph new}
\begin{enumerate}[leftmargin=*]
\item (cf. \autoref{THM: Graph mfd almost rigid}, and \cite{Wil18} for the closed case) The class of graph manifolds, with empty or toral boundary, is profinitely almost rigid.
\item\label{di222} (cf. \autoref{THM: Graph Zx-regular implies homeo}) Let $M$ and $N$ be graph manifolds with non-empty boundary.  Suppose $f:\widehat{\pi_1M}\ttt\widehat{\pi_1N}$ is an isomorphism respecting the peripheral structure, and that $f$ is peripheral $\widehat{\Z}^{\times}$-regular at (at least) one boundary component of $M$. Then, $M$ and $N$ are homeomorphic.
\item\label{di333} (cf. \autoref{THM: Hgroup finite index}) Let $M$ be a graph manifold with boundary, and let $\mathcal{P}$ be a collection of some boundary components of $M$. 
Then, $M$ is {\PAA} at $\mathcal{P}$. 
\end{enumerate}
\end{theorem}

\subsection{Structure of the paper}
\autoref{SEC: Profinite groups} and \autoref{SEC: Bass-Serre} serve as preliminary materials for this paper. 
\autoref{SEC: Prime decomposition} and \autoref{SEC: JSJ} are further expansions on the  profinite detection of 3-manifold decomposition. In particular, \autoref{inthm: JSJ} is introduced in \autoref{SEC: JSJ}.

The most important concept, peripheral $\Zx$-regularity, is introduced in \autoref{SEC: ZX-regular}, where we verify its ``well-definedness'',  and prove a {gluing lemma}.

In \autoref{SEC: Hyperbolic}, we characterize profinite isomorphisms between hyperbolic 3-manifolds based on the results of Liu \cite{Liu23}, and prove \autoref{inthm: hyperbolic}.

\autoref{SEC: Seifert} studies profinite isomorphisms between Seifert fibered spaces, mostly inspired by Wilkes \cite{Wil17,Wil18}.  
As the building blocks of a graph manifold, this serves as a preparation for \autoref{SEC: Graph Mfd}, where we investigate the profinite properties of graph manifolds possibly with boundary, and prove \autoref{inthm: Graph new}.

We prove the main conclusions in \autoref{SEC: Main}. \autoref{Mainthm: SF rigid} is proven in \autoref{subsec: Seifert rigid}, \autoref{Mainthm: SF non-rigid} is proven in \autoref{SEC: Counter-example}, and finally  \autoref{Mainthm: Almost rigid} is proven in \autoref{subsec: almost rigid}.



\subsection*{Acknowledgement}
 The author is greatly thankful to his advisor Yi Liu for his useful advice, especially in pointing out this topic; and the author also thanks Ruicen Qiu for plotting the figures in this paper.  
The author is grateful to the referees for careful proofreading and valuable comments that really improved the exposition of this paper. 

\section{Profinite groups}\label{SEC: Profinite groups}

A \textit{profinite space} $X=\limi X_i$ is an inverse limit of finite spaces equipped with discrete topology over a directed partially ordered set, and a \textit{profinite group} $G=\limi G_i$ is an inverse limit of finite groups equipped with discrete topology over a directed partially ordered set.

It is shown in \cite[Theorem 1.1.12]{RZ10} that $X$ is a profinite space if and only if $X$ is compact, Hausdorff and totally disconnected, and in \cite[Theorem 2.1.3]{RZ10} that $G$ is a profinite group if and only if $G$ is a topological group whose underlying space is a profinite space.

\subsection{Profinite completion}

Let $G$ be an abstract group. The \textit{profinite topology} on $G$ is generated by the sub-basis consisting of all the (left and right) cosets of   finite-index subgroups in $G$. 
 A subset $E\subseteq G$ is called \textit{separable} if $E$ is closed in the profinite topology. 
 $G$ is called \textit{residually finite} if $\{1\}$ is separable. 
$G$ is called \textit{LERF} if any finitely generated subgroup of $G$ is separable.


 Recall in \autoref{DEF: Profinite completion}, the profinite completion of an abstract group $G$ is defined as a profinite group $$\widehat{G}=\limi\limits_{N\lhd_f G}\quo{G}{N},$$ 
where the notation $\lhd_f$ denotes finite-index normal subgroup. 

There is a canonical homomorphism $\iota: G\to \widehat{G}$, $g\mapsto (gN)_{N\lhd_f G}$, so that $\iota(G)$ is dense in $\widehat{G}$. Since $\mathrm{ker}\,\iota=\bigcap_{N\lhd_f G} N$, $\iota$ is injective if and only if $G$ is residually finite. 
Moreover, the profinite completion is functorial. Any homomorphism of abstract groups $f:H\to G$ induces a homomorphism of profinite groups $\widehat{f}:\widehat{H}\to \widehat{G}$ so that the following diagram commutes.

\begin{equation*}
\begin{tikzcd}
H \arrow[r, "f"] \arrow[d, "\iota_H"'] & G \arrow[d, "\iota_G"] \\
\widehat{H} \arrow[r, "\widehat{f}"]   & \widehat{G}           
\end{tikzcd}
\end{equation*}


We list two useful properties of the completion functor. 

\begin{proposition}[{\cite[Proposition 3.2.5]{RZ10}}]\label{PROP: Completion Exact}
The profinite completion functor is right exact, i.e. for a short exact sequence of abstract groups
\begin{equation*}
\centering
\begin{tikzcd}
1 \arrow[r] & K \arrow[r, "\varphi"] & G \arrow[r, "\psi"] & L \arrow[r] & 1\;,
\end{tikzcd}
\end{equation*}
there is an exact sequence
\begin{equation*}
\centering
\begin{tikzcd}
\widehat K \arrow[r, "\widehat \varphi"] & \widehat G \arrow[r, "\widehat \psi"] & \widehat L \arrow[r] & 1.
\end{tikzcd}
\end{equation*}
\end{proposition}
\begin{proposition}[{\cite[Lemma 3.2.6]{RZ10}}]\label{PROP: Left exact}
If $f:H\to G$ is injective, then $\widehat{f}$ is injective if and only if \textit{the profinite topology on $G$ induces the full profinite topology of $H$} by $f$, i.e. for any finite-index subgroup $M\le H$, there exists a finite-index subgroup $N\le G$ such that $f^{-1}(N)\subseteq M$.
\end{proposition}


For any subgroup $H\le G$, we denote $\overline{H}$ as the closure of $\iota(H)$ in $\widehat{G}$. Note that $H$ does not always inject into $\overline{H}$, for instance, when $G$ is not residually finite. In addition, the surjective homomorphism $\widehat{H}\to \overline{H}$ is an isomorphism if and only if the condition in \autoref{PROP: Left exact} holds.

\begin{proposition}[{\cite[Proposition 3.2.2]{RZ10}}]\label{THM: correspondence of subgroup}
Let $G$ be an abstract group. Then there is a one-to-one correspondence. 
\begin{equation*}
\begin{tikzcd}[row sep=0.2em]
\left\{\text{Finite-index subgroups of }G\right\} \arrow[leftrightarrow,"1:1"]{r} & \left\{\text{Open subgroups of }\widehat{G}\right\} \\
H  \arrow[r, maps to] &\overline{H}                            \cong \widehat{H}    \\
\iota^{-1}(U) & U \arrow[l, mapsto]
\end{tikzcd}
\end{equation*}
%
%
%
In addition, this correspondence sends normal subgroups to normal subgroups.
\end{proposition}

\subsection{Profinite isomorphism and finite quotients}

The profinite completion encodes full data of finite quotients as stated in \autoref{Finite quotient}. In this subsection, we introduce a more general version.

\begin{definition}\label{DEF: Finite quotient marked with subgroups}
Let $G$ be a group and $H_1,\cdots, H_m$ be a collection of subgroups. Denote the set of 
finite quotients of $G$ marked with conjugacy classes of subgroups as 
\small
$$\mathcal{C}(G;H_1,\cdots, H_m):= \left\{\left.\left[\quo{G}{U};\,\quo{H_1}{(H_1\cap U)},\cdots,\quo{H_m}{(H_m\cap U)}\right]\,\right|\, U\lhd_f G\right\},$$
\normalsize
where $[A;B_1,\cdots, B_m]$ represents the isomorphism class of group marked with $m$ conjugacy classes of subgroups. To be precise,   
$[A;B_1,\cdots,B_m]=[A';B_1',\cdots,B_m']$ if and only if there exists an isomorphism $f:A\ttt A'$, such that $f(B_i)$ is a conjugate of $B_i'$ for each $1\le i\le m$.
\end{definition}
\begin{theorem}\label{THM: finite quotients with marked subgroups}
Let $G$ and $G'$ be finitely generated groups. Suppose $H_1,\cdots, H_m$ are subgroups of $G$, and $H_1',\cdots, H_m'$ are subgroups of $G'$.
Then the following are equivalent.
\begin{enumerate}[leftmargin=*]
\item $\mathcal{C} (G;H_1,\cdots, H_m)=\mathcal{C} (G';H_1',\cdots, H_m')$.
\item There exists an isomorphism as profinite groups $f:\widehat{G}\to \widehat{G'}$, such that $f(\overline{H_i})$ is a conjugate of $\overline{H_i'}$ in $\widehat{G'}$ for each $1\le i \le m$.
\item There exists an isomorphism as abstract groups $f:\widehat{G}\to \widehat{G'}$, such that $f(\overline{H_i})$ is a conjugate of $\overline{H_i'}$ in $\widehat{G'}$ for each $1\le i \le m$.
\end{enumerate}
\end{theorem}

\begin{remark}
When $m=0$, $\mathcal{C}(G;)$ is simply $\mathcal{C}(G)$ defined as usual. Therefore, \autoref{Finite quotient} is the special case of \autoref{THM: finite quotients with marked subgroups}.
\end{remark}

The proof of this theorem is a standard generalization of \cite{DFPR82} (see also \cite[Theorem 3.2.7]{RZ10}), together with the theorem of Nikolov-Segal \cite{NS03}. As the proof is irrelevant with our mainline, we include the proof in \autoref{APPproof} for completeness.

\subsection{Profinite completion and abelianization}

\begin{definition}
\begin{enumerate}[leftmargin=*]
\item Let $H$ be an abstract group, the {\em abelianization} of $H$ is denoted as $\ab{H}=H/[H,H]$, where $[H,H]$ is the commutator subgroup.
\item Let $G$ be a profinite group, the {\em topological abelianization} of $G$ is denoted as $\Ab{G}=G/\overline{[G,G]}$, where $\overline{[G,G]}$ is the closure of the commutator subgroup.
\end{enumerate}
\end{definition}

\begin{lemma}\label{LEM: Abelianization}
Let $G$ be a group, there is a natural isomorphism $\Ab{\widehat{G}}\cong \widehat{\ab{G}}$.
\end{lemma}
\begin{proof}
According to \autoref{PROP: Completion Exact}, the short exact sequence $1\to [G,G]\to G\to \ab{G}\to 1$ induces an exact sequence of profinite groups $\widehat{[G,G]}\to \widehat{G}\to \widehat{\ab{G}}\to 1$. Therefore, there is a natural isomorphism $\widehat{G}/\overline{[G,G]}\cong\widehat{\ab{G}}$. It is obvious that $\overline{[G,G]}\subseteq \overline{[\widehat{G},\widehat{G}]}$. On the other hand, since the image of $G$ is dense in $\widehat{G}$,  we have $\overline{[\widehat{G},\widehat{G}]}\subseteq \overline{[G,G]}$. Thus, $\overline{[G,G]}= \overline{[\widehat{G},\widehat{G}]}$, and $\widehat{G}/\overline{[G,G]}=\widehat{G}/\overline{[\widehat{G},\widehat{G}]}=\Ab{\widehat{G}}$.
\end{proof}

For any abelian group $G$, let $\Tor(G)$ denote its torsion subgroup. 
Let $\widehat{\Z}$ denote the profinite completion of the infinite cyclic group. Then, $\widehat{\Z}$ can be identified with the direct product of the  $p$-adic integer rings, i.e. $\widehat{\Z}=\prod_{p:\,\text{prime}}\Z_{p}$.

\begin{lemma}\label{LEM: Completion of abelian group}
Let $G$ be a finitely generated abelian group. Then $\widehat{G}$ is also abelian, and
\begin{enumerate}[leftmargin=*]
\item the inclusion $G\hookrightarrow \widehat{G}$ induces an isomorphism $\Tor(G)\cong \Tor(\widehat{G})$;
\item there is a natural isomorphism $\tensor G\ttt \widehat{G}$;
\item there is a natural isomorphism $\widehat{G}/\Tor(\widehat{G})\ttt \widehat{G/\Tor(G)} $.
\end{enumerate}
\end{lemma}

\newsavebox{\abcdefq}
\begin{lrbox}{\abcdefq}
$
\widehat{H}=\limi\limits_{N_i\lhd_f H_i} \quo{H_1\times H_2}{N_1\times N_2}=\limi\limits_{N_i\lhd_f H_i}\quo{H_1}{N_1}\times \quo{H_2}{N_2}=\widehat{H_1}\times \widehat{H_2}.
$
\end{lrbox}

\begin{proof}
For a direct product of groups $H=H_1\times H_2$, it is easy to verify that $\{N_1\times N_2\mid N_1\lhd_f H_1,\, N_2\lhd_f H_2\}$ is a cofinal system of the finite index normal subgroups of $H$. Therefore, 
\begin{equation}\label{EQU: product group}
\scalebox{0.95}{\usebox{\abcdefq}}
\end{equation}

We decompose the finitely generated abelian group $G$ as $G=\mathbb{Z}^m\times \Tor(G)$, where $\Tor(G)$ is a finite group. The above conclusion shows that $\widehat{G}=\widehat{\Z}^m\times \Tor(G)$. 
Since $\widehat{\Z}$ is torsion-free, it follows that $\Tor(\widehat{G})=\Tor(G)$, finishing the proof of (1).

In addition, decomposing $\widehat{\Z}$ as $\prod \Z_p$, we have an isomorphism $\tensor \Z/p^n\Z\ttt \Z/p^n\Z$ for any prime power $p^n$. Any finite abelian group $T$ can be decomposed as a direct sum of subgroups in the form of $\Z/p^n\Z$, so there is a natural isomorphism $\tensor T\ttt T\ttt \widehat{T}$. Thus, the decomposition $G=\mathbb{Z}^m\times  \Tor(G)$ implies (2).

(3) follows from the short exact sequence
\begin{equation*}
\begin{tikzcd}
1 \arrow[r] & \widehat{\Tor(G)} \arrow[r] & \widehat{G} \arrow[r] & \widehat{G/\Tor(G)} \arrow[r] & 1
\end{tikzcd}
\end{equation*}
deduced from \autoref{PROP: Completion Exact} and \autoref{PROP: Left exact}.
\end{proof}

\def\free#1{#1_{\mathrm{free}}}
For an abelian group $G$, denote $\free{G}=G/\Tor(G)$. Then, 
\autoref{LEM: Abelianization} and \autoref{LEM: Completion of abelian group} together imply the following proposition.
\begin{proposition}\label{PROP: Induce linear}
Let $G,H$ be finitely generated groups, and $f:\widehat{G}\ttt\widehat{H}$ be an isomorphism of profinite groups. Then $f$ induces an isomorphism 
\begin{equation*}
\begin{tikzcd}
f_{\ast}:\;\tensor\free{\ab G} \arrow[r,"\cong"] &\tensor\free{\ab H}
\end{tikzcd}
\end{equation*}
so that the following diagram commutes.
{
\begin{equation*}
\begin{tikzcd}[column sep=large]
                                                    & \widehat{G} \arrow[d, two heads] \arrow[r, "\cong"',"f"]                               & \widehat{H} \arrow[d, two heads]                              &                                                       \\
G \arrow[ru] \arrow[d, two heads]                   & \Ab{\widehat{G}} \arrow[r, "\cong"] \arrow[d, "\cong"']                            & \Ab{\widehat{H}} \arrow[d, "\cong"]                           & H \arrow[lu] \arrow[d, two heads]                     \\
\ab{G} \arrow[r, hook] \arrow[dd, two heads]        & \widehat{\ab{G}} \arrow[d, two heads] \arrow[r, "\cong"]                           & \widehat{\ab{H}} \arrow[d, two heads]                         & \ab{H} \arrow[l, hook'] \arrow[dd, two heads]         \\
                                                    & \free{(\widehat{\ab G})} \arrow[d, "\cong"'] \arrow[r, "\cong"]      & \free{(\widehat{\ab H})} \arrow[d, "\cong"]     &                                                       \\
\free{\ab G} \arrow[rd, hook] \arrow[r, hook] & \widehat{\free{\ab G}} \arrow[r, "\cong"]                                    & \widehat{\free{\ab H}}                                & \free{\ab H} \arrow[ld, hook'] \arrow[l, hook'] \\
                                                    &\tensor\free{\ab G}\arrow[r, "\cong"',"f_\ast"] \arrow[u, "\cong"] & \tensor\free{\ab H} \arrow[u, "\cong"'] &                                                      
\end{tikzcd}
\end{equation*}
}
\normalsize
\end{proposition}

\section{Bass-Serre theory in the profinite setting}\label{SEC: Bass-Serre}

\subsection{Graph of groups}\label{SEC: Graph of Groups}

In this paper, we will only focus on finite graph of groups, and we will omit the notation ``finite'' if the finiteness is clear. 

A {\em finite oriented graph} $\Gamma$ consists of a finite set $\Gamma=V(\Gamma)\sqcup E(\Gamma)$, where $V(\Gamma)$ and $E(\Gamma)$ denote the set of vertices and edges respectively, and two {\em relation maps} $d_0,d_1:E(\Gamma)\to V(\Gamma)$, which assigns the two endpoints of each edge.

In the classical setting, a {\em finite graph of groups}, denoted by $(\mathcal{G},\Gamma)$, consists of a finite connected oriented graph $\Gamma$, a group $\mathcal{G}_x$ associated to each $x\in V(\Gamma)\cup E(\Gamma)$, and injective homomorphisms $\varphi_i^e: \mathcal{G}_e\to \mathcal{G}_{d_i(e)}$ for each $e\in E(\Gamma)$ and $i=0,1$.

\begin{definition}\label{DEF: Fundamental group of graph of group}
Let $(\mathcal{G},\Gamma)$ be a finite graph of groups and let $S$ be a  maximal subtree of $\Gamma$, the fundamental group with respect to $S$ is defined as
\begin{align*}
\pi_1(\mathcal{G},\Gamma,S)=\frac{\left(\ast_{v\in V(\Gamma)}\mathcal{G}_v\right )\ast F_{E(\Gamma)}}{\left\langle\!\left\langle t_e:e\in E(S),\,\varphi_1^e(g)^{-1}t_e\varphi_0^e(g)t_e^{-1}:e\in E(\Gamma)\text{ and }g\in \mathcal{G}_e\right\rangle\!\right\rangle}\:,
\end{align*}
where $F_{E({\Gamma})}$ is the free group over $E({\Gamma})$ with generators $t_e$, $e\in E(\Gamma)$.
\end{definition}


It follows from \cite[Proposition 20]{Ser80} that up to isomorphism, $\pi_1(\mathcal{G},\Gamma,S)$ is independent with the choice of the maximal subtree $S$. Thus, we may omit the notation $S$ and simply denote the fundamental group as $\pi_1(\mathcal{G},\Gamma)$ when there is no need for a precise maximal subtree.

A geometric interpretation of graph of groups turns out to be a graph of topological spaces. Indeed, let $X_v$ (resp. $X_e$) be CW-complexes so that $\pi_1X_v=\mathcal G_v$ (resp. $\pi_1X_e=\mathcal G_e$), and let $\phi_i^e: X_e\to X_{d_i(e)}$ be the maps (up to homotopy) so that $\phi_i^e$ induces $\varphi_i^e: \pi_1X_e=\mathcal G_e\to \pi_1X_{d_i(e)}=\mathcal G_{d_i(e)}$. 
We can construct a CW-complex according to this information.
\begin{align*}
X_{\Gamma}=\left(\bigsqcup_{e\in E(\Gamma)}X_e\times [0,1]\right)\bigcup_{\phi_i^e:\,X_e\times\{i\}\to X_{d_i(e)}}\left(\bigsqcup_{v\in V(\Gamma)}X_v\right )
\end{align*}

It turns out that $\pi_1(X_\Gamma)\cong \pi_1(\mathcal{G},\Gamma)$, see \cite{SW79}. 
In view of this method, the fundamental group of an irreducible 3-manifold splits as a graph of group through JSJ-decomposition.

The theory of graph of groups generalizes to  profinite groups.
\begin{definition}
A \textit{finite graph of profinite groups}, denoted by $(\mathcal{P},\Gamma)$, consists of a finite, connected, oriented graph $\Gamma$, a profinite group $\mathcal{P}_x$ for each $x\in V(\Gamma)\cup E(\Gamma)$, and a continuous injective homomorphism $\varphi_i^e: \mathcal{P}_e\to \mathcal{P}_{d_i(e)}$ for each $e\in E(\Gamma)$ and $i=0,1$.
\end{definition}

Analogous to \autoref{DEF: Fundamental group of graph of group}, we can define a profinite version of the fundamental group of a finite graph of profinite groups following 
\cite{Rib17}. 
Let $G\amalg H$ denote the free profinite product of two profinite groups $G$ and $H$.
\begin{definition}\label{DEF: Profinite fundamental group of graph of profinite group}
Let $(\mathcal{P},\Gamma)$ be a finite graph of profinite groups and let $S$ be a  maximal subtree of $\Gamma$. The \textit{profinite fundamental group} with respect to $S$ is defined as
\begin{align*}
\Pi_1(\mathcal{P},\Gamma,S)=\left.\left( (\amalg_{v\in V(\Gamma)}\mathcal{P}_v )\amalg \mathcal{F}_{E(\Gamma)}\right)\right /N,
\end{align*}
where $\mathcal{F}_{E({\Gamma})}$ is the free profinite group over the finite space $E({\Gamma})$ with generators denoted by $t_e$ ($e\in E(\Gamma)$), and $$N={\overline{\left\langle\!\left\langle t_e:e\in E(S),\,\varphi_1^e(g)^{-1}t_e\varphi_0^e(g)t_e^{-1}:e\in E(\Gamma)\text{ and }g\in \mathcal{P}_e\right\rangle\!\right\rangle}}$$ is the minimal closed normal subgroup containing the 
generators.
\end{definition}

$\Pi_1(\mathcal{P},\Gamma,S)$ is a profinite group, and is irrelevant, up to isomorphism, with the choice of the maximal subtree $S$ as shown by \cite[Theorem 6.2.4]{Rib17}. Similarly, we denote the profinite fundamental group by $\Pi_1(\mathcal{P},\Gamma)$ if the clarification of the maximal subtree is not essential.

\begin{definition}
A finite graph of profinite groups $(\mathcal{P},\Gamma)$ is called \textit{injective} if for any vertex $v\in \Gamma$, the homomorphism $\mathcal{P}_v\to (\amalg_{v\in V(\Gamma)}\mathcal{P}_v )\amalg \mathcal{F}_{E(\Gamma)}\to \Pi_1(\mathcal{P},\Gamma)$ is injective.
\end{definition}

Note that there exist non-injective examples in the profinite settings as constructed by \cite{Rib73}, though injectivity always holds 
in the classical settings, see 
\cite[Proposition 6.2.1]{Geo07}.

One may relate a graph of abstract groups with a graph of profinite groups as shown by the following construction. 
Let $(\mathcal{G},\Gamma)$ be a finite graph of groups, and suppose that the profinite topology of $\mathcal{G}_{d_i(e)}$ induces the full profinite topology on $\mathcal{G}_{e}$ through $\varphi_i^e$ for any $e\in E(\Gamma)$ and $i=0,1$. Then $(\widehat{\mathcal{G}},\Gamma)$ is a finite graph of profinite groups, where each vertex group or edge group ($\widehat{\mathcal G_v}$ or $\widehat{\mathcal G_e}$) is the profinite completion of the corresponding vertex or edge group. The continuous homomorphism $\widehat{\varphi_i^e}:\widehat{\mathcal G_e}\to \widehat{\mathcal G_{d_i(e)}}$ is indeed injective by \autoref{PROP: Left exact}. 

From \autoref{DEF: Fundamental group of graph of group} and \autoref{DEF: Profinite fundamental group of graph of profinite group}, one establishes a group homomorphism $\psi: \pi_1(\mathcal{G},\Gamma)\to \Pi_1(\widehat{\mathcal{G}},\Gamma)$ through the canonical homomorphisms $\mathcal{G}_v\to \widehat{\mathcal{G}_v}$ and $\mathcal{G}_e\to \widehat{\mathcal{G}_e}$. 
Clearly, the image of $\psi$ is dense in $\Pi_1(\widehat{\mathcal{G}},\Gamma)$. It follows that $\Pi_1(\widehat{\mathcal{G}},\Gamma)$ is indeed the profinite completion of $\pi_1(\mathcal{G},\Gamma)$ under suitable conditions.

\begin{definition}
A graph of groups $(\mathcal{G},\Gamma)$ is called \textit{efficient} if the following three properties hold.
\begin{enumerate}[leftmargin=*]
\item $\pi_1(\mathcal{G},\Gamma)$ is residually finite.
\item The profinite topology of $\pi_1(\mathcal{G},\Gamma)$ induces the full profinite topology on each vertex subgroup $\mathcal G_v$ and each edge subgroup $\mathcal G_e$.
\item All vertex subgroups and edge subgroups are separable in $\pi_1(\mathcal{G},\Gamma)$.
\end{enumerate}
\end{definition}

\begin{proposition}[{\cite[Theorem 6.5.3]{Rib17}} ]\label{THM: Efficient completion}
Let $(\mathcal{G},\Gamma)$ be a finite graph of group which is efficient.
\begin{enumerate}[leftmargin=*]
\item $\widehat{\pi_1(\mathcal{G},\Gamma)}$ and $ \Pi_1(\widehat{\mathcal{G}},\Gamma)$  are isomorphic as profinite groups. Indeed, there is an isomorphism such that the following diagram commutes.
\begin{equation*}
\centering
\begin{tikzcd}
{\pi_1(\mathcal{G},\Gamma)} \arrow[r, "\psi"] \arrow[d, hook] & {\Pi_1(\widehat{\mathcal{G}},\Gamma)} \\
{\widehat{\pi_1(\mathcal{G},\Gamma)}} \arrow[ru, "\cong"']    &                                      
\end{tikzcd}
\end{equation*}
\item\label{Eff 2} 
$(\widehat{\mathcal{G}},\Gamma)$ is injective.
\end{enumerate}
\end{proposition}

In view of (\ref{Eff 2}), we usually identify $\widehat{\mathcal{G}_v}$ with its image in $\Pi_1(\widehat{\mathcal{G}},\Gamma)$.

\subsection{Profinite Bass-Serre tree}

The classical construction of a Bass-Serre tree turns out to be a profinite tree in the profinite settings. 
We briefly introduce the definitions in this subsection, and refer the readers to \cite{Rib17} for more details.

\begin{definition}\label{Def: prof tree}
\begin{enumerate}[leftmargin=*]
\item
A \textit{profinite graph} consists of a profinite space $\Gamma$, a closed subspace $V(\Gamma)\subset \Gamma$, and two continuous maps $d_0,d_1:\Gamma\to V(\Gamma)$ such that $d_i|_{V(\Gamma)}=id$. 
We denote $E(\Gamma):=\Gamma \setminus V(\Gamma)$. An element in $V(\Gamma)$ is called a \textit{vertex} and an element in $E(\Gamma)$ is called an \textit{edge}. 

By an \textit{abstract graph} we mean a combinatorial graph  not assigning a specific topology. The notation distinguishes from a profinite graph which is equipped with a compact, Hausdorff, totally disconnected topology.
\item
Let $\Gamma$ and $\Gamma'$ be two profinite graphs. An \textit{isomorphism} $\Phi:\Gamma\to \Gamma'$ is a homeomorphism which is meanwhile an isomorphism of abstract oriented graphs, i.e. $\Phi(d_i(x))=d_i'(\Phi(x))$ for any $x\in \Gamma$ and $i=0,1$.

\item\label{333333}
A profinite graph $\Gamma$ is a \textit{profinite tree} if the following chain complex is exact. \begin{equation*}
\centering
\begin{tikzcd}
0 \arrow[r] & {\widehat{\Z}\[{(E(\Gamma),\ast)}} \arrow[r,"\partial"] & \widehat{\Z}\[{V(\Gamma)} \arrow[r,"\epsilon"] & \widehat{\Z} \arrow[r] & 0 
\end{tikzcd}
\end{equation*} 
\item
We say that \textit{a profinite group $G$ acts on a profinite graph $\Gamma$} if there is a combinatorial group action $G\curvearrowright \Gamma$ so that the map $G\times \Gamma\to \Gamma$, $(g,x)\mapsto g\cdot x$ is continuous.
\item 
A \textit{path} of length $n$ in a profinite graph $\Gamma$ consists of $n+1$ vertices $v_0,\cdots, v_n\in V(\Gamma)$, and $n$ edges $e_1,\cdots, e_n\in E(\Gamma)$, such that $e_i$ adjoins $v_{i-1}$ and $v_i$. The path is called \textit{injective} if $v_0,\cdots, v_n$ are $n+1$ distinct vertices.

\item 
The action by a profinite group $G$ on a profinite tree $T$ is called \textit{$n$-acylindrical} for some $n\in \mathbb{N}$ if the stabilizer of any injective path of length greater than $n$ is trivial.

\end{enumerate}
\end{definition}

\begin{definition}\label{def: BStree}
Let $(\mathcal{P},\Gamma)$ be a finite graph of profinite groups, and denote   $\Pi=\Pi_1(\mathcal{P},\Gamma)$. For each $v\in V(\Gamma)$, let $\Pi_v$ denote the image of $\mathcal{P}_v$ in $\Pi$ through the map $\mathcal{P}_v\to (\amalg_{v\in V(\Gamma)}\mathcal{P}_v )\amalg \mathcal{F}_{E(\Gamma)}\to \Pi$, and for each $e\in E(\Gamma)$, let $\Pi_e$ denote the image of $\mathcal{P}_e$ in $\Pi$ through the map $\varphi_1^e:\mathcal{P}_e\to \mathcal{P}_{d_1(e)}$. Then, for each $x\in V(\Gamma)\cup E(\Gamma)$, $\Pi_x$ is a closed subgroup in $\Pi$, so the coset space $\Pi/\Pi_v$ is a profinite space.
Let $$V(T)=\bigsqcup_{v\in V(\Gamma)}\Pi/\Pi_v,\quad E(T)=\bigsqcup_{e\in E(\Gamma)}\Pi/\Pi_e,\quad \text{and }T=V(T)\bigsqcup E(T)$$ be equipped with disjoint-union topology. The relation map is given by $$d_0(g\Pi_e)=gt_e\Pi_{d_0(e)},\quad d_1(g\Pi_e)=g\Pi_{d_1(e)},\quad \forall e\in E(\Gamma), g\in \Pi.$$ Then $T$ is called the \textit{profinite Bass-Serre tree} associated to $(\mathcal{P},\Gamma)$.
\end{definition}

Indeed, $T$ is a disjoint union of finitely many profinite spaces, and hence $T$ is a profinite space. 
It is easy to verify that the relation maps $d_0,d_1$ are continuous.
Moreover, \cite[Corollary 6.3.6]{Rib17}  shows that $T$ is indeed a profinite tree. In fact, omitting the topology and the orientation, $T$ is combinatorially a disjoint union of abstract trees, see \cite[Exercise 2.4.4]{Rib17}. 

We conclude the profinite version of Bass-Serre theory as the following proposition.

\begin{proposition}\label{PROP: Bass-Serre}
Let $(\mathcal{P},\Gamma)$ be a finite graph of profinite groups, and let $T$ be the associated profinite Bass-Serre tree. Then, there is a natural action by $\Pi_1(\mathcal{P},\Gamma)$ on $T$ through left multiplication, which is obviously continuous. In addition, $T/\Pi_1(\mathcal{P},\Gamma)=\Gamma$, and the stabilizer of each vertex or edge in $T$ are conjugates of the image of $\mathcal{P}_v$ or $\mathcal{P}_e$ in $\Pi_1(\mathcal{P},\Gamma)$. 
\end{proposition}

\subsection{Congruent isomorphism}

In \autoref{indef: congruent} we have defined the notion of a congruent isomorphism between graph of groups.  Let us recall its definition. 

A \textit{congruent isomorphism} $f_{\bullet}$ between two finite graphs of abstract/profinite groups $(\mathcal{G},\Gamma)$ and $(\mathcal{H},\Gamma')$ consists of the following ingredients.
\begin{enumerate}[leftmargin=*]
\item There is an isomorphism of oriented graphs $\F: \Gamma\stackrel{\cong}{\longrightarrow} \Gamma'$;
\item For each $x\in V(\Gamma)\cup E(\Gamma)$, there is an isomorphism of abstract/profinite groups $f_x: \mathcal{G}_x\stackrel{\cong}{\longrightarrow} \mathcal{H}_{\F(x)}$;
\item 
For each $e\in E(\Gamma)$ and $i=0,1$, there exists an element $h_{i}^e\in \mathcal{H}_{\F(d_i(e))}$ such that $$h_i^e\cdot (\varphi')_i^{\F(e)} \left(f_e(g)\right )\cdot (h_i^e)^{-1}=f_{d_i(e)}\left (\varphi_i^e(g) \right ),\quad \forall g\in \mathcal{G}_e.$$ In other words, the following diagram commutes.
\begin{equation}\label{Diagram Isom Conj}
\centering
\begin{tikzcd}[column sep=1.3cm]
\mathcal{G}_e \arrow[dd, "\cong","f_e"'] \arrow[rr, "\varphi_i^e"] &                                                  & \mathcal{G}_{d_i(e)} \arrow[dd, "\cong"',"f_{d_i(e)}"] \\
 & &\\
\mathcal{H}_{\F(e)} \arrow[r, "(\varphi')_i^{\F(e)} "]     & \mathcal{H}_{\F(d_i(e))} \arrow[r, "C_{h_i^e}", "\cong"'] & \mathcal{H}_{\F(d_i(e))}               
\end{tikzcd}
\end{equation}
\end{enumerate}

The following two results are direct generalizations from the classical Bass-Serre theory into the profinite version.

First, being congruently isomorphic is a sufficient condition for having isomorphic profinite fundamental groups. 

\begin{proposition}\label{LEM: congruent isom induce isom}
Suppose $f_{\bullet}$ is a congruent isomorphism between two finite graphs of profinite groups $(\mathcal{G},\Gamma)$ and $(\mathcal{H},\Gamma')$. Then $f_{\bullet}$ induces an isomorphism $f: \Pi_1(\mathcal{G},\Gamma)\ttt \Pi_1(\mathcal{H},\Gamma')$, such that for each $v\in V(\Gamma)$, there exists $x_v\in \Pi_1(\mathcal{H},\Gamma')$ so that the following diagram commutes.
\begin{equation*}
\begin{tikzcd}[row sep=0.6cm]
\mathcal{G}_v \arrow[dd] \arrow[rr, "f_v","\cong"'] &                                                   & \mathcal{H}_{\F(v)} \arrow[dd] \\
                                           &                                                   &                                \\
{\Pi_1(\mathcal{G},\Gamma)} \arrow[r, "f","\cong"'] & {\Pi_1(\mathcal{H},\Gamma')} \arrow[r, "C_{x_v}"] & {\Pi_1(\mathcal{H},\Gamma')}  
\end{tikzcd}
\end{equation*}
As a corollary, for any connected subgraph $\Delta\subseteq \Gamma$, denote $\Delta'=\F(\Delta)$. Then $f_{\bullet}$ induces an isomorphism $f_{\Delta}: \Pi_1(\mathcal{G},\Delta)\ttt \Pi_1(\mathcal{H},\Delta')$. 
\end{proposition}
\begin{proof}
This follows from the same proof of \cite[Chapter 8, Lemma 21]{Coh89} in the classical Bass-Serre theory. One can also prove this proposition directly from \autoref{DEF: Profinite fundamental group of graph of profinite group}, by induction on $\#E(\Gamma)$ starting from the simplest cases of HNN-extensions and amalgamations. 
\end{proof}

Conversely, a profinite group may split as a graph of profinite groups in many ways. 
The following lemma provides a method based on Bass-Serre theory to show when two splittings are essentially the same, i.e. when two graphs of profinite groups are congruently isomorphic.

\begin{lemma}\label{LEM: Equivariant isomorphism of BS tree induce congruent isomorphism}
Let $(\mathcal{G},\Gamma)$ and $(\mathcal{H},\Gamma')$ be two injective finite graphs of profinite groups, and let $T$ and $T'$ denote their profinite Bass-Serre trees. Suppose there exist an isomorphism of profinite fundamental groups $f:\Pi_1(\mathcal{G},\Gamma)\ttt \Pi_1(\mathcal{H},\Gamma')$ and a $f$-equivariant isomorphism of oriented abstract graphs $\Phi: T\ttt T'$ on which the profinite fundamental groups act. Then $f$ induces a congruent isomorphism $f_\bullet$ between  $(\mathcal{G},\Gamma)$ and $(\mathcal{H},\Gamma')$.
\end{lemma}

\newsavebox{\eqquyi}
\begin{lrbox}{\eqquyi}
$
\begin{tikzcd}[column sep=2em]
f_x: \mathcal{G}_x=\Stab(\widetilde{x}) \arrow[r, "f","\cong"'] & \Stab(\Phi(\widetilde{x}))=\Stab(h_x{\tiny \cdot} \widetilde{\F(x)}) \arrow[r, "C_{h_x}^{-1}","\cong"'] & \Stab(\widetilde{\F(x)})= \mathcal{H}_{\F(x)}.
\end{tikzcd}
$
\end{lrbox}

\newsavebox{\eqquer}
\begin{lrbox}{\eqquer}
$
\begin{tikzcd}
\varphi^i_e:\;\mathcal{G}_e=\Stab(\widetilde{e}) \arrow[r, hook] & \Stab(d_i(\widetilde{e}))=\Stab(g\cdot \widetilde{v}) \arrow[r, "C^{-1}_{g}","\cong"'] & \Stab(\widetilde{v})=\mathcal{G}_{v}.
\end{tikzcd}
$
\end{lrbox}

\newsavebox{\eqqusan}
\begin{lrbox}{\eqqusan}
$
\begin{tikzcd}[row sep=tiny, column sep=2em]
(\varphi')^i_{\F(e)}:\mathcal{H}_{\F(e)}=\Stab(\widetilde{\F(e)}) 
\arrow[r, hook] &
\Stab(d_i'(\widetilde{\F(e)})) \arrow[d, equal] & \\
 & \Stab(h\cdot \widetilde{\F(v)}) \arrow[r, "C^{-1}_{h}","\cong"'] & \Stab(\widetilde{\F(v)})=\mathcal{H}_{\F(v)}.
\end{tikzcd}
$
\end{lrbox}
\begin{proof}
%


We  denote the quotient maps as $$p: T\to T/\Pi_1(\mathcal G,\Gamma)=\Gamma\quad \text{and} \quad q: T'\to T'/\Pi_1(\mathcal H,\Gamma')=\Gamma'.$$ Then, the $f$-equivariant isomorphism $\Phi$ induces a combinatorial isomorphism between the oriented quotient graphs, which we denote as  $\F: \Gamma=T/\Pi_1(\mathcal G,\Gamma)\ttt T'/\Pi_1(\mathcal H,\Gamma')=\Gamma'$. 
%

Since $(\mathcal{G},\Gamma)$ is injective, we can identify $\mathcal{G}_v$ with its image in $\Pi_1(\mathcal{G},\Gamma)$ for each $v\in V(\Gamma)$, and identify $\mathcal{G}_e$ with its image in $ \Pi_1(\mathcal{G},\Gamma)$ through the map $\varphi_1^e:\mathcal{G}_e\to \mathcal{G}_{d_1(e)}$. The similar notation follows for $(\mathcal{H},\Gamma')$.
Then, 
for each $x\in \Gamma$, $\mathcal{G}_x\le \Pi_1(\mathcal{G},\Gamma)$ is the stabilizer of a pre-image $\widetilde{x}\in p^{-1}(x)\subseteq T$ that corresponds to the coset $\mathcal{G}_x$ in $\Pi_1(\mathcal{G},\Gamma)/\mathcal{G}_x$. Likewise, for each $y\in \Gamma'$, $\mathcal{H}_y\le \Pi_1(\mathcal{H},\Gamma')$ is the stabilizer of a pre-image $\widetilde{y}\in q^{-1}(y)\subseteq T'$ that corresponds to the coset $\mathcal{H}_y$ in $\Pi_1(\mathcal{H},\Gamma')/\mathcal{H}_y$. 

For each $x\in \Gamma$, $\Phi(p^{-1}(x))=q^{-1}(\F(x))$, so $\Phi(\widetilde{x})$ and $\widetilde{\F(x)}$ belong to the same orbit of $\Pi_1(\mathcal{H},\Gamma')$. To be explicit, there exists $h_x\in \Pi_1(\mathcal{H},\Gamma')$ such that $\Phi(\widetilde{x})= h_x\cdot\widetilde{\F(x)}$. As a result, $f$ induces an isomorphism of profinite groups 
\begin{equation*}
\scalebox{0.92}{\usebox{\eqquyi}}
\end{equation*}

Given $e\in E(\Gamma)$ and $i=0,1$, denote $v=d_i(e)$ and the relation map $\varphi_i^e$ can be determined as follows. $d_i(\widetilde{e})\in p^{-1}(v)$ belongs to the same orbit of $\widetilde{v}$, and there exists a specific element $g\in \Pi_1(\mathcal{G},\Gamma)$ such that $d_i(\widetilde{e})=g\cdot \widetilde{d_i(e)}=g\cdot \widetilde{v}$ and
\begin{equation*}
\scalebox{0.92}{\usebox{\eqquer}}
\end{equation*}

Similarly, there exists $h\in \Pi_1(\mathcal{H},\Gamma')$ such that $d_i'(\widetilde{\F(e)})=h\cdot \widetilde{d_i'(\F(e))}=h\cdot \widetilde{\F(v)}$, and the relation map $(\varphi')^i_{\F(e)}$ in $(\mathcal{H},\Gamma')$ is decomposed as
\begin{equation*}
\scalebox{0.92}{\usebox{\eqqusan}}
\end{equation*}

Thus,  the following diagram commutes.
\begin{equation*}
\begin{tikzcd}
\mathcal{G}_e \arrow[dd, "C_{h_e}^{-1}\circ f","f_e"'] \arrow[rrr, "\varphi_e^i","C_{g^{-1}}"'] &  &  & \mathcal{G}_v \arrow[d, "C_{h_v}^{-1}\circ f"',"f_v"] \\
                                                                           &  &  & \mathcal{H}_{\F(v)} \arrow[d, "C_{h^{-1}h_e^{-1}f(g)h_v}"]     \\
\mathcal{H}_{\F(e)} \arrow[rrr, "\varphi_{\F(e)}^{\prime i}"',"C_{h^{-1}}"]             &  &  & \mathcal{H}_{\F(v)}                            
\end{tikzcd}
\end{equation*}

In fact, $h^{-1}h_e^{-1}f(g)h_v\in \Stab(\widetilde{\F(v)})=\mathcal{H}_{\F(v)}$ as shown by the following relations. 
{\small 
\begin{equation*}
    \begin{tikzcd}[column sep=2em]
\widetilde{\F(v)} \arrow[r, "h_v", maps to] & \Phi(\widetilde{v}) \arrow[r, "f(g)", maps to] & \Phi(g\cdot\widetilde{v})=\Phi(d_i(\widetilde{e}))=d_i(\Phi(\widetilde{e})) & d_i'(\widetilde{\F(e)}) \arrow[l, "h_e"', maps to] & \widetilde{\F(v)} \arrow[l, "h"', maps to]
\end{tikzcd}
\end{equation*}
}
This  implies the commutative diagram (\ref{Diagram Isom Conj}) by taking $${h_i^e=(h^{-1}h_e^{-1}f(g)h_v)^{-1}}.$$ Therefore, $f_\bullet$ is indeed a congruent isomorphism.
\end{proof}

\begin{remark}
\begin{enumerate}[leftmargin=*]
\item It is important  that the graphs of profinite groups are injective, so that we can identify the profinite groups $\mathcal{G}_v$ with the vertex stabilizers in $T$. The remaining proof is in fact the same as in classical Bass-Serre theory, where all graphs of (abstract) groups are automatically injective according to \cite[Proposition 6.2.1]{Geo07}.
\item The  isomorphism of abstract graphs $\Phi:T\ttt T'$ is indeed an isomorphism of profinite graphs. In fact, for each $x\in \Gamma$, we can choose a preimage $\tilde{x}\in T$. Then $\Phi(\tilde{x})$ is also a preimage of $\F(x)\in \Gamma'$. By construction of the profinite Bass-Serre trees, $$T=\bigsqcup_{x\in \Gamma }\quotient{\Pi_1(\mathcal{G},\Gamma)}{\Stab({\tilde{x}})}\,,\; T'=\bigsqcup_{x\in \Gamma}\quotient{\Pi_1(\mathcal{H},\Gamma')}{\Stab ({\Phi(\tilde{x})})}$$
are equipped with disjoint union topology of coset spaces. The $f$-equivariance of $\Phi$ implies that $f$ restricts to an isomorphism $\Stab ({\tilde{x}})\ttt{\Stab ({\Phi(\tilde{x})})}$, which is an isomorphism as profinite groups. Thus, $\Phi$ is a homeomorphism on each coset space, and is thus a homeomorphism on the disjoint union.  
\item 
Moreover, $f_\bullet$ again induces an isomorphism $f': \Pi_1(\mathcal{G},\Gamma)\ttt \Pi_1(\mathcal{H},\Gamma')$ by \autoref{LEM: congruent isom induce isom}. It follows that $f$ and $f'$ only differ by choices of sections of $T\to \Gamma$ and $T'\to \Gamma'$ in defining the profinite fundamental groups, resulting possible conjugations. This ambiguity can be eliminated as shown by \cite[Theorem 6.2.4]{Rib17}.
\end{enumerate}
\end{remark}

\section{Profinite detection of prime decomposition}\label{SEC: Prime decomposition}


Let $M$ be a compact orientable 3-manifold. Then $M$ can be decomposed as a connect sum
$M \cong M_1\#\cdots \# M_m\# (S^2\times S^1)^{\# r}$ via the {\em Kneser-Milnor prime decomposition}, 
where each $M_i$ is irreducible and $M_i\not\cong S^3$. In addition, the numbers $m$ and $r$ are unique, and 
the prime factors $M_i$ are uniquely determined up to homeomorphism and reordering.

In fact, when $M$ is  boundary-incompressible, the Kneser-Milnor prime decomposition is exactly correspondent to the prime decomposition of its fundamental group as a free product \cite{Hei72}. Thus, the numbers $m$ and $r$ as well as the fundamental groups $\pi_1(M_i)$ can be uniquely determined by $\pi_1(M)$, through a group theoretical method. 

The similar conclusion holds in the profinite settings. 

\begin{theorem}[{\cite[Theorem 6.22]{Wil19}}]\label{THM: Detect Kneser-Milnor}
Let $M$ and $N$ be two compact, orientable, boundary-incompressible 3-manifolds, with prime decomposition $M\cong M_1\#\cdots\# M_m\#(S^2\times S^1)^{\#r}$ and $N \cong N_1\#\cdots \# N_n\#(S^2\times S^1)^{\#s}$, where   $M_i$ and $N_j$ are irreducible. Suppose $ \widehat{\pi_1M}\cong \widehat{\pi_1N}$. Then $m=n$, $r=s$, and up to reordering, $\widehat{\pi_1M_i}\cong \widehat{\pi_1N_i}$.
\end{theorem}

In combination with the cohomological goodness of 3-manifold groups \cite{Cav12}, \autoref{THM: Detect Kneser-Milnor} implies the following corollary which might be familiar to experts.
\begin{corollary}\label{COR: Determine boundary}
Suppose $M$ and $N$ are two compact, orientable, boundary-incompressible 3-manifolds, and $\widehat{\pi_1M}\cong \widehat{\pi_1N}$. 
\begin{enumerate}[leftmargin=*]
\item\label{4.3-1} If $M$ is irreducible, then so is $N$.
\item\label{4.3-2} If $M$ is closed, then so is $N$.
\item\label{4.3-3} If $\partial M$ only consists of tori, then so does $\partial N$.
\end{enumerate}
\end{corollary}
\begin{proof}
Let $M\cong M_1\#\cdots\# M_m\#(S^2\times S^1)^{\#r}$ and $N \cong N_1\#\cdots \# N_n\#(S^2\times S^1)^{\#s}$ be the Kneser-Milnor prime decompositions of $M$ and $N$. According to \autoref{THM: Detect Kneser-Milnor}, $m=n$, $r=s$, and up to a reordering, $\widehat{\pi_1M_i}\cong \widehat{\pi_1N_i}$. 

This directly implies  (\ref{4.3-1}), which is equivalent to $m=1$ and $r=0$. For (\ref{4.3-2}) and (\ref{4.3-3}), it suffices to prove the conclusion for each pair of irreducible factors $M_i$ and $N_i$.

If $M_i$ has finite fundamental group, then $M_i$ is closed and $\widehat{\pi_1N_i}\cong \widehat{\pi_1M_i}\cong \pi_1M_i$ is also finite. Since $\pi_1N_i$ is residually finite \cite{Hem87}, it follows that $\pi_1N_i$ is also finite and $N_i$ is also closed.

If $M_i$ has infinite fundamental group, then $\widehat{\pi_1M_i}\cong\widehat{\pi_1N_i}$ is also infinite, and so is $\pi_1N_i$.  The standard application of sphere theorem implies that $M_i$ and $N_i$ are aspherical. \cite[Theorem 3.5.1 $\&$ Lemma 3.7.1]{Cav12} shows that $\pi_1M_i$ and $\pi_1N_i$ are cohomologically good in the sense of Serre \cite[pp. 15-16]{Ser01}. In particular, for any prime number $p$, there is an isomorphism  $H^{k}(\widehat{\pi_1M_i},\Fp)\cong H^{k}(\pi_1M_i,\Fp)$ induced by inclusion, and similarly, $H^{k}(\widehat{\pi_1N_i},\Fp)\cong H^{k}(\pi_1N_i,\Fp)$ for each $k\in \mathbb{N}$. Since $M_i$ and $N_i$ are aspherical, we obtain 
\begin{align*}
&H^{k}(M_i;\Fp)\cong H^{k}(\pi_1M_i,\Fp)\cong H^{k}(\widehat{\pi_1M_i},\Fp) & \\
\cong & H^{k}(\widehat{\pi_1N_i},\Fp)\cong H^{k}(\pi_1N_i,\Fp)\cong H^{k}(N_i;\Fp),&\forall k\in \mathbb{N}.
\end{align*}
Consequently, $M_i$ is closed if and only if $H^3(M_i;\Fp)\cong H^3(N_i;\Fp)\cong \Fp$ for all prime numbers $p$, which is equivalent to $N_i$ being closed. This proves (\ref{4.3-2}).

In addition, the Poincar\'e duality implies $\chi(\partial M_i)=2\chi (M_i)$. Thus, when $M_i$ and $N_i$ have infinite fundamental groups,  
\begin{align*}
\chi(\partial M_i)=2\chi (M_i)=2\chi (N_i)=\chi (\partial N_i).
\end{align*}
Note that $\partial M_i$ consists of tori if and only if $M_i$ is not closed and $\chi(\partial M_i)=0$, and likewise for $N_i$. This proves (\ref{4.3-3}).
\end{proof}

\section{Profinite detection of JSJ-decomposition}\label{SEC: JSJ}


Based on \autoref{THM: Detect Kneser-Milnor}, we may restrict ourselves to the profinite properties of irreducible 3-manifolds. The technique of this paper mainly adapts to   irreducible 3-manifolds with empty or incompressible toral boundary, as they  admit JSJ-decomposition  by merely tori. 
In this section, we shall briefly review the ingredients of \cite[Theorem B]{WZ19} which encodes the information of JSJ-decomposition from the profinite completion, and obtain a new description of this result (\autoref{THM: Profinite Isom up to Conj}) which better fits in the technique of this paper.

\subsection{JSJ-decomposition as a graph of groups}\label{SEC: JSJ graph}

The JSJ-decomposition together with the hyperbolization theorem yields the modern version as follows, and we refer the readers to \cite[Theorem 1.7.6]{AFW15} for general references.

\begin{theorem}[JSJ-Decomposition]
Let $M$ be a compact, orientable, irreducible 3-manifold with empty or incompressible toral boundary. Then there exists a minimal collection $\mathcal{T}$ of disjoint embedded incompressible tori, such that each connected component of $M\setminus \mathcal{T}$ is either a Seifert fibered space or a finite-volume hyperbolic manifold. In addition, $\mathcal{T}$ is unique up to isotopy.
\end{theorem}

We say that the JSJ-decomposition is \textit{non-trivial}, if $\mathcal{T}$ is non-empty. The JSJ-decomposition identifies $\pi_1M$ as the fundamental group of a graph of groups, based on the geometric construction introduced in \autoref{SEC: Graph of Groups}. We denote this graph of group $(\mathcal{G}_M,\Gamma_M)$ as the \textit{JSJ-graph of groups} of $M$. Indeed, the underlying graph $\Gamma_M$ is a finite graph, the edge groups are $\pi_1(T^2)\cong \mathbb{Z}\times \mathbb{Z}$, and the vertex groups are the fundamental groups of the Seifert fibered or hyperbolic JSJ-pieces.

Wilton-Zalesskii \cite[Theorem A]{WZ10} showed that the JSJ-graph of groups is efficient when $M$ is closed, based on the finite covers of the geometric pieces constructed by Hamilton \cite{Ham01}, and the criteria for residual  finiteness by Hempel \cite{Hem87}. The proof works exactly the same when $M$ has incompressible toral boundary, see also  \cite{WilkesJSJ}.

\begin{proposition}[{\cite[Theorem A]{WZ10}}]\label{PROP: JSJ Efficient}
Let $M$ be a compact, orientable, irreducible 3-manifold with empty or incompressible toral boundary. Then the JSJ-graph of groups of $M$ is efficient.
\end{proposition}

Let $(\widehat{\mathcal{G}_M},\Gamma_M)$ be the profinite completion of the JSJ-graph of groups, which we denote as the \textit{JSJ-graph of profinite groups}. Then \autoref{PROP: JSJ Efficient} together with \autoref{THM: Efficient completion}  implies that $(\widehat{\mathcal{G}_M},\Gamma_M)$ is injective, and  
$$\widehat{\pi_1M} \cong \Pi_1(\widehat{\mathcal{G}_M},\Gamma_M).$$

\subsection{Standard profinite tree}

%
As we have established an isomorphism $\widehat{\pi_1M}\cong \Pi_1(\widehat{\mathcal{G}_M},\Gamma_M)$, 
by \autoref{PROP: Bass-Serre}, $\widehat{\pi_1M}$ acts on the corresponding profinite Bass-Serre tree, which we refer to as the \textit{standard profinite tree} $\T_M$.

\begin{proposition}[{\cite[Proposition 6.8]{Wil18}}]\label{PROP: 4-acylindrical}
Let $M$ be a compact, orientable, irreducible 3-manifold with empty or incompressible toral boundary. Suppose $M$ is not a closed 3-manifold admitting $Sol$-geometry. Then the action by $\widehat{\pi_1M}$ on the standard profinite tree $\T_M$ is 4-acylindrical.
\end{proposition}


In fact, \cite[Proposition 6.8]{Wil18} was only stated for closed manifolds, while  the case-by-case proof  adapts with no alteration to the setting with toral boundary.


The following trailblazing  result was first proven by Wilton-Zalesskii \cite{WZ19}. 

\begin{proposition}[{\cite[Theorem 4.3]{WZ19}}]\label{PROP: equivariant isomorphism of abstract graph}
Let $M$ and $N$ be compact, orientable, irreducible 3-manifolds with empty or incompressible toral boundary, which are not closed 3-manifolds admitting $Sol$-geometry. Suppose $f:\widehat{\pi_1M}\ttt\widehat{\pi_1N}$ is an isomorphism of profinite groups. We may forget the topology and the orientation of the standard profinite trees $\T_M$ and $\T_N$, and view them as abstract combinatorial graphs. Then there exists an $f$-equivariant isomorphism of abstract unoriented graphs $\Phi: \T_M \to \T_N$.
\end{proposition}

In fact, \cite[Theorem 4.3]{WZ19} proved this proposition in the closed case.  
In general,  the two essential steps, the efficiency of the JSJ-graph of groups and the acylindricity of $\widehat{\pi_1M}\curvearrowright \T_M$, also work for the bounded case as shown by \autoref{PROP: JSJ Efficient} and \autoref{PROP: 4-acylindrical}. The remaining proof follows from \cite[Theorem 4.3]{WZ19} with no alteration.

\subsection{Congruent isomorphism for JSJ-decomposition}

Combining \autoref{PROP: equivariant isomorphism of abstract graph} with \autoref{LEM: Equivariant isomorphism of BS tree induce congruent isomorphism}, we can transform the result of Wilton-Zalesskii into the formulation of congruent isomorphisms between JSJ-graphs of profinite groups.

\begin{theorem}\label{THM: Profinite Isom up to Conj}
Suppose $M$ and $N$ are two orientable, irreducible 3-manifolds with empty or incompressible toral boundary, which are not closed 3-manifolds admitting $Sol$-geometry. If $f:\widehat{\pi_1M}\ttt \widehat{\pi_1N}$ is an isomorphism of profinite groups, then $f$ induces a congruent isomorphism $f_{\bullet}$ between their JSJ-graph of profinite groups $(\widehat{\mathcal{G}_M},\Gamma_M)$ and $(\widehat{\mathcal{G}_N},\Gamma_N)$ after reassigning orientations of the JSJ-graphs. In addition,  $\F$ sends   hyperbolic vertices to  hyperbolic vertices, and sends Seifert fibered vertices to Seifert fibered vertices. 
\end{theorem}
\begin{proof}
Let us first explain the assignment of orientations. 
By \autoref{PROP: equivariant isomorphism of abstract graph}, there is an $f$-equivariant isomorphism $\Phi: \T_M \to \T_N$ as abstract unoriented graphs. The $f$-equivariance implies  an isomorphism between the unoriented quotient graphs $\F:\Gamma_M=\T_M/ \widehat{\pi_1M}\to \T_N/\widehat{\pi_1N}= \Gamma_N$. 
We can choose orientations on $\Gamma_M$ and $\Gamma_N$, so that the induced $\F:\Gamma_M\to \Gamma_N$ respects orientation. The orientations lift to $\T_M$ and $\T_{N}$ so that the combinatorial isomorphism $\Phi:\T_M\to \T_{N}$ also respects the orientation.

Recall that the JSJ-graphs of profinite groups $(\widehat{\mathcal{G}_M},\Gamma_M)$ and $(\widehat{\mathcal{G}_N},\Gamma_N)$ are injective, according to \autoref{PROP: JSJ Efficient} and \autoref{THM: Efficient completion}. As a result, $f$ induces a congruent isomorphism between the JSJ-graph of profinite groups as shown by \autoref{LEM: Equivariant isomorphism of BS tree induce congruent isomorphism}.

The detection of hyperbolic or Seifert type of the vertices follows from the works of Wilton-Zalesskii \cite{WZ17,WZ17b} (see also \cite[Theorem 4.18 and 4.20]{Reid:2018}), where they showed that a finite-volume hyperbolic 3-manifold and a compact Seifert fibered space do not have profinitely isomorphic fundamental groups. 
\end{proof}


\section{Peripheral $\Zx$-regularity}\label{SEC: ZX-regular}



In \autoref{DEF: Intro peripheral regular}, we introduced the definition of peripheral $\Zx$-regularity. In this section, we shall present a detailed explanation on this concept. 
Let us first describe the peripheral subgroups in the profinite completion of a 3-manifold group.

\begin{lemma}\label{COR: Peripheral group tensor}
Suppose $M$ is a compact, orientable 3-manifold with incompressible toral boundary. For each peripheral subgroup $\pi_1(\partial_iM)\cong \Z\times\Z$, there is a natural isomorphism $$\tensor \pi_1(\partial_iM)\cong \widehat{\pi_1(\partial_iM)}\cong \overline{\pi_1(\partial_iM)}.$$
 In particular,  $\overline{\pi_1(\partial_iM)}\cong \widehat{\Z}\times \widehat{\Z}$.
\end{lemma}
\begin{proof}
\cite[Theorem 1]{Ham01} shows that any abelian subgroup of $\pi_1M$ is separable in $\pi_1M$. 
Thus, all finite-index subgroups of $\pi_1(\partial_iM)\cong \Z\times\Z$ are separable in $\pi_1M$, and consequently, 
the profinite topology of $\pi_1M$ induces the full profinite topology on $\pi_1(\partial_iM)$. 
This lemma then follows \autoref{PROP: Left exact} which implies $\widehat{\pi_1(\partial_iM)}\cong \overline{\pi_1(\partial_iM)}$, and  \autoref{LEM: Completion of abelian group}  which implies  $\tensor \pi_1(\partial_iM)\cong \widehat{\pi_1(\partial_iM)}$.
\end{proof}

\subsection{Peripheral structure  in the profinite completion}

\subsubsection{The hyperbolic case}

For a cusped hyperbolic 3-manifold, the peripheral subgroups form a malnormal family in its fundamental group. The same conclusion holds in the profinite completion.

\begin{lemma}[{\cite[Lemma 4.5]{WZ17}}]\label{LEM: Malnormal}
Let $M$ be a finite-volume hyperbolic 3-manifold with cusps. Let $\left\{\pi_1\partial_iM\right\}$ be the conjugacy representatives of the peripheral subgroups. Then their closures $\{\overline{\pi_1\partial_iM}\}$ in the profinite completion $\widehat{\pi_1M}$ form a malnormal family, i.e.  $\overline{\pi_1\partial_iM}\cap g\overline{\pi_1\partial_jM}g^{-1}$ is trivial unless $i=j$ and $g\in \overline{\pi_1\partial_iM}$.
\end{lemma}

In particular, the closures of distinct peripheral subgroups are not conjugate, and $N_{\widehat{\pi_1M}}(\overline{\pi_1\partial_iM})=\overline{\pi_1\partial_iM}$.

\begin{lemma}\label{LEM: Hyp preserving peripheral}
Suppose $M$ and $N$ are finite-volume hyperbolic 3-manifolds, and $f:\widehat{\pi_1M}\ttt\widehat{\pi_1N}$ is an isomorphism. 
Then $f$ respects the peripheral structure. 
\end{lemma}
\begin{proof}

\cite[Proposition 3.1]{WZ19} shows that any closed subgroup of $\widehat{\pi_1M}$ isomorphic to $\widehat{\Z}\times \widehat{\Z}$ conjugates into the closure of a unique peripheral subgroup. In particular, when $M$ is closed, $\widehat{\pi_1M}$ does not contain any closed subgroup isomorphic to $\widehat{\Z}\times \widehat{\Z}$ by \cite[Theorem A]{WZ17}.

According to \autoref{LEM: Malnormal}, the closure of any peripheral subgroup in $\widehat{\pi_1M}$ is malnormal. Thus, it cannot be properly contained in any abelian subgroup. Therefore, the closure of the peripheral subgroups in $\widehat{\pi_1M}$ are exactly the distinct conjugacy representatives of the maximal closed subgroups in $\widehat{\pi_1M}$ isomorphic to $\widehat{\Z}\times\widehat{\Z}$.

This holds likewise for $\widehat{\pi_1N}$. Thus, the peripheral structure of a finite-volume hyperbolic 3-manifold is given by the conjugacy classes of the maximal closed $\widehat{\Z}\times\widehat{\Z}$ subgroups, and is preserved by the isomorphism $f:\widehat{\pi_1M}\ttt\widehat{\pi_1N}$.
\end{proof}

\subsubsection{The Seifert fibered case}

Any compact orientable Seifert fibered space $M$ with non-empty incompressible toral boundary falls into exactly one of the following three cases.
\begin{enumerate}[leftmargin=*]
\item $M\cong S^1\times S^1\times I$. In this case, the two boundary components are parallel.
\item $M\cong S^1\ttimes S^1\ttimes I$ is the twisted $I$-bundle over Klein bottle, which is also referred to as the \textit{minor Seifert manifold} following the notation of Wilkes \cite{Wil17}.
\item $M$ fibers over a base orbifold $\O$ with negative Euler characteristic. This is referred to as a \textit{major Seifert manifold} following \cite{Wil17}.
\end{enumerate}

\begin{lemma}[{\cite[Lemma 6.4]{Wil18}}]\label{LEM: major peripheral}
Suppose $M$ is a major Seifert fibered manifold. Let $\{\pi_1\partial_iM\}$ be the conjugacy representatives of the peripheral subgroups, and let $K$ be the infinite cyclic subgroup of $\pi_1M$ generated by a  regular fiber. Then for $g\in \widehat{\pi_1M}$,  $g\overline{\pi_1\partial_iM}g^{-1}\cap\overline{\pi_1\partial_jM}=\overline{K}$ unless $i=j$ and $g\in \overline{\pi_1\partial_iM}$.
\end{lemma}

In particular, the closures of distinct peripheral subgroups are not conjugate, and $N_{\widehat{\pi_1M}}(\overline{\pi_1\partial_iM})=\overline{\pi_1\partial_iM}$.

\begin{lemma}\label{LEM: minor peripheral}
Let $M$ be the minor Seifert fibered manifold, and let $\pi_1\partial M$ be the peripheral subgroup of the unique boundary component. Then, $\overline{\pi_1\partial M}$ is an index-2 normal subgroup in $\widehat{\pi_1M}$, and each $g\in  \widehat{\pi_1M}\setminus \overline{\pi_1\partial M}$ acts by conjugation on $\overline{\pi_1\partial M}$ as the profinite completion of an involution in $\Aut(\pi_1\partial M)$.
\end{lemma}
\begin{proof}
It is known that $\pi_1M\cong \Z\rtimes \Z=\left\langle a\right\rangle \rtimes\left\langle b\right\rangle $, where $bab^{-1}=a^{-1}$. 
\autoref{PROP: Completion Exact} and \autoref{PROP: Left exact} then imply that $\widehat{\pi_1M}\cong \widehat{\Z}\rtimes \widehat{\Z}=\widehat{\langle a\rangle}\rtimes\widehat{\langle b\rangle}$. 
For any $\lambda\in \widehat{\Z}$, $b^{\lambda}$ acts on $\widehat{\Z}=\widehat{\left\langle a\right\rangle }$ as scalar multiplication by $(-1)^{\lambda\pmod 2}$.

Note that the peripheral subgroup $\pi_1\partial M=\left\langle a\right\rangle \times \left\langle b^2\right\rangle$ is an index-2 normal subgroup in $\pi_1M$. Thus by \autoref{THM: correspondence of subgroup}, $\overline{\pi_1\partial M}=\widehat{\left\langle a\right\rangle }\times \widehat{\left\langle b^2\right\rangle }\cong \widehat{\Z}\times\widehat{\Z}$ is also an index-2 normal subgroup in $\widehat{\pi_1M}\cong \widehat{\Z}\rtimes\widehat{\Z}$.

Any element $g\in \widehat{\pi_1M}\setminus \overline{\pi_1\partial M}$ can be expressed as $g=bh$ for some $h\in \overline{\pi_1\partial M}$, and $b$ acts by conjugation on $\overline{\pi_1\partial M}$ as the profinite completion of the involution $\sigma\in \Aut(\pi_1\partial M)$ defined by $\sigma(a)=a^{-1}$ and $\sigma(b^2)=b^2$. Then, $g$ also acts by  conjugation on $\overline{\pi_1\partial M}$ as~$\widehat{\sigma}$.
\end{proof}

\subsubsection{The general case}

\begin{lemma}\label{LEM: peripheral subgroup fix unique vertex}
Suppose $M$ is a compact, orientable, irreducible 3-manifold with incompressible toral boundary, and $\widehat{\pi_1M}$ acts on the standard profinite tree $\T_M$. Let $\partial_iM$ be a boundary component of $M$, and let $\pi_1\partial_iM$ be a conjugacy representative of the peripheral subgroup.   Then the subgroup $\overline{\pi_1\partial_iM}\subseteq \widehat{\pi_1M}$ fixes a unique vertex in $\T_M$.
\end{lemma}
\begin{proof}
It suffices to assume $M$ has non-trivial JSJ-decomposition. Suppose $\partial_iM$ belongs to a JSJ-piece $M_x$.  Up to conjugation, we may assume $\pi_1\partial_iM \subseteq \pi_1M_x\subseteq \pi_1M$. According to \autoref{PROP: Bass-Serre}, $\widehat{\pi_1M_x}$ fixes a vertex $v\in V(\T_M)$, so that $v$ belongs to the pre-image of $x$ under  the quotient map $\T_M\to \Gamma_M$. Thus, $\overline{\pi_1\partial_iM}$ also fixes $v$. 

Suppose $\overline{\pi_1\partial_iM}$ fixes another vertex $w\neq v\in V(\T_M)$. Then, $\overline{\pi_1 \partial_iM}$ fixes pointwise the minimal profinite subtree $\Delta$ containing $v$ and $w$ (see \cite[Proposition 2.4.9]{Rib17}). By \autoref{def: BStree}, $E(\T_M)$ is closed in the disjoint union topology $\T_M=V(\T_M)\bigsqcup E(\T_M)$, so $E(\Delta)=E(\T_M)\cap \Delta$ is also closed in $\Delta$. Then, according to \cite[Proposition 2.1.6 (c)]{Rib17}, $\Delta$ contains an edge $e$ adjoint  to the vertex $v$. 
Let $\partial_kM_x$ denote the boundary component of $M_x$ given by the JSJ-torus corresponding to the image of $e$ in $\Gamma_M$. Then, $\partial_iM\neq \partial_kM_x$. In particular, $\#\partial M_x\ge 2$, so $M_x$ is either a hyperbolic or a major Seifert piece. 

In addition, according to \autoref{PROP: Bass-Serre}, $\Stab(e) $ is conjugate to $\overline{\pi_1\partial_k M_x}$ within $\widehat{\pi_1M_x}=\Stab(v)$. Thus, $\overline{\pi_1\partial_iM}\subseteq h\overline{\pi_1\partial_k M_x}h^{-1}$ for some $h\in \widehat{\pi_1M_x}$. However, this yields a contradiction with \autoref{LEM: Malnormal} if $M_x$ is hyperbolic, or with \autoref{LEM: major peripheral} if $M_x$ is a major Seifert piece, finishing the proof. 
\end{proof}

\begin{proposition}\label{PROP: peripheral structure}
Let $M$ be a compact, orientable 3-manifold with incompressible toral boundary. For each boundary component $\partial_iM$, let  $\pi_1\partial_iM$ denote a conjugacy representative of peripheral subgroup in $\pi_1M$.
\begin{enumerate}[leftmargin=*]
\item \label{6.2-1}
Let $\partial_iM$ and $\partial_jM$ be two different boundary components of $M$. 
\\
(i) If $\partial_iM$ and $\partial_jM$  belong to the same prime factor homeomorphic to $S^1\times S^1\times I$, then $\pi_1\partial_iM$ and $\pi_1\partial_jM$ are conjugate in $\pi_1M$.
\\
(ii) Otherwise, $\overline{\pi_1\partial_iM}$ and $\overline{\pi_1\partial_jM}$ are not conjugate in $\widehat{\pi_1M}$.
\item\label{6.2-2} Let $N_{\widehat{\pi_1M}}\left(\overline{\pi_1\partial_iM}\right)$ denote the normalizer of $\overline{\pi_1\partial_iM}$ in $\widehat{\pi_1M}$.
\\
(i) If the prime factor of $M$ containing $\partial_iM$ is homeomorphic to the orientable $I$-bundle over Klein bottle, then 
$N_{\widehat{\pi_1M}}\left(\overline{\pi_1\partial_iM}\right)$ contains $\overline{\pi_1\partial_iM}$ as an index-2 subgroup, and any element $g\in N_{\widehat{\pi_1M}}\left(\overline{\pi_1\partial_iM}\right)\setminus \overline{\pi_1\partial_iM}$ acts by conjugation on $\overline{\pi_1\partial_iM}$ as the profinite completion of an involution in $\Aut(\pi_1\partial_iM)\cong \mathrm{GL}_2(\Z)$.
\\
(ii) Otherwise, $N_{\widehat{\pi_1M}}\left(\overline{\pi_1\partial_iM}\right)=\overline{\pi_1\partial_iM}$, and acts trivially by conjugation on $\overline{\pi_1\partial_iM}$.
\end{enumerate}
\end{proposition}

\begin{proof}
We first show that this proposition holds when $M$ is irreducible. 
When $M$ has trivial JSJ-decomposition, the conclusion follows from 
\autoref{LEM: Malnormal}, \autoref{LEM: major peripheral} and \autoref{LEM: minor peripheral}. 
Now we assume that $M$ has non-trivial JSJ-decomposition. 

Let $\partial_i M$ and $\partial_jM$ be two possibly identical boundary components, which belong to the JSJ-pieces $M_x$ and $M_y$ respectively. For each $v\in V(\Gamma_M)$, fix a conjugacy representative $\pi_1M_v\le \pi_1M$.  Up to conjugation, we may assume $\pi_1\partial_iM \subseteq \pi_1M_x\subseteq \pi_1M$ and $\pi_1\partial_jM\subseteq \pi_1M_y\subseteq \pi_1M$. Suppose $\overline{\pi_1 \partial_iM}=g\overline{\pi_1\partial_jM}g^{-1}$ for some $g\in \widehat{\pi_1M}$. It suffices to show that $M_x=M_y$, $\partial_iM=\partial_jM$, and $g\in \overline{\pi_1 \partial_iM}$, so that we deduce the conclusion from the proven case for the JSJ-piece.


We consider the action by $\widehat{\pi_1M}$ on the standard profinite tree $\T_M$, and let $q: \T_M\to \T_M/\widehat{\pi_1M}=\Gamma_M$ denote the quotient map. Then $\widehat{\pi_1M_x}$ fixes a unique vertex $v\in q^{-1}(x)$, and $\widehat{\pi_1M_y}$ fixes a  unique vertex $w\in q^{-1}(y)$, since the edge stabilizers are abelian. Thus, $\overline{\pi_1\partial_iM}\subseteq \widehat{\pi_1M_x}$ fixes $v$, and $g\overline{\pi_1\partial_jM}g^{-1}\subseteq g\widehat{\pi_1M_y}g^{-1}$ fixes $g\cdot w$. As a result, $\overline{\pi_1 \partial_iM}$ fixes  both $v$ and $g\cdot w$ in $V(\T_M)$. 

According to \autoref{LEM: peripheral subgroup fix unique vertex}, $v=g\cdot w$. 
In particular, $v$ and $w$ belong to the same orbit; thus, $M_x=M_y$, which in turn  implies that $v=w$ and $g\in \Stab (v)=\widehat{\pi_1M_x}$. Therefore, $\overline{\pi_1 \partial_iM}=g\overline{\pi_1\partial_jM}g^{-1}$ within $\widehat{\pi_1M_x}$, so $\partial_iM=\partial_jM$ and $g\in \overline{\pi_1 \partial_iM}$ according to \autoref{LEM: Malnormal} and \autoref{LEM: major peripheral}.


Now we move on to the case when $M$ is reducible, with prime decomposition $M=M_1\#\cdots\#M_m\#(S^2\times S^1)^r$. Then, fixing the conjugacy representatives, we have
$$
\widehat{\pi_1M}\cong\widehat{\pi_1M_1}\amalg \cdots \amalg\widehat{\pi_1M_m}\amalg \widehat{F_r}.
$$
Given two peripheral subgroups $\pi_1\partial_iM$ and $\pi_1\partial_jM$, up to conjugation, we may as well assume $\pi_1\partial_iM\subseteq \pi_1M_k\subseteq \pi_1M$ and $\pi_1\partial_jM\subseteq \pi_1M_l \subset\pi_1M$. In this case if $\overline{\pi_1\partial_iM}=g\cdot\overline{\pi_1\partial_jM}\cdot g^{-1}$ for  some $g\in\widehat{\pi_1M}$, then $k=l$ and $g\in \widehat{\pi_1M_k}$. Now the conclusion follows from the irreducible case proved for $M_k$.
\end{proof}

 \subsection{Well-definedness of peripheral $\Zx$-regularity}\label{subsec: Well-def}


Recall in \autoref{DEF: Intro peripheral regular}, we have defined the notion of peripheral $\Zx$-regularity. 
%
Let $M$ and $N$ be compact, orientable 3-manifolds with incompressible toral boundary. Let $\partial_i M$ be a boundary component of $M$ and let $\pi_1\partial_iM$ be a conjugacy representative of the corresponding peripheral subgroup. Let $f:\widehat{\pi_1M}\ttt\widehat{\pi_1N}$ be a profinite isomorphism. For $\lambda\in \Zx$, $f$ is \textit{peripheral $\lambda$-regular} at $\partial_iM$ if the following two conditions hold.
\begin{enumerate}[leftmargin=*]
\item\label{zxregulardef1} There exists a boundary component $\partial_j N$ of $N$, a conjugacy representative of the peripheral subgroup $\pi_1\partial_jN$, and an element $g\in \widehat{\pi_1N}$ such that $f(\overline{\pi_1\partial_iM})=C_g^{-1} (\overline{\pi_1\partial_jN})$.
\item\label{zxregulardef2} 
There exists an isomorphism $\psi:\pi_1\partial_iM\ttt\pi_1\partial_jN$ such that 
\begin{equation*}
C_{g}\circ f=\lambda\otimes \psi:\, \overline{\pi_1\partial_iM}\cong \tensor \pi_1\partial_iM \tto \tensor \pi_1\partial_jN\cong \overline{\pi_1\partial_jN}.
\end{equation*}
\end{enumerate}
And $f$ is  called \textit{peripheral $\Zx$-regular} at $\partial_iM$ if it is peripheral $\lambda$-regular for some $\lambda\in \Zx$.

\autoref{PROP: peripheral structure} guarantees the ``well-definedness'' of peripheral $\Zx$-regularity. 
In fact, \autoref{PROP: peripheral structure}~(\ref{6.2-1}) implies that the boundary component $\partial_jN$ in condition (\ref{zxregulardef1}) is unique, except for the case of parallel boundary components in  $S^1\times S^1\times I$. In addition, \autoref{PROP: peripheral structure}~(\ref{6.2-2}) further implies that the existence of decomposition as $\lambda\otimes \psi$ in condition (\ref{zxregulardef2}) does not depend on the choice of the   conjugacy representative $\pi_1\partial_jN$, nor does it depend on the choice of the  element $g$ which induces the conjugation. In particular, the coefficient $\lambda\in \Zx$ is unique up to a $\pm$-sign. 


\subsection{Geometric meaning of peripheral $\Zx$-regularity}
The geometric explanation of peripheral $\Zx$-regularity is that the given profinite isomorphism carries  profinite isomorphisms between the Dehn fillings at the corresponding boundary component. 

To be explicit, suppose $f:\widehat{\pi_1M}\ttt\widehat{\pi_1N}$ is peripheral $\Zx$-regular at $\partial_iM$, and $C_{g}\circ f|_{\overline{\pi_1\partial_iM}}=\lambda\otimes \psi$. For each boundary slope $c\in \pi_1\partial_iM$,  $\psi\in \mathrm{Hom}(\pi_1\partial_iM,\pi_1\partial_jN)$ sends $c$ to a boundary slope $\psi(c)\in \pi_1\partial_jN$. Let $M_c$ be the Dehn filling of $M$ at $\partial_iM$ along the slope $c$, and let $N_{\psi(c)}$ be the Dehn filling of $N$ at $\partial_jN$ along the slope $\psi(c)$. Then $\widehat{\pi_1M_c}\cong \widehat{\pi_1N_{\psi(c)}}$. 

Indeed, $\pi_1M_c\cong \pi_1M/\langle \!\langle c\rangle\!\rangle$, and \autoref{PROP: Completion Exact} implies $\widehat{\pi_1M_c}\cong \widehat{\pi_1M}/ \overline{\langle\!\langle c\rangle\!\rangle}$. Then the profinite isomorphism $f$ descends to an isomorphism between the quotients
\begin{equation*}
\begin{tikzcd}
\widehat{\pi_1M_c}\cong \quo{\widehat{\pi_1M}}{\overline{\left\langle\!\langle c\right\rangle\!\rangle}} \arrow[r,"\cong"] &  \quo{\widehat{\pi_1N}}{\overline{\left\langle\!\langle \psi(c)\right\rangle\!\rangle}}\cong \widehat{\pi_1N_{\psi(c)}}.
\end{tikzcd}
\end{equation*}

This idea is further expanded in \cite{Xu24b}.

\subsection{Passing peripheral $\Zx$-regularity to JSJ-pieces}

In this subsection, we consider the following setting. 
Suppose $M$ and $N$ are compact, orientable, irreducible 3-manifolds with empty or  incompressible toral boundary which are not closed $Sol$-manifolds. Suppose $f:\widehat{\pi_1M}\ttt \widehat{\pi_1N}$ is a profinite isomorphism, and let  $f_{\bullet}$ be the congruent isomorphism between the JSJ-graphs of profinite groups $(\mathcal{G}_M,\Gamma_M)$ and $(\mathcal{G}_N,\Gamma_N)$ induced by $f$ according to \autoref{THM: Profinite Isom up to Conj}.

\begin{lemma}\label{lem: 111111}
Let $\partial_iM$ be a boundary component of $M$ which belongs to a JSJ-piece $M_v$. Suppose $f(\overline{\pi_1\partial_iM})$ is conjugate to $\overline{\pi_1\partial_jN}$ in $\widehat{\pi_1N}$ for some boundary component $\partial_jN$. Then $\partial_jN\subseteq N_{\F(v)}$, and $f_v(\overline{\pi_1\partial_iM})$ is conjugate to $\overline{\pi_1\partial_jN}$ within $\widehat{\pi_1N_{\F(v)}}$. 
\end{lemma}
\begin{proof}
For each $u\in V(\Gamma_M)$, fix a conjugacy representative $\pi_1M_u\le \pi_1M$.  
Suppose $\partial_jN$ belongs to the JSJ-piece $N_w$. Up to conjugation,  we may assume $\pi_1\partial_iM\subseteq \pi_1M_v\subseteq \pi_1M$, and $\pi_1\partial_j N\subseteq \pi_1N_w\subseteq \pi_1N$. Based on this assumption, suppose $f(\overline{\pi_1\partial_iM})=C_{h}(\overline{\pi_1\partial_jN})$ for some $h\in \widehat{\pi_1N}$. 

By construction of \autoref{THM: Profinite Isom up to Conj}, there exists an element $g\in \widehat{\pi_1N}$ such that the following diagram commutes.
\begin{equation}\label{6.10dig}
\begin{tikzcd}
\widehat{\pi_1M_v} \arrow[rr, "f_v"] \arrow[d,hook] &                                   & \widehat{\pi_1N_{\F(v)}} \arrow[d,hook] \\
\widehat{\pi_1M} \arrow[r, "f"]                                            & \widehat{\pi_1N} \arrow[r, "C_{g}^{-1}"] & \widehat{\pi_1N}                                                 
\end{tikzcd}
\end{equation}
In particular, $f(\overline{\pi_1\partial_iM})\subseteq C_{g}(\widehat{\pi_1N_{\F(v)}})$, so
\begin{equation}\label{equfixboth}
\overline{\pi_1\partial_jN}\subseteq C_{h^{-1}g}(\widehat{\pi_1N_{\F(v)}}).
\end{equation}

Consider the action by $\widehat{\pi_1N}$ on the standard profinite tree $\T_N$. Denote the quotient map as $q: \T_N\to \T_N/\widehat{\pi_1N}=\Gamma_N$. Since the edge stabilizers are abelian, $\widehat{\pi_1N_{\F(v)}}$ fixes a unique vertex $x\in q^{-1}(\F(v))$, and  $\widehat{\pi_1N_w}$ fixes a unique vertex $y\in q^{-1}(w)$. Then by (\ref{equfixboth}), $\overline{\pi_1\partial_jN}\subseteq \widehat{\pi_1N_w}$ fixes both $y$ and $h^{-1}g\cdot x$.  According to \autoref{LEM: peripheral subgroup fix unique vertex}, $y=h^{-1}g\cdot x$. In particular, $x$ and $y$ belong to the same $\widehat{\pi_1N}$-orbit, so $w=\F(v)$ and in turn, $x=y$. 

Therefore, $\partial_jN\subseteq N_{\F(v)}$, and $h^{-1}g\in \Stab(x)=\widehat{\pi_1N_{\F(v)}}$. Then, according to diagram~(\ref{6.10dig}), $f_v(\overline{\pi_1\partial_iM})=C_{g^{-1}h}(\overline{\pi_1\partial_jN})$, where $g^{-1}h\in \widehat{\pi_1N_{\F(v)}}$. 
\end{proof}

\begin{corollary}\label{COR: peripheral preserving}
$f$ respects the peripheral structure if and only if for each $v\in V(\Gamma_M)$, $f_v$ respects the peripheral structure. 
\end{corollary}
\begin{proof}
Let us first prove the ``only if'' direction. In fact, according to the coherence relation (\ref{Diagram Isom Conj}) in the congruent isomorphism, each $f_v$ automatically preserves the peripheral subgroups given by the JSJ-tori, up to conjugacy. On the other hand, according to \autoref{lem: 111111}, the components of $\partial M\cap M_v$ and $\partial N\cap N_{\F(v)}$ are in one-to-one correspondence, and their corresponding conjugacy classes of peripheral subgroups  are also preserved by $f_v$.

The ``if'' direction is easier. For $\partial_iM\subseteq M_v\cap \partial M$, by assumption, $f_v(\overline{\pi_1\partial_iM})$ is conjugate to a peripheral subgroup $\overline{\pi_1\partial_jN_{\F(v)}}$. Then, $\partial_jN_{\F(v)}\subseteq N_{\F(v)}\cap \partial N$, since the peripheral structures of $M_v$ and $N_{\F(v)}$ given by the JSJ-tori are preserved by $f_v$ due to the same reason as above. This gives a one-to-one correspondence between the components of $M_v\cap \partial M$ and $N_{\F(v)}\cap \partial N$, which can be  combined into a one-to-one correspondence between the components of $\partial M$ and $\partial N$. In addition, $f(\overline{\pi_1\partial_iM})$ is conjugate to $\overline{\pi_1\partial_jN_{\F(v)}}$ in $\widehat{\pi_1N}$ according to the commutative diagram (\ref{6.10dig}), so $f$ also respects the peripheral structure. 
\end{proof}

When considering 3-manifolds admitting JSJ-decomposition, it is a crucial step to switch from the peripheral $\Zx$-regularity of  a profinite isomorphism between two  irreducible 3-manifolds to the peripheral $\Zx$-regularity of the induced  profinite isomorphism between  the JSJ-pieces. 

\begin{lemma}\label{EX: peripheral regular}
Let $\partial_iM$ be a boundary component of $M$ which belongs to a JSJ-piece $M_v$. Then $f$ is peripheral $\Zx$-regular at $\partial_iM$ if and only if $f_v:\widehat{\pi_1M_v}\ttt \widehat{\pi_1N_{\F(v)}}$ is peripheral $\Zx$-regular at $\partial_iM$.
\end{lemma}

\begin{proof}
The ``if'' direction is clear from the commutative diagram (\ref{6.10dig}). Now we expand on  the ``only if'' direction. 

Indeed, $f(\overline{\pi_1\partial_iM})$ is conjugate to  a peripheral subgroup $\overline{\pi_1\partial_jN}$ in $\widehat{\pi_1N}$.  By \autoref{lem: 111111}, $\partial_jN\subseteq N_{\F(v)}$. Up to taking suitable conjugacy representatives, we assume $\pi_1\partial_iM\subseteq \pi_1M_v\subseteq \pi_1M$ and $\pi_1\partial_jN\subseteq \pi_1N_{\F(v)}\subseteq \pi_1N$, and there exists $k\in \widehat{\pi_1N_{\F(v)}}$ such that $f_v(\overline{\pi_1\partial_iM})=C_{k}^{-1}(\overline{\pi_1\partial_jN})$. Then, by the commutative diagram (\ref{6.10dig}), there exists $g\in \widehat{\pi_1N}$ such that $f_v=C_{g}^{-1}\circ f: \widehat{\pi_1M_v}\to \widehat{\pi_1N_{\F(v)}}$ and $f(\overline{\pi_1\partial_iM})=C_{gk^{-1}}(\overline{\pi_1\partial_jN})$.

According to the well-definedness of peripheral $\Zx$-regularity (\autoref{PROP: peripheral structure}), the map $$C_{kg^{-1}}\circ f|_{\overline{\pi_1\partial_iM}}: \overline{\pi_1\partial_iM}\cong \tensor \pi_1\partial_iM \tto \tensor \pi_1\partial_jN \cong \overline{\pi_1\partial_j N}$$ can be decomposed as $\lambda\otimes \psi$, where $\lambda\in \Zx$ and $\psi: \pi_1\partial_iM\ttt\pi_1\partial_jN$, regardless of the choice of the conjugator $kg^{-1}$ and the conjugacy representatives of peripheral subgroups. 
 This implies that $$C_{k}\circ f_v|_{\overline{\pi_1\partial_iM}}=\lambda\otimes \psi: \overline{\pi_1\partial_iM}\subseteq  \widehat{\pi_1M_x}\tto \overline{\pi_1\partial_jN}\subseteq \widehat{\pi_1N_{\F(x)}} ,$$ where $k\in \widehat{\pi_1N_{\F(v)}}$, which is to say that  $f_v$ is peripheral $\Zx$-regular at $\partial_iM$.
 %
%
\end{proof}

\subsection{The gluing lemma}\label{subsec: Gluing}

The most essential application of peripheral $\Zx$-regularity is displayed in   the commutative diagram (\ref{indig1}). 
In fact, 
under the condition of peripheral $\Zx$-regularity, we can directly compare the gluing maps between corresponding JSJ-pieces in a profinitely isomorphic pair of irreducible 3-manifolds. 
%
Namely, at each edge of the JSJ-graph, one derives a gluing map $\psi_e=\varphi_1\circ\varphi_0^{-1}: \pi_1\partial_kM_u\to \pi_1\partial_iM_v$ as is witnessed in the JSJ-graph of groups
\begin{equation*}
\begin{tikzcd}
\pi_1M_u\supseteq \pi_1\partial_kM_u & \pi_1T^2_e \arrow[l, "\varphi_0"',"\cong"] \arrow[r, "\varphi_1","\cong"'] & \pi_1\partial_iM_v\subseteq \pi_1M_v.
\end{tikzcd}
\end{equation*}


\begin{lemma}[Gluing lemma]\label{inLem1'}
Suppose $M$ and $N$ are compact, orientable, irreducible 3-manifolds with empty or  incompressible toral boundary which are not closed $Sol$-manifolds.     Let $f_\bullet$ be a congruent isomorphism between the JSJ-graphs of profinite groups $(\widehat{\mathcal{G}_M},\Gamma_M)$ and $(\widehat{\mathcal{G}_N},\Gamma_N)$. 
Suppose $u,v\in V(\Gamma_M)$ are two possibly identical vertices connected by an edge $e\in E(\Gamma_M)$. Following diagram (\ref{indig1}), denote the gluings at the  JSJ-tori corresponding to $e$ and $\F(e)$ as
\begin{equation*}
M_u\supseteq \partial_kM_u \longleftrightarrow \partial_iM_v\subseteq M_v \text{ and }  N_{\F(u)}\supseteq \partial_lN_{\F(u)} \longleftrightarrow \partial_jN_{\F(v)}\subseteq N_{\F(v)}.
\end{equation*}
    \begin{enumerate}[leftmargin=*]
        \item\label{lemmain-1'} The profinite isomorphism $f_u:\widehat{\pi_1M_u}\ttt \widehat{\pi_1N_{\F(u)}}$ is peripheral $\Zx$-regular at $\partial_kM_u$ if and only if $f_v:\widehat{\pi_1M_v}\ttt \widehat{\pi_1N_{\F(v)}}$ is peripheral $\Zx$-regular at $\partial_iM_v$. In addition, when peripheral $\Zx$-regularity holds, the coefficients in $\Zx$ are identical up to a $\pm$-sign. 
        \item\label{lemmain-2'} 
Assume peripheral $\Zx$-regularity in (\ref{lemmain-1'}) holds. Suppose $f_u$ restricting on $\partial_kM_u$ is induced by a homeomorphism $\Phi_u: M_u\ttt N_{\F(u)}$ with coefficient $\lambda$, and  $f_v$ restricting on $\partial_kM_v$ is induced by a homeomorphism $\Phi_v:M_v\ttt N_{\F(v)}$ with the same coefficient $\lambda$  (see \autoref{inDef: induced by homeo}). 
 Then, 
$\Phi_u$ and $\Phi_v$ are compatible, up to isotopy, with the gluing maps at the JSJ-tori. 
In particular, \\
$\bullet$ 
if $u\neq v$, then $\Phi=(\Phi_u,\Phi_v)$ yields a homeomorphism $M_u\cup_{T^2}M_v\ttt N_{\F(u)}\cup_{T^2}N_{\F(v)}$; \\
$\bullet$ 
if $u=v$ and $\Phi_u=\Phi_v$, then $\Phi_u$ induces a homeomorphism $M_u\cup_{T^2}\ttt N_{\F(u)}\cup_{T^2}$. 
    \end{enumerate}
\end{lemma}
\begin{proof}
(\ref{lemmain-1'}) For brevity, we  identify $\Gamma_M$ and $\Gamma_N$ with the same graph $\Gamma$ through the isomorphism $\F$. Then, the congruent isomorphism $f_\bullet$ implies the following commutative diagram 
  \begin{equation}\label{lem'dig1}
     \begin{tikzcd}[column sep=tiny]
\widehat{\pi_1M_u} \arrow[r,symbol=\supseteq] \arrow[dd, "C_{h_0}\circ f_u"'] & \overline{\pi_1\partial_kM_u} \arrow[dd] & & & & \widehat{\pi_1T^2} \arrow[llll, "\varphi_0"'] \arrow[rrrr, "\varphi_1"] \arrow[dd,"f_e"]   & & & & \overline{\pi_1\partial_iM_v} \arrow[r,symbol=\subseteq] \arrow[dd] & \widehat{\pi_1M_v} \arrow[dd, "C_{h_1}\circ f_v"] \\
 & & & &  & & & & & &\\
\widehat{\pi_1N_u} \arrow[r,symbol=\supseteq ]                              & \overline{\pi_1\partial_lN_u}      & & &     & \widehat{\pi_1T^2} \arrow[llll, "\varphi_0'"'] \arrow[rrrr, "\varphi_1'"] & & & & \overline{\pi_1\partial_jN_v} \arrow[r,symbol=\subseteq]           & \widehat{\pi_1N_v}                              
\end{tikzcd}
\end{equation}
where $h_0\in \widehat{\pi_1N_u}$ and $h_1\in \widehat{\pi_1N_v}$.

By \autoref{COR: Peripheral group tensor}, we further obtain
\begin{equation}\label{lem'dig2}
\begin{tikzcd}[column sep=large, row sep=large]
\tensor \pi_1\partial_kM_u \arrow[r, "1\otimes \psi_e","\cong"'] \arrow[d, "C_{h_0}\circ f_u|_{{\overline{\pi_1\partial_kM_u}}}"',"\cong"] & \tensor \pi_1\partial_iM_v \arrow[d, "C_{h_1}\circ f_v|_{\overline{\pi_1\partial_iM_v}}","\cong"'] \\
\tensor \pi_1\partial_lN_u \arrow[r, "1\otimes \psi_e'","\cong"']                                                                & \tensor \pi_1\partial_jN_v                                                               
\end{tikzcd}
\end{equation}
where $\psi_e$ and $\psi_e'$ are respectively gluing maps of $M$ and $N$ at the JSJ-torus corresponding to $e$.

According to \autoref{PROP: peripheral structure}, regardless of the choice of $h_0\in \widehat{\pi_1N_u}$, if $f_u$ is peripheral $\Zx$-regular at $\partial_kM_u$, then $C_{h_0}\circ f_u|_{{\overline{\pi_1\partial_kM_u}}}$ admits a decomposition as $\lambda\otimes \varphi_u$.  In this case, the commutative diagram (\ref{lem'dig2}) implies that $C_{h_1}\circ f_v|_{{\overline{\pi_1\partial_iM_v}}}$ admits a decomposition as $\lambda\otimes (\psi_e'\circ \varphi_u\circ \psi_e^{-1})$, which is to say that $f_v$ is  peripheral $\lambda$-regular at $\partial_iM_v$. The converse holds similarly by symmetry, which implies (\ref{lemmain-1'}).


\newsavebox{\bbbaaa}
\begin{lrbox}{\bbbaaa}
$
C_{g_0}\circ f_u|_{\overline{\pi_1\partial_kM_u}}:\; \overline{\pi_1\partial_kM_u}\cong \tensor \pi_1\partial_kM_u\xrightarrow{\lambda\otimes \Phi_{u\ast}} \tensor C_{s_0}(\pi_1\partial_lN_u)\cong \overline{C_{s_0}(\pi_1\partial_lN_u)}.
$
\end{lrbox}

(\ref{lemmain-2'}) Let us first consider $f_u$. The commutative diagram (\ref{lem'dig1}) implies $f_u(\overline{\pi_1\partial_k M_u})=C_{h_0}^{-1}(\overline{\pi_1\partial_lN_u})$. Consequently, the closure of a peripheral subgroup corresponding to $\Phi_u(\partial_kM_u)$ is a conjugate of $\overline{\pi_1\partial_lN_u}$ in $\widehat{\pi_1N_u}$. 
Since $N$ is not a closed $Sol$-manifold, $N_u$ is not homeomorphic to $S^1\times S^1\times I$, and by \autoref{PROP: peripheral structure}~(\ref{6.2-1}), $\Phi_u(\partial_kM_u)=\partial_lN_u$. 
In particular, there exists $s_0\in \pi_1N$ such that $\Phi_{u\ast}(\pi_1\partial_kM_u)=C_{s_0}(\pi_1\partial_lN_u)$. 
By  \autoref{inDef: induced by homeo}, there exists an element $g_0\in \widehat{\pi_1N_u}$ such that 
\begin{equation*}
\scalebox{0.88}{\usebox{\bbbaaa}}
\end{equation*}

Since $\pi_1N_u$ is residually finite,  we identify $\pi_1N_u$ as a subgroup in $\widehat{\pi_1N_u}$ through the canonical homomorphism. Then $s_0$ is also viewed as an element in $\widehat{\pi_1N_u}$. Consequently, we derive that $f_u(\overline{\pi_1\partial_k M_u})=C_{g_0}^{-1}(C_{s_0}(\overline{\pi_1\partial_lN_u}))$. Recall that $f_u(\overline{\pi_1\partial_k M_u})=C_{h_0}^{-1}(\overline{\pi_1\partial_lN_u})$. Thus, $$s_0^{-1}g_0h_0^{-1}\in N_{\widehat{\pi_1N_u}}(\overline{\pi_1\partial_lN_u}).$$

\newsavebox{\cccddd}
\begin{lrbox}{\cccddd}
$
C_{h_0}\circ f|_{\overline{\pi_1\partial_kM_u}}=C_{s_0^{-1}g_0h_0^{-1}}\circ C_{h_0}\circ f|_{\overline{\pi_1\partial_kM_u}}=\lambda\otimes (C_{s_0^{-1}}\circ \Phi_{u\ast}): \overline{\pi_1\partial_kM_u}   \to    \overline{ \pi_1\partial_lN_u }.
$
\end{lrbox}

If $N_u$ is not   the orientable $I$-bundle over Klein bottle,  \autoref{PROP: peripheral structure}~(\ref{6.2-2}) implies that $s_0^{-1}g_0h_0^{-1}\in \overline{\pi_1\partial_lN_u}$, so it acts trivially by conjugation on $\overline{\pi_1\partial_lN_u}$. In particular,
\begin{equation*}
\scalebox{0.88}{\usebox{\cccddd}}
\end{equation*}

\newsavebox{\eeefff}
\begin{lrbox}{\eeefff}
$
C_{h_0}\raisebox{1pt}{\scalebox{.8}{$\circ$}} f|_{\overline{\pi_1\partial_kM_u}}=C_{t_0s_0^{-1}g_0h_0^{-1}}\raisebox{1pt}{\scalebox{.8}{$\circ$}} C_{h_0}\raisebox{1pt}{\scalebox{.8}{$\circ$}} f|_{\overline{\pi_1\partial_kM_u}}=\lambda\otimes (C_{t_0s_0^{-1}}\raisebox{1pt}{\scalebox{.8}{$\circ$}}\Phi_{u\ast}): \overline{\pi_1\partial_kM_u}   \to    \overline{ \pi_1\partial_lN_u }.
$
\end{lrbox}

If $N_u$ is  the orientable $I$-bundle over Klein bottle, since $\widehat{\pi_1N_u}$ contains $\overline{\pi_1\partial_lN_u}$ as an index-2 open subgroup and $\pi_1N_u$ is dense in $\widehat{\pi_1N_u}$, we can find an element $t_0\in \pi_1N_u$, such that $t_0s_0^{-1}g_0h_0^{-1}\in \overline{\pi_1\partial_lN_u}$. In this case, we have
\begin{equation*}
\scalebox{0.88}{\usebox{\eeefff}}
\end{equation*}

Let $r_0=s_0^{-1}$ in the former case, and $r_0=t_0s_0^{-1}$ in the latter case. In both cases, we have 
$
r_0\in \pi_1N_u$, and $$C_{h_0}\circ f|_{\overline{\pi_1\partial_kM_u}}=\lambda\otimes (C_{r_0}\circ \Phi_{u\ast}): \overline{\pi_1\partial_kM_u}\to \overline{\pi_1\partial_l N_u}.
$$

\newsavebox{\jjjkkk}
\begin{lrbox}{\jjjkkk}
$
C_{h_1}\circ f|_{\overline{\pi_1\partial_iM_v}}=\lambda\otimes (C_{r_1}\circ \Phi_{v\ast}): \overline{\pi_1\partial_iM_v}\cong \tensor\pi_1\partial_iM_v \to \tensor\pi_1\partial_j N_v \cong \overline{\pi_1\partial_j N_v}.
$
\end{lrbox}

Similarly, for $f_v$, there exists $r_1\in \pi_1N_v$  such that 
\begin{equation*}
\scalebox{0.88}{\usebox{\jjjkkk}}
\end{equation*}

Now we can cancel the coefficient $\lambda$ in the commutative diagram (\ref{lem'dig2}), so that we obtain the following diagram.
\begin{equation*}
\begin{tikzcd}[row sep=large]
\tensor \pi_1\partial_kM_u \arrow[r, "1\otimes \psi_e"] \arrow[d, "1\otimes (C_{r_0}\circ \Phi_{u\ast})"'] & \tensor \pi_1\partial_iM_v \arrow[d, "1\otimes (C_{r_1}\circ \Phi_{v\ast})"] \\
\tensor \pi_1\partial_lN_u \arrow[r, "1\otimes \psi_e'"]                              & \tensor \pi_1\partial_jN_v                             
\end{tikzcd}
\end{equation*}

Restricting  to the $\Z$-submodules $\Z^2\cong \Z\otimes \pi_1\partial_kM_u\subseteq \widehat{\Z}\otimes \pi_1\partial_kM_u\cong \widehat{\Z}^2$ etc., we recover a commutative diagram between the original peripheral subgroups 
\begin{equation*}
\begin{tikzcd}[row sep= large, column sep=tiny]
\pi_1M_u \arrow[d, " C_{r_0}\circ \Phi_{u_\ast}"] \arrow[r, symbol=\supseteq] & \pi_1\partial_kM_u \arrow[rrrr, "\psi_e"] \arrow[d] & & & &  \pi_1\partial_iM_v \arrow[d] \arrow[r, symbol=\subseteq] & \pi_1M_v \arrow[d, " C_{r_1}\circ \Phi_{v_\ast}"'] \\
\pi_1N_u   \arrow[r, symbol=\supseteq] & \pi_1\partial_lN_u \arrow[rrrr, "\psi_e'"]                                       & & & &\pi_1\partial_jN_v     \arrow[r, symbol=\subseteq] & \pi_1N_v                                 
\end{tikzcd}
\end{equation*}
where $r_0\in \pi_1N_u$ and $r_1\in \pi_1N_v$. 

Since $\MCG(T^2)\cong \Aut(\pi_1T^2)$, this implies that $\Phi_u$ and $\Phi_v$ are compatible, up to isotopy, with the gluing maps at the JSJ-tori. Therefore, when $u\neq v$, $\Phi=(\Phi_u,\Phi_v)$ yields a homeomorphism between  $M_u\cup_{T^2}M_v$ and $ N_u\cup_{T^2}N_v$; and when $u=v$, if $\Phi_u$ is isotopic to $\Phi_v$, then $\Phi_u$ directly induces a homeomorphism between $M_u\cup_{T^2}$ and $N_u\cup_{T^2}$.
\end{proof}

\begin{remark}
In practice, one might encounter the case where $f_u$ restricting on $\partial_kM_u$ is induced by a homeomorphism $\Phi_u$ with coefficient $\lambda$, while $f_v$ restricting on $\partial_kM_v$ is induced by a homeomorphism $\Phi_v$ with   coefficient $-\lambda$. In this case, the conclusions of \autoref{inLem1'} (\ref{lemmain-2'}) do not hold.
\end{remark}

The gluing lemma can be easily generalized, with the vertices $u$ and $v$ replaced by subgraphs $\Lambda$ and  $\Upsilon$.
\begin{corollary}\label{Cor: gluing}
Suppose $M$ and $N$ are compact, orientable, irreducible 3-manifolds with empty or  incompressible toral boundary which are not closed $Sol$-manifolds.    
Let $f_\bullet$ be a congruent isomorphism between the JSJ-graphs of profinite groups $(\widehat{\mathcal{G}_M},\Gamma_M)$ and $(\widehat{\mathcal{G}_N},\Gamma_N)$.  
 Suppose $\Lambda$ and  $\Upsilon$ are either disjoint or identical connected subgraphs of $\Gamma_M$  connected by an edge $e$. Let $M_\Lambda$, $N_{\F(\Lambda)}$, $M_\Upsilon$, $N_{\F(\Upsilon)}$ be the submanifolds corresponding to these subgraphs. According to \autoref{LEM: congruent isom induce isom} and \autoref{PROP: JSJ Efficient}, the congruent isomorphism $f_\bullet$ restricting on the subgraphs induces two isomorphisms
\begin{equation*}
\begin{aligned}
&f_{\Lambda}: \widehat{\pi_1M_\Lambda}\cong \Pi_1(\widehat{\mathcal{G}_M},\Lambda)\ttt \Pi_1(\widehat{\mathcal{G}_N},\F(\Lambda))\cong \widehat{\pi_1N_{\F(\Lambda)}},\\
&f_{\Upsilon}: \widehat{\pi_1M_\Upsilon}\cong \Pi_1(\widehat{\mathcal{G}_M},\Upsilon)\ttt \Pi_1(\widehat{\mathcal{G}_N},\F(\Upsilon))\cong \widehat{\pi_1N_{\F(\Upsilon)}}.
\end{aligned}
\end{equation*}
Denote the boundary gluings at the JSJ-tori corresponding to $e$ and $\F(e)$ as 
{\small
$$
M_\Lambda\supseteq \partial_kM_\Lambda \longleftrightarrow \partial_iM_\Upsilon\subseteq M_\Upsilon\text{ and }N_{\F(\Lambda)}\supseteq \partial_lN_{\F(\Lambda)} \longleftrightarrow \partial_jN_{\F(\Upsilon)}\subseteq N_{\F(\Upsilon)}.
$$
}
    \begin{enumerate}[leftmargin=*]
        \item\label{lemmain-1''} $f_\Lambda$ is peripheral $\Zx$-regular at $\partial_kM_\Lambda$ if and only if $f_\Upsilon$ is peripheral $\Zx$-regular at $\partial_iM_\Upsilon$. In addition, when peripheral $\Zx$-regularity is satisfied, the coefficients in $\Zx$ are identical up to a $\pm$-sign. 
        \item \label{lemmain-2''}
Assume peripheral $\Zx$-regularity in (\ref{lemmain-1''}) holds. Suppose $f_\Lambda$ restricting on $\partial_kM_\Lambda$ is induced by a homeomorphism $\Phi_\Lambda: M_{\Lambda}\ttt N_{\F(\Lambda)}$ with coefficient $\lambda$, and  $f_\Upsilon$ restricting on $\partial_iM_\Upsilon$ is induced by a homeomorphism $\Phi_\Upsilon: M_{\Upsilon}\ttt N_{\F(\Upsilon)}$ with the same coefficient $\lambda$. 
 Then, 
$\Phi_\Lambda$ and $\Phi_\Upsilon$ are compatible, up to isotopy, with the gluing maps at the JSJ-tori. 
In particular, \\
$\bullet$ 
if $\Lambda$ is disjoint with $\Upsilon$, then $\Phi=(\Phi_\Lambda,\Phi_\Upsilon)$ yields a homeomorphism $M_\Lambda\cup_{T^2}M_\Upsilon\ttt N_{\F(\Lambda)}\cup_{T^2}N_{\F(\Upsilon)}$; \\
$\bullet$ if $\Lambda=\Upsilon$ and $\Phi_\Lambda=\Phi_\Upsilon$, then $\Phi_\Upsilon$ directly induces a homeomorphism $M_\Upsilon\cup_{T^2}\ttt N_{\F(\Upsilon)}\cup_{T^2}$. 
    \end{enumerate}
\end{corollary}
\begin{proof}
In fact, one can obtain a similar commutative diagram as (\ref{lem'dig1}), and the remaining proof is exactly the same as \autoref{inLem1'}.
\end{proof}

\section{Finite-volume hyperbolic 3-manifolds}\label{SEC: Hyperbolic}
By a finite-volume hyperbolic 3-manifold, we mean a compact orientable
3-manifold with empty or incompressible toral boundary, whose interior can be equipped with a complete Riemannian 
metric of constant sectional curvature $-1$ which has finite volume.


Recently, Liu \cite{Liu23} proved the following significant result.
\begin{theorem}[{\cite{Liu23}}]\label{Liu1}
The class of finite-volume hyperbolic 3-manifolds is profinitely almost rigid. 
\end{theorem}

However, the profinite automorphism group of  a finite-volume hyperbolic manifold has not been fully understood. In this section, we characterize some features of profinite isomorphisms between hyperbolic 3-manifolds when restricted on the peripheral subgroups.  
This section is mainly devoted to the proof of the following theorem.

\begin{theorem}\label{THM: Hyperbolic peripheral regular}
Let $M$ and $N$ be cusped finite-volume hyperbolic 3-manifolds, so that $\widehat{\pi_1M}\cong \widehat{\pi_1N}$. 
\begin{enumerate}[leftmargin=*]
\item\label{THM: Hyperbolic peripheral regular (1)} For any isomorphism $f:\widehat{\pi_1M}\ttt\widehat{\pi_1N}$, there exists  $\lambda\in \widehat{\Z}^{\times}$ such that $f$ is peripheral $\lambda$-regular at all boundary components of $M$. 
\item\label{THM: Hyperbolic peripheral regular (2)} For any chosen peripheral subgroups $P=\pi_1\partial_iM$ and $Q=\pi_1\partial_jN$, denote 
{\small 
$$B(M,N;P,Q)= \left\{\psi:P\ttt Q\left|\;\begin{aligned} &\text{there exists } f:\widehat{\pi_1M}\ttt\widehat{\pi_1N}\text{ and }\eta\in \widehat{\mathbb{Z}}^{\times}\text{,}\\& \text{such that }f(\overline{P})=\overline{Q}\text{ and }f|_{\overline{P}}=\eta\otimes {\psi}\end{aligned}\right.\right\}.$$}
Then $\#B(M,N;P,Q)\le 12$.
\end{enumerate}
\end{theorem}

\begin{corollary}\label{COR: Hyperbolic PAA}
Let $M$ be a cusped finite-volume hyperbolic 3-manifold, and let $\mathcal{P}$ be a collection of some  boundary components of $M$. 
Then $M$ is {\PAA} at $\mathcal{P}$ (see \autoref{indef: PAA}).
\end{corollary}

\begin{proof}
Suppose $\mathcal{P}=\{\partial_1M,\cdots,\partial_kM\}$, and choose conjugacy representatives of peripheral subgroups $P_1=\pi_1\partial_1M,\cdots ,P_k=\pi_1\partial_kM$. Then $B(M,M;P_i,P_i)$ can be viewed as a subgroup of $\Aut(P_i)\cong \mathrm{GL}_2(\Z)$. 

According to \autoref{PROP: peripheral structure}, there is a well-defined group homomorphism
\begin{equation*}
\begin{tikzcd}[row sep=0.08cm]
\mathcal{F}:\; \XX_{\Zx}(M,\mathcal{P}) \arrow[r] & \prod\limits_{i=1}^{k} {B(M,M;P_i,P_i)}/{\{\pm id\}}\\
f \arrow[r, maps to] & ([\psi_1],\cdots,[\psi_k])
\end{tikzcd}
\end{equation*}
such that for each $1\le i \le k$, $C_{g_i}\circ f|_{\overline{P_i}}=\lambda_i\otimes \psi_i$ for some $g_i\in \widehat{\pi_1M}$ and $\lambda_i\in \Zx$. Indeed \autoref{THM: Hyperbolic peripheral regular}~(\ref{THM: Hyperbolic peripheral regular (1)}) implies that the coefficients $\lambda_i$ are unified up to an ambiguity of $\pm$-sign.

Then, any $f\in \ker(\mathcal{F})$ restricting on each $\partial_iM$ is induced by the identity homeomorphism $Id_M$ with a coefficient in $\Zx$. In other words, $\ker(\mathcal{F})\subseteq \XX_{h}(M,\mathcal{P})$. 
Moreover, according to \autoref{THM: Hyperbolic peripheral regular}~(\ref{THM: Hyperbolic peripheral regular (2)}), $\# B(M,M;P_i,P_i)\le 12$. Thus, $\ker(\mathcal{F})$ has index at most $6^k$ in $\XX_{\Zx}(M,\mathcal{P})$, which is to say $\XX_{h}(M,\mathcal{P})$ has finite index in $\XX_{\Zx}(M,\mathcal{P})$. 
%
\end{proof}

\subsection{Homological $\Zx$-regularity and Thurston norm}

For a  topological space $X$, we denote $\hfree{X}$ as $H_1(X;\mathbb{Z})$ modulo its torsion subgroup. Then, there is a natural isomorphism $H^1(X;\R)\cong \Hom_{\Z}(\hfree{X},\R)$.

Suppose $M$ and $N$ are 3-manifolds,  
and $f:\widehat{\pi_1M}\ttt \widehat{\pi_1N}$ is  an isomorphism. Note that $H_1(M;\Z)=\ab{\pi_1(M)}$ and $H_1(N;\Z)= \ab{\pi_1(N)}$. According to \autoref{PROP: Induce linear}, $f$ induces an isomorphism $f_{\ast}: \tensor\hfree{M}\ttt \tensor\hfree{N}$. In particular, $b_1(M)= b_1(N)$.

When $M$ and $N$ are finite-volume hyperbolic 3-manifolds, Liu \cite{Liu23} proved two essential features of this map $f_{\ast}$, namely homological $\Zx$-regularity and the preservation of Thurston norm.

\begin{proposition}[{\cite[Theorem 6.1 $\&$ Corollary 6.3]{Liu23}}]\label{PROP: Zx regular}
Let $M$ and $N$ be finite-volume hyperbolic 3-manifolds with positive first Betti number. Suppose there is an isomorphism $f:\widehat{\pi_1M}\ttt\widehat{\pi_1N}$, and denote $f_{\ast}: \tensor \hfree{M}\ttt \tensor\hfree{N}$ as the induced isomorphism.
\begin{enumerate}[leftmargin=*]
\item\label{7.3-1} There exists an element $\lambda\in \Zx$ and an isomorphism $\Phi:\hfree{M}\ttt\hfree{N}$ such that $f_\ast=\lambda\otimes \Phi$. 
\item\label{7.3-2} 
The dual map 
$$\scalebox{.95}{$\Phi^\ast:H^1(N;\R)\cong \Hom(\hfree{N},\R)\xrightarrow{\,\cong\,} \Hom(\hfree{M},\R) \cong H^1(M;\R)$}$$
induced by $\Phi$ 
preserves the Thurston norm $ \parallel \cdot \parallel_{\mathrm{Th}}$ (see \cite{Thu86} for definition); in other words, ${{\parallel \alpha\parallel}_{\mathrm{Th}}}={{\parallel \Psi^\ast(\alpha)\parallel}_{\mathrm{Th}}}$  for all $\alpha \in H^1(N;\mathbb{R})$. 
\end{enumerate}
\end{proposition}

\subsection{Finite cover and peripheral subgroups}

In this subsection, we prove  the following proposition as a preparation for the proof of \autoref{THM: Hyperbolic peripheral regular}. The proof is largely based on the virtual specialization theorem for finite-volume hyperbolic 3-manifolds \cite{Wise}.

\begin{proposition}\label{COR: Injective into homology of boundary}
Let $M$ be a cusped finite-volume hyperbolic 3-manifold. Then, there exists a finite regular cover $M^\star$ of $M$ such that every boundary component of ${M}^\star$ is $H_1$-injective, i.e.  $H_1(\partial_jM^\star;\Z)\rightarrow H_1(M^\star;\Z)$ is injective for each boundary component $\partial_jM^\star$. 
As a consequence, the map $\pi_1(\partial_jM^\star)\to \pi_1(M^\star)\to \hfree{M^\star}$ is injective. 
\end{proposition}

This allows us to virtually  embed the peripheral subgroup into the homology group of the entire manifold, providing a method to deduce peripheral $\Zx$-regularity (\autoref{THM: Hyperbolic peripheral regular}) from the homological $\Zx$-regularity (\autoref{PROP: Zx regular}). 

Before the proof of \autoref{COR: Injective into homology of boundary}, let us introduce a notion.

A subgroup $H\le G$ is called  {\em a retract of $G$} if there exists a group homomorphism $\varphi: G\to H$ such that $\varphi|_{H}=id$, and $H\le G$ is called {\em a virtual retract of $G$} if there exists a finite-index subgroup $K\le G$, such that $K\supseteq H$ and $H$ is a retract of $K$. 

\begin{lemma}\label{VRinj}
    Suppose $H$ is a retract of $G$. Then, the homomorphism $H^{ab}=H_1(H;\Z)\to H_1(G;\Z)=G^{ab}$ is injective.
\end{lemma}
\begin{proof}
    Let $\varphi:G\to H$ be the homomorphism so that $\varphi|_H=id$. Then the conclusion follows directly from the diagram.
    \begin{equation*}
    \begin{tikzcd}
H \arrow[r, "\iota", hook] \arrow[d] \arrow[rr, "id", bend left=25] & G \arrow[r, "\varphi"] \arrow[d]   & H \arrow[d] \\
H^{ab} \arrow[r, "\iota_{\ast}"] \arrow[rr, "id"', bend right=25]   & G^{ab} \arrow[r, "\varphi_{\ast}"] & H^{ab}     
\end{tikzcd}
\end{equation*}

\end{proof}
%

The next lemma is a combination of \cite[Proposition 1.5 and Corollary 1.6]{Min19}.
\begin{lemma}[{\cite{Min19}}]\label{LEM: Virtual retract}
If a group $G$ is virtually compact special, then any finitely generated, virtually abelian subgroup of $G$ is a virtual retract of $G$. 
\end{lemma}

We are now ready to prove \autoref{COR: Injective into homology of boundary}.
\begin{proof}[Proof of \autoref{COR: Injective into homology of boundary}]
Suppose $M$ has $n$ boundary components $\partial_1M,\cdots, \partial_nM$. 
For each boundary component $\partial_iM$, fix a conjugacy representative of the peripheral subgroup $P_i=\pi_1(\partial_iM)\cong \Z\times \Z$.  \cite[Theorem 14.29]{Wise}  implies that $\pi_1M$ is virtually compact special. Thus, \autoref{LEM: Virtual retract} implies that $P_i$ is a virtual retract of $\pi_1M$, i.e. there exists a finite index subgroup $K_i\le \pi_1M$ containing $P_i$ so that $P_i$ is a retract of $K_i$. 

Let $N_i\to M$ denote the finite cover corresponding to the subgroup $K_i$. Note that $P_i\subseteq K_i$, so $\partial_iM$ lifts to $N_i$. In other words, there is a component $\widetilde{\partial_iM}$ in the  pre-image of $\partial_iM$ in $N_i$, which maps homeomorphically to $\partial_iM$ such that $P_i\subseteq K_i=\pi_1(N_i)$ is exactly the corresponding peripheral subgroup of $\widetilde{\partial_iM}$. Then, according to \autoref{VRinj}, the homomorphism $P_i=P^{ab}_i=H_1(\widetilde{\partial_iM};\Z)\to H_1(N_i;\Z)=K_i^{ab}$ induced by inclusion is injective. 

Let $H$ be a finite-index normal subgroup of $\pi_1M$ contained in $\cap_{i=1}^{n}K_i$, and let $M^\star$ be the finite regular cover of $M$ corresponding to the subgroup $H$. We show that $M^\star$ satisfies our requirement. 

In fact, the covering map $p: M^\star\to M$ factors through each $N_i$, which we denote as  $M^\star\xrightarrow{p_i} N_i\to M$. We claim that any component $S\subseteq p_i^{-1}(\widetilde{\partial_iM})$ is $H_1$-injective in $M^\star$. In fact, the covering map $T^2\cong S\to \widetilde{\partial_iM}\cong T^2$ is $\pi_1$-injective, and is hence $H_1$-injective since the fundamental group of a torus is abelian. Thus, we have the following commutative diagram, implying that $H_1(S;\Z)\to H_1(M^\star;\Z)$ is injective. 
\begin{equation*}
    \begin{tikzcd}[row sep=0.4cm, column sep=small]
H_1(S;\Z) \arrow[d, hook] \arrow[r]             & H_1(M^\star;\Z) \arrow[d, "{p_i}_\ast"] \\
H_1(\widetilde{\partial_iM};\Z) \arrow[r, hook] & H_1(N_i;\Z)                    
\end{tikzcd}
\end{equation*}

Note that any component in $p_i^{-1}(\widetilde{\partial_iM})$ is indeed a component of $p^{-1}(\partial_iM)$. In other words, for each boundary component $\partial_iM\subseteq \partial M$, there exists an $H_1$-injective component in $p^{-1}(\partial_iM)$. Recall that $M^\star$ is a regular covering of $M$, so the $H_1$-injective component in $p^{-1}(\partial_iM)$ can be carried to any component of $p^{-1}(\partial_iM)$ by a deck transformation, which is an automorphism on $M^\star$. Thus, for each $i$, all components of $p^{-1}(\partial_iM)$ are $H_1$-injective in $M^\star$. Therefore, $M^\star$ satisfies our requirement.

The injectivity of $\pi_1(\partial_jM^\star)\to \hfree{M^\star}$ follows from the following commutative diagram 
\begin{equation*}
    \begin{tikzcd}
P=\pi_1\partial_iM^\star \arrow[d, "\cong"'] \arrow[r] & \pi_1M^\star \arrow[d]    &           \\
H_1(\partial_jM^\star;\Z) \arrow[r, hook]              & H_1(M^\star;\Z) \arrow[r] & \hfree{M^\star}
\end{tikzcd}
\end{equation*}%
since $\pi_1(\partial_j\widetilde{M})\cong H_1(\partial_j\widetilde{M};\Z)\cong \Z\oplus\Z$ is torsion-free.
\end{proof}

\subsection{Peripheral $\Zx$-regularity in the hyperbolic case}

We first point out that in order to prove (\ref{THM: Hyperbolic peripheral regular (2)}) of \autoref{THM: Hyperbolic peripheral regular}, it suffices to prove a weaker conclusion.

\begin{lemma}\label{LEM: finite 12}
$\#B(M,N;P,Q)<+\infty$ implies $\#B(M,N;P,Q)\leq 12$.
\end{lemma}
\begin{proof}
This follows from two direct observations.

First, if $B(M,N;P,Q)$ is non-empty,  by choosing an element $\phi\in B(M,N;P,Q)$,  we establish a one-to-one correspondence between $B(M,N;P,Q)$ and $B(M,M;P,P)$, through  sending $\psi \in B(M,N;P,Q)$ to $\phi^{-1}\circ \psi \in B(M,M;P,P)$.

Second, $B(M,M;P,P)$ is a subgroup of $\Aut(P)\cong \mathrm{GL}(2,\Z)$.

The conclusion then follows from the fact that any finite subgroup of $\mathrm{GL}(2,\Z)$ contains at most 12 elements, see \cite{Vo68}.
\end{proof}

The proof of \autoref{THM: Hyperbolic peripheral regular} consists of two steps. We first prove a special case of the theorem as shown by the following lemma, and then relate the general case with this special case through the method of finite coverings.

\begin{lemma}\label{LEM: Hyperbolic peripheral regular special case}
\autoref{THM: Hyperbolic peripheral regular} holds if for each peripheral subgroup $\pi_1\partial_iM$, the map $\pi_1\partial_iM\to \pi_1M\to \hfree{M}$ is injective.
\end{lemma}
\begin{proof}
(\ref{THM: Hyperbolic peripheral regular (1)}) 
Let $f:\widehat{\pi_1M}\ttt\widehat{\pi_1N}$ be an isomorphism. 
%
For any peripheral subgroup $P=\pi_1\partial_iM$, \autoref{LEM: Hyp preserving peripheral}  implies that there is a peripheral subgroup $Q=\pi_1\partial_jN$ such that $f(\overline{P})=C_{g}^{-1}(\overline{Q})$ for some $g\in \widehat{\pi_1N}$. Denote $f'=C_{g}\circ f$, so that 
$f'(\overline{P})=\overline{Q}$. 

Let $f_{\ast}:\tensor\hfree{M}\ttt \tensor\hfree{N}$ be the induced isomorphism. Since $f_{\ast}$ is defined on the abelianization level, it is not affected by the conjugation of $g$. Thus, $f_\ast=f'_\ast$. In this case, $b_1(N)=b_1(M)\ge 2$, so according to \autoref{PROP: Zx regular} (\ref{7.3-1}), there exists $\lambda\in \Zx$ and 
$\Phi: \hfree{M}\ttt \hfree{N}$ 
such that $f'_\ast=f_{\ast}=\lambda\otimes  \Phi$. 

Consider the following commutative diagram.
\begin{equation*}
\begin{tikzcd}[column sep=1.5em]
\overline{P} \arrow[rd, hook] \arrow[rrrr, " f'|_{\overline{P}}","\cong"']                                    &                                                                                &                              &          & \overline{Q} \arrow[ld, hook']                                       \\
                                                                                                             & \widehat{\pi_1M} \arrow[d, two heads] \arrow[rr, "f'","\cong"']      &         & \widehat{\pi_1N} \arrow[d, two heads]  &                                                               \\
                                                                                                             & \widehat{\hfree{M}} \arrow[d, "\cong"'] \arrow[rr, "f'_\ast=f_\ast","\cong"']&  & \widehat{\hfree{N}} \arrow[d, "\cong"] &                                                               \\
                                                                                                             & \tensor \hfree{M} \arrow[rr, "f'_\ast=f_\ast","\lambda\otimes \Phi"']       & & \tensor\hfree{N}      &                                                               \\
\tensor P \arrow[uuuu, "\cong"] \arrow[ru] \arrow[rrrr, "f'|_{\overline{P}}","\cong"'] &                                                                                &                                  &      & \tensor {Q} \arrow[uuuu, "\cong"'] \arrow[lu]
\end{tikzcd}
\end{equation*}
\normalsize

Note that $\widehat{\Z}$ is torsion-free, so $\widehat{\Z}$ is  a flat $\Z$-module. Thus, the injectivity of $P\to \hfree{M}$  implies the injectivity of $\tensor P\to \tensor \hfree{M}$. 
It follows from the commutative diagram that  the map $\tensor{Q}\to \tensor \hfree{N}$ is also injective. Consequently, $Q\to \hfree{N}$ is also injective. Thus, we can canonically identify $P$ as a subgroup of $\hfree{M}$, and identify $Q$ as a subgroup of $\hfree{N}$. Since $f'_\ast=f_\ast=\lambda\otimes \Phi$, the commutative diagram implies that $\Phi(P)=Q$. Therefore,    $f'|_{\overline{P}}=\lambda\otimes \Phi|_{P}$, where $\Phi|_{P}:P\ttt Q\subseteq \hfree{N}$ is an isomorphism. Thus, $f$ is peripheral $\lambda$-regular at $\partial_iM$, and the coefficient $\lambda$ is unified for all boundary components $\partial_iM$. This proves (\ref{THM: Hyperbolic peripheral regular (1)}) of \autoref{THM: Hyperbolic peripheral regular}.

(\ref{THM: Hyperbolic peripheral regular (2)}) Let $\Phi: \hfree{M}\ttt\hfree{N}$ be as in part (1). Note that the dual isomorphism $\Phi^\ast:H^1(N;\mathbb{R})\ttt H^1(M;\mathbb{R})$ preserves the integral classes. In addition, \autoref{PROP: Zx regular} (\ref{7.3-2}) implies that $\Phi^\ast$ also preserves the Thurston norm. 
Denote 
\newsavebox{\equtempa}
\begin{lrbox}{\equtempa}
$
A^{\ast}(M,N):=\left\{\Psi^\ast: H^1(N;\mathbb{R})\ttt H^1(M;\mathbb{R})\, \left|\; \begin{gathered} 
\Psi^\ast \text{ preserves the integral lattice}\\ \text{and the Thurston norm}  
%
\end{gathered}\right.\right\}.
$
\end{lrbox}
\begin{equation*}
\scalebox{0.94}{\usebox{\equtempa}}
\end{equation*}
Then $\Phi^\ast\in A^\ast(M,N)$. 

We denote the dual of $A^\ast(M,N)$ as $$A(M,N)=\left\{\Psi: \hfree{M}\ttt\hfree{N} \mid \Psi^{\ast}\in A^{\ast}(M,N)\right\}.$$ Then, $\Phi\in A(M,N)$. 

There is an apparent bijection between $A(M,N)$ and $A^{\ast}(M,N)$. In addition, $\Phi$ establishes a bijection between $A^\ast(M,M)$ and $A^\ast(M,N)$, by sending $\Psi^\ast\in A^\ast(M,M)$ to $\Psi^\ast\circ \Phi^\ast \in A^\ast(M,N)$. Thus, $\#A(M,N)=\#A^\ast(M,N)=\#A^\ast(M,M)$.

Note that 
both $B(M,N;P,Q)$ and $A(M,N)$ are invariant under scalar multiplication by $\pm 1$. 
Therefore, part (1) shows  that any element $\psi\in B(M,N;P,Q)$ is a restriction of some element $\Phi\in A(M,N)$ on $P$, i.e. $\psi=\Phi|_{P}$. Thus, $$\#B(M,N;P,Q)\le \# A(M,N)=\#A^{\ast}(M,M).$$

Since $M$ is a finite-volume hyperbolic 3-manifold, the Thurston norm is indeed a norm on the real vector space $H^1(M;\R)$,  according to \cite[Theorem 2.11]{Kap01}. Furthermore, according to \cite[Theorem 2]{Thu86}, the unit ball $\mathcal{B}_{\mathrm{Th}}(M)$ of the Thurston norm for $M$ is a convex polyhedron. Indeed,  $\mathcal{B}_{\mathrm{Th}}(M)$ is the convex hull of finitely many points in $H^1(M;\R)$ symmetric about the origin. Note that there are only finitely many linear automorphisms of a finite-dimensional real vector space preserving a finite, compact, symmetric,  codimension-0 polyhedron. Since any element in $A^{\ast}(M,M)$ preserves $\mathcal{B}_{\mathrm{Th}}(M)$, it follows that $\#A^{\ast}(M,M)<\infty$.

Thus, $\#B(M,N;P,Q)<\infty$. Combining with \autoref{LEM: finite 12}, this proves (\ref{THM: Hyperbolic peripheral regular (2)}) of  \autoref{THM: Hyperbolic peripheral regular}.
\end{proof}

Before completing the proof of \autoref{THM: Hyperbolic peripheral regular}, we need a last group theoretical preparation. 

\begin{lemma}\label{LEM: Linear map on finite index subgroup}
Let $P$ and $Q$ be finitely generated free abelian groups. Suppose $P_0$ is a finite-index subgroup of $P$, and identify $\overline{P_0}$ with $\widehat{P_0}$.
\begin{enumerate}[leftmargin=*]
\item\label{7.8-1} For every homomorphism $\varphi_0: P_0\to Q$, there exists at most one homomorphism $\varphi: P\to Q$ such that $\varphi|_{P_0}=\varphi_0$.
\item\label{7.8-2} For every continuous homomorphism $\Phi_0:\widehat{P_0}\to \widehat{Q}$, there exists at most one continuous homomorphism $\Phi:\widehat{P}\to \widehat{Q}$ such that $\Phi|_{\overline{P_0}}=\Phi_0$.
\item\label{7.8-3} Based on (\ref{7.8-2}), suppose there exists $\lambda\in \Zx$ and $\varphi_0\in \Hom_{\Z}(P_0,Q)$ such that $\Phi_0=\lambda\otimes {\varphi_0}$. If the continuous homomorphism $\Phi$ exists, then $\Phi$ can be decomposed as $\lambda\otimes {\varphi}$, where $\varphi\in \Hom_{\Z}(P,Q)$ and $\varphi|_{P_0}=\varphi_0$.
\end{enumerate}
\end{lemma}

\begin{proof}
(\ref{7.8-1}) This follows from a general conclusion for torsion-free abelian groups. In fact, since $P_0$ is a finite-index subgroup, there exists a positive integer $n$ such that $n\cdot P\subseteq P_0$. Suppose $\varphi$ and $\varphi'$ are homomorphisms so that $\varphi|_{P_0}=\varphi_0=\varphi'|_{P_0}$. Then, for any element $x\in P$, $n\cdot x\in P_0$ and $n\cdot \varphi(x)=\varphi(n\cdot x)= \varphi'(n\cdot x)=n\cdot \varphi'(x)$. Thus, $n\cdot(\varphi(x)-\varphi'(x))=0$. Since $Q$ is torsion-free, we derive that $\varphi(x)=\varphi'(x)$ for all elements $x\in P$.

(\ref{7.8-2}) Note that the profinite completion of a finitely generated free abelian group is a finitely generated free $\widehat{\Z}$-module, while $\widehat{\Z}\cong \prod \Z_p$ is torsion-free. Thus, $\widehat{P}$ and $\widehat{Q}$ are torsion-free. Besides, $\widehat{P_0}$ is a finite-index subgroup of $\widehat{P}$ by \autoref{THM: correspondence of subgroup}. Hence, (\ref{7.8-2}) follows from the same proof as (\ref{7.8-1}). 

(\ref{7.8-3}) Now we suppose $\Phi_0=\lambda\otimes {\varphi_0}$. 
We identify $P$ and $Q$ with their images in $\widehat{P}$ and $\widehat{Q}$ through the canonical homomorphisms. In order to show that the  homomorphism $\Phi$ can be decomposed as $\Phi=\lambda\otimes{\varphi}$, it suffices to prove that $\varphi(P)\subseteq \lambda \cdot Q$. We still denote $n$ as the positive integer so that $n\cdot P\subseteq P_0$. Then for any $x\in P$, $n\cdot x\in P_0\subseteq \overline{P_0}$ and $n\cdot \Phi(x)=\Phi_0(nx)\in \lambda Q$. Thus, $n\cdot \Phi(x)\in \lambda \cdot Q\cap n\cdot \widehat{Q}=\lambda(Q\cap n\cdot \widehat{Q})$, where the last equality holds since $\lambda\in\Zx$. By \cite[Lemma 2.2]{Wil18}, $Q\cap n\cdot \widehat{Q}=n\cdot Q$, so $n\cdot \Phi(x)\in n\lambda\cdot  Q$. Again, since $\widehat{Q}$ is torsion-free, we derive that $\Phi(x)\in \lambda Q$, finishing the proof of (\ref{7.8-3}).
\end{proof}

We are now ready to prove the general case of \autoref{THM: Hyperbolic peripheral regular}.
\begin{proof}[Proof of \autoref{THM: Hyperbolic peripheral regular}]

(\ref{THM: Hyperbolic peripheral regular (1)}) 
According to \autoref{COR: Injective into homology of boundary}, there exists a $m$-fold regular cover $M^\star\to M$ so that the map $\pi_1\partial_iM^\star\to \hfree{M^\star}$ is injective for each peripheral subgroup $\pi_1\partial_iM^\star$. We identify $\pi_1M^\star$ as an index-$m$ normal subgroup of $\pi_1M$. Then $\overline{\pi_1M^\star}$ is an open normal subgroup of $\widehat{\pi_1M}$ with index $m$, according to \autoref{THM: correspondence of subgroup}.

For any isomorphism $f:\widehat{\pi_1M}\ttt\widehat{\pi_1N}$, $f$ sends $\overline{\pi_1M^\star}$ to an index-$m$ open normal subgroup of $\widehat{\pi_1N}$. Again, by \autoref{THM: correspondence of subgroup}, $f (\overline{\pi_1M^\star})$ is the closure of an index-$m$ normal subgroup of $\pi_1N$. Since $\pi_1N$ is finitely generated, $\pi_1N$ contains only finitely many index-$m$ normal subgroups. Each of them corresponds to a $m$-fold regular cover of $N$, denoted by $N^\star_1,\cdots, N^\star_n$, where each $\pi_1N^\star_k$ is an index-$m$ normal subgroup of $\pi_1N$.
%
Thus, there exists $1\le k\le n$ such that $f(\overline{\pi_1M^\star})=\overline{\pi_1N^\star_k}$, and we denote the  restriction of $f$ as   $f_0:\widehat{\pi_1M^\star}\ttt\widehat{\pi_1N^\star_k}$.

For each peripheral subgroup $P=\pi_1\partial_iM$, $P_0:=P\cap \pi_1M^\star$ can be identified as a peripheral subgroup $\pi_1\widetilde{\partial_iM}\subseteq \pi_1M^\star$, where $\widetilde{\partial_i{M}}$ is one component of the pre-image of $\partial_iM$ in $M^\ast$. In particular, $P_0$ is a finite-index subgroup of $P$.

Note that $M^\star$ and $N^\star_k$ are again finite-volume hyperbolic 3-manifolds. According to \autoref{LEM: Hyp preserving peripheral}, $f_0$ respects the peripheral structure, so there exists a peripheral subgroup $Q_0=\pi_1 \partial_l N^\star_k $ and an element $g\in \widehat{\pi_1N_k^\star}$ such that $f_0(\overline{P_0})=C_{g}^{-1}(\overline{Q_0})$. Let $\partial _j N$ be the image of $\partial _j {N_k^\star}$ in $N$, and let  $Q=\pi_1\partial_j N$ be the conjugacy representative of the peripheral subgroup which contains $Q_0$. Then, $f(\overline{P})$ contains $f_0(\overline{P_0})$ and intersects non-trivially with $C_g^{-1}(\overline{Q})$. Since $f$ also respects the peripheral structure by  \autoref{LEM: Hyp preserving peripheral}, $f(\overline{P})$ is also a conjugate of the closure of  a peripheral subgroup in $N$. Then the malnormality of peripheral subgroups (\autoref{LEM: Malnormal}) implies that $f(\overline{P})$ is exactly $C_g^{-1}(\overline{Q})$.

Denote $f'=C_{g}\circ f: \widehat{\pi_1M}\ttt\widehat{\pi_1N}$. Note that $\widehat{\pi_1N^\star_k}$ is a normal subgroup of $\widehat{\pi_1N}$, so we still have $f'(\widehat{\pi_1M^\star})=\widehat{\pi_1N_k^\star}$. In addition, $f'(\overline{P_0})=\overline{Q_0}$, and $f'(\overline{P})=\overline{Q}$. 

Recall that all boundary components of $M^\star$ are $H_1$-injective. According to \autoref{LEM: Hyperbolic peripheral regular special case}, there exists a unified $\lambda\in \Zx$ such that 
$f_0$ is peripheral $\lambda$-regular at all boundary components of $M^\star$. 
In other words, $f'|_{\overline{P_0}}=C_{g}\circ f_0|_{\overline{P_0}}$ can be decomposed as $\lambda\otimes \psi_0$, where $\psi_0 \in B(M^\star,N_k^\star;P_0,Q_0)\subseteq \mathrm{Isom}_{\Z} (P_0,Q_0)$. It then follows from \autoref{LEM: Linear map on finite index subgroup} (\ref{7.8-3}) that ${f'}|_{\overline{P}}=\lambda\otimes \psi$, where $\psi:P\to Q$ satisfies $\psi|_{P_0}=\psi_0$, and $\psi$ must be an isomorphism since $f'(\overline{P})=\overline{Q}$. This implies that $f$ is peripheral $\lambda$-regular at each $\partial_iM$, which finishes the proof of \autoref{THM: Hyperbolic peripheral regular} (\ref{THM: Hyperbolic peripheral regular (1)}).

(\ref{THM: Hyperbolic peripheral regular (2)}) 
In addition, given any peripheral subgroups $P=\pi_1\partial_iM$ and $Q=\pi_1\partial_jN$, for each $1\le k\le n$, denote 
\newsavebox{\equtempb}
\begin{lrbox}{\equtempb}
{ 
$B_k(M,N;P,Q)= \left\{\psi:P\ttt Q\left|\,\begin{gathered} \text{there exist } f:\widehat{\pi_1M}\ttt\widehat{\pi_1N}\text{ and }\eta\in \widehat{\mathbb{Z}}^{\times},\\ \text{such that }  f\big(\overline{\pi_1M^\star}\big)=\overline{\pi_1N_k^\star},\\ f(\overline{P})=\overline{Q}\text{, and }f|_{\overline{P}}=\eta\otimes {\psi}\end{gathered}\right.\right\}.$}
\end{lrbox}
\begin{equation*}
\scalebox{0.98}{\usebox{\equtempb}}
\end{equation*}

By the reasoning in part (1), $B(M,N;P,Q)=\bigcup_{k=1}^{n}B_k(M,N;P,Q)$. Denote $P_0=P\cap \pi_1M^\star$, and  $Q_0^{k}=Q\cap \pi_1N_k^\star$ for each $1\le k\le n$, which are finite-index subgroups of $P$ and $Q$ respectively.  $P_0$ and $Q_0^k$ are in fact peripheral subgroups in $\pi_1M^{\star}$ and $\pi_1 N^\star_k$.   In addition, any element $\psi\in B_k(M,N;P,Q)$ restricts to an element $\psi_0\in B(M^\star,{N_k^\star};P_0,Q_0^{k})$. Thus, by \autoref{LEM: Linear map on finite index subgroup} (\ref{7.8-1}), $\#B_k(M,N;P,Q)\le \#B(M^\star,{N_k^\star};P_0,Q_0^{k})$.

As a consequence,
\newsavebox{\equtempc}
\begin{lrbox}{\equtempc}
$
\#B(M,N;P,Q)\le \sum\limits_{k=1}^{n}\#B_k(M,N;P;Q)
\le \sum\limits_{k=1}^{n} \#B(M^\star,{N_k^\star};P_0,Q_0^{k})\le12n<\infty,
$
\end{lrbox}
\begin{equation*}
\scalebox{0.94}{\usebox{\equtempc}}
\end{equation*}
where the final inequality follows from \autoref{LEM: Hyperbolic peripheral regular special case}.  
The proof of \autoref{THM: Hyperbolic peripheral regular} (\ref{THM: Hyperbolic peripheral regular (2)}) then follows from \autoref{LEM: finite 12}. 
\end{proof}

\section{Seifert fibered spaces}\label{SEC: Seifert}

In our context, all Seifert fibered 3-manifolds are compact and orientable.

\subsection{Profinite properties of Seifert fibered spaces}\label{subSEC: SFS}

Recall that Seifert fibered 3-manifolds with non-empty incompressible boundary consist of three cases, following the notation of \cite{Wil17}:
\begin{itemize}[leftmargin=*]
\item the thickened torus $S^1\times S^1\times I$; 
\item the minor Seifert manifold, namely the twisted $I$-bundle over Klein bottle $S^1\ttimes S^1\ttimes I$;
\item major Seifert manifolds, i.e. those  admitting a Seifert fibration   over a base orbifold $\O$ with negative Euler characteristic and non-empty boundary.
\end{itemize}

According to \cite[Lemma 3.2]{Sco83}, from a Seifert fibration $M\to \O$ one derives a short exact sequence $$1\tto K \tto \pi_1M\tto \pi_1\O\tto 1,$$ where $K$ denotes the cyclic subgroup generated by a regular fiber, and $\pi_1\O$ denotes the orbifold fundamental group. When $M$ has non-empty boundary, $K$ is infinite cyclic. In particular, when $M$ is a major Seifert fibered manifold, $M$ admits a unique Seifert fibration and $K$ is the unique maximal virtually central subgroup of $\pi_1M$.

In addition, the LERFness of $\pi_1M$ \cite{Sco78} together with \autoref{PROP: Completion Exact} and \autoref{PROP: Left exact} 
implies a short exact sequence
\begin{equation}\label{EQU: Seifert exact sequence}
\begin{tikzcd}
1 \arrow[r] & \widehat{K}\cong \overline{K} \arrow[r] & \widehat{\pi_1M} \arrow[r] & \widehat{\pi_1\O} \arrow[r] & 1.
\end{tikzcd}
\end{equation}

\begin{proposition}[\cite{Wil17}]\label{PROP: Virtually central}
Let $M$ be a Seifert fibered manifold with non-empty incompressible boundary.
Then, $\widehat{\pi_1M}$ determines whether $M$ is major, minor, or the thickened torus. In addition, when $M$ is a  major Seifert manifold,
\begin{enumerate}[leftmargin=*]
\item\label{PROP 8.1-1}  $\overline K$ is the unique maximal virtually central normal procyclic subgroup of $\widehat{\pi_1M}$;
\item  $\widehat{\pi_1M}$ determines whether the base orbifold $\O$ is orientable;
\item\label{PROP 8.1-3} $\widehat{\pi_1M}$ determines the isomorphism type of the orbifold fundamental group $\pi_1\O$.
\end{enumerate}
\end{proposition}
\begin{proof}
Among Seifert fibered manifolds with non-empty incompressible boundary, a manifold $M$ is major, minor, or the thickened torus if and only if  the fundamental group $\pi_1M$ is not virtually abelian, virtually abelian but non-abelian, or abelian, which can be detected from the profinite completion.

When $M$ is a  major Seifert manifold, (\ref{PROP 8.1-1}) follows from \cite[Theorem 5.5]{Wil17}. In addition, the base orbifold $\O$ is orientable if and only if $K$ is central in $\pi_1M$, which is equivalent to $\overline {K}$ being central in $\widehat{\pi_1M}$. Thus, by (\ref{PROP 8.1-1}), $\widehat{\pi_1M}$ determines whether $\O$ is orientable. Moreover, (\ref{PROP 8.1-1}) shows that $\widehat{\pi_1M}$ determines the isomorphism type of $\widehat{\pi_1\O}\cong \widehat{\pi_1M}/\overline{K}$. Note that $\chi(\O)<0$, so $\pi_1\O$ is a finitely-generated Fuchsian group. \cite[Theorem 1.1]{BCR16} shows that the isomorphism type of $\widehat{\pi_1\O}$ determines the isomorphism type of $\pi_1\O$, which finishes the proof of (\ref{PROP 8.1-3}).
\end{proof}

\begin{corollary}\label{COR: Seifert rigid}
    The class of groups $$\mathcal{S}=\{ \pi_1M \mid M \text{: orientable Seifert fibered space with }  \text{non-empty }   \text{boundary}\}$$ is (relatively) profinite rigid. In other words, for any two  Seifert fibered spaces $M$ and $N$ with non-empty boundary, $\widehat{\pi_1M}\cong \widehat{\pi_1N}$ if and only if $\pi_1M\cong \pi_1N$.
\end{corollary}
\begin{proof}
    The only Seifert fibered space with compressible boundary is the solid torus, whereas $\pi_1(D^2\times S^1)\cong\Z$ is profinitely rigid among all finitely-generated residually finite groups. Also note that $\pi_1(S^1\times S^1\times I)\cong \Z\times \Z$ and  $\pi_1(S^1\ttimes S^1\ttimes I)\cong \Z\rtimes \Z$ are both profinitely rigid in $\mathcal{S}$ by \autoref{PROP: Virtually central}. Thus, it suffices to restrict ourselves to major Seifert manifolds. 

    We claim that for a major Seifert manifold $M$, the isomorphism type of $\pi_1M$ is uniquely determined by the following two data: 
    \begin{enumerate}[leftmargin=*]
        \item\label{s-1} the isomorphism type of the orbifold fundamental group  $\pi_1\O$;
        \item\label{s-2} whether $\O$ is orientable.
    \end{enumerate}
    Part of this result was proven in \cite[Lemma 4.2]{Hem14}.

    In fact, $\pi_1\O$ is virtually free, and its isomorphism type determines the orbifold Euler characteristic $\chi(\O)$. To be explicit, if $\pi_1\O$ contains an index-$m$ subgroup isomorphic to $F_r$, then $\chi(\O)=\frac{1-r}{m}$. In addition, $\pi_1\O$ determines the singularities of $\O$, each corresponding to a conjugacy class of maximal  torsion subgroups in $\pi_1\O $, and the index of the singular point is exactly the order of the (finite cyclic) torsion subgroup. We denote the index of singularities by $p_1,\cdots, p_n$, and denote the underlying surface of $\O$ as $F$. Then, $\chi(F)=\chi(O)+\sum(1-\frac{1}{p_i})$ can also be determined from these data.

    As shown in \autoref{fig: orbifold}, the orbifold $\O$ can be obtained from $F$ by gluing disks $D_i$ with exactly one singularity of index $p_i$ to a boundary component of $F$ along the arcs $\alpha_i$. Consequently, $M$ can be obtained from the orientable $S^1$-bundle over $F$, by gluing solid tori $V_i$ to a boundary component $S^1\times \partial_iF$ along annuli $S^1\times \alpha_i$, so that the fiber $S^1\times\{\ast\}$ intersects $p_i$ times the meridian of $V_i$.

    \begin{figure}[ht]
        \centering
\includegraphics[width=2.8in, alt={A surface with no singular points, and several disks glued to one of its boundary components along disjoint arcs, each containing one singular point.}]{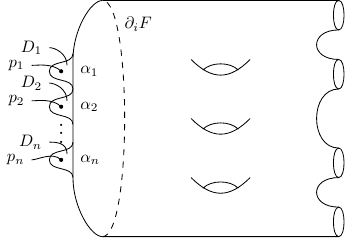}
        \caption{The orbifold $\O$}
        \label{fig: orbifold}
    \end{figure}

    Note that the fundamental group of the orientable $S^1$-bundle over $F$ has the presentation: $\left\langle c_1,\cdots,c_k,h\mid c_jhc_j^{-1}=h^\epsilon \right\rangle$, where $k=1-\chi(F)$, and $\epsilon=1$ if $F$ is orientable, while $\epsilon=-1$ if $F$ is non-orientable. Here, $ h$ represents a regular fiber $S^1\times \{\ast\}$. Then, by Seifert--van Kampen theorem, we obtain that $\pi_1(M)=\left\langle c_1,\cdots,c_k,h,v_1,\cdots,v_n\mid c_jhc_j^{-1}=h^\epsilon, h=v_i^{p_i} \right\rangle$, where each $v_i$ is a conjugacy representative of the core curve of $V_i$. Hence, the isomorphism type of $\pi_1M$ is determined from the two data  (\ref{s-1}) and (\ref{s-2}).
    
    Since these two data can be encoded from $\widehat{\pi_1M}$ as shown by \autoref{PROP: Virtually central}, we derive that  $\widehat{\pi_1M}$ determines the unique isomorphism type of $\pi_1M$.
\end{proof}

\begin{remark}
    Also note that the class $\mathcal{S}$ can be profinitely distinguished from the fundamental groups of all compact, orientable 3-manifolds, according to \autoref{COR: Determine boundary} which determines irreducibility and the toral boundary, together with \autoref{THM: Profinite Isom up to Conj} which determines the JSJ-decomposition. Thus, any $G\in \mathcal{S}$ is   profinitely rigid among the fundamental groups of all compact, orientable 3-manifolds.
\end{remark}

Note that in the bounded case, the group theoretic profinite rigidity does not imply rigidity in the class of manifolds as defined in \autoref{indef: almost rigidity}, since there exist non-homeomorphic Seifert manifolds with isomorphic fundamental groups. Nevertheless, we  can still derive profinite almost rigidity in the manifold version.

\begin{proposition}\label{PROP: Seifert almost rigid}
The class of orientable Seifert fibered spaces is profinitely almost rigid.
\end{proposition}
\begin{proof}
The only orientable Seifert fibered space with compressible boundary is the solid torus, which has the same fundamental group as $S^2\times S^1$. 
Except for this case, \autoref{COR: Determine boundary} implies that 
the profinite completion detects whether the manifold is closed. 
The profinite almost rigidity of closed Seifert fibered 3-manifolds follows from \cite[Theorem 5.1]{Wil17}; and the profinite almost rigidity of Seifert fibered 3-manifolds with incompressible boundary follows from \autoref{COR: Seifert rigid}, since there are only finitely many Seifert fibered spaces with isomorphic fundamental groups, see \autoref{Joh}.
\end{proof}

\subsection{Profinite isomorphism between Seifert manifolds}

We start with the simple case of minor Seifert manifolds.
The fundamental group of a minor Seifert manifold can be presented as 
\begin{equation}\label{EQU: Minor fundamental group}
\pi_1M=\left\langle a,b\mid bab^{-1}=a^{-1}\right\rangle,
\end{equation}
where $\langle a\rangle$ represents the fiber subgroup when the manifold is viewed as a Seifert fibration over the M\"obius band, while $\langle b^2\rangle$ represents the fiber subgroup when it is viewed as  a Seifert fibration over a disk with two singular points of index 2. Then. $\widehat{\pi_1M}=\overline{\left\langle a\right\rangle}\rtimes \overline{\left\langle b\right\rangle}\cong \widehat{\Z}\rtimes\widehat{\Z}$. The unique peripheral subgroup in $\widehat{\pi_1M}$ is given by $\overline{\left\langle a\right\rangle }\times \overline{\left\langle b^2\right\rangle}$.

\begin{lemma}\label{LEM: Minor completion}
Suppose $M$ and $N$ are minor Seifert manifolds, whose fundamental groups are 
presented as 
\begin{equation*}
\begin{gathered}
\pi_1M=\left\langle a,b\mid bab^{-1}=a^{-1}\right\rangle,\;\;
\pi_1N=\left\langle \alpha,\beta \mid \beta \alpha \beta^{-1}=\alpha^{-1}\right\rangle.
\end{gathered}
\end{equation*}
For any isomorphism $f:\widehat{\pi_1M}\ttt \widehat{\pi_1N}$, there exist  $\lambda,\mu\in \Zx$ and $\theta\in \widehat{\Z}$ such that 
\begin{equation*}
f(a)=\alpha^\mu,\;f(b)=\alpha^\theta \beta^\lambda,\; f(b^2)=\beta^{2\lambda}.
\end{equation*}
In particular, $f$ always respects the peripheral structure.
\end{lemma} 
\begin{proof}
It is easy to verify that the center of $\widehat{\pi_1M}$ is $\overline{\left\langle b^2\right\rangle}$, and similarly, the center of $\widehat{\pi_1N}$ is $\overline{\left\langle \beta^2\right\rangle}$. Thus, $f(\overline{\left\langle b^2\right\rangle})=\overline{\left\langle \beta^2\right\rangle}$ so there exists $\lambda \in \Zx$ such that $f(b^2)=\beta^{2\lambda}$. 

Let $p: \widehat{\pi_1N} \to \widehat{\Z}\cong \widehat{\left\langle \beta\right\rangle}$ denote the homomorphism defined by quotienting $\overline{\left\langle \alpha\right\rangle}$. Then $p(f(b))^2=p(f(b^2))=\beta^{2\lambda}$. Since $\widehat{\Z}$ is torsion-free, $p(f(b))=\beta^\lambda$, which implies that $f(b)=\alpha^\theta \beta^\lambda$ for some $\theta\in \widehat{\Z}$. 

Since $\widehat{\Z}$ is abelian, we have $p(f(a))^2=p(f(bab^{-1}a))=p(f(\mathbf{1}_{\widehat{\pi_1M}}))=\beta^0$. As $\widehat{\Z}$ is torsion-free, it follows that $p(f(a))=\beta^0$, i.e. $f(a)\in \overline{\left\langle\alpha\right\rangle } $, or equivalently, $f(\overline{\left\langle a\right\rangle})\subseteq \overline{\left\langle\alpha\right\rangle}$. The symmetric proof  for $f^{-1}$ implies that 
$f^{-1}(\overline{\left\langle\alpha\right\rangle})\subseteq \overline{\left\langle a\right\rangle}$. As a consequence, 
$f(\overline{\left\langle a\right\rangle})= \overline{\left\langle\alpha\right\rangle}$, so $f(a)=\alpha^\mu $ for some $\mu\in \Zx$.

In particular, $f(\overline{\left\langle a\right\rangle }\times \overline{\left\langle b^2\right\rangle})=\overline{\left\langle \alpha \right\rangle }\times \overline{\left\langle \beta^2\right\rangle}$. In other words, $f$ respects the peripheral structure.
\end{proof}

In order to avoid ambiguity in assigning the fiber subgroup of the minor Seifert manifold, we include the following convention
\begin{convention}\label{convention: minor}
    From now on, we always view the minor Seifert manifold as a Seifert fibration over the disk with two singular points of index 2.  In this case, the fiber subgroup $K_M\le \pi_1M$ is generated by $b^2$, following the presentation (\ref{EQU: Minor fundamental group}).
\end{convention}

\begin{lemma}\label{COR: Seifert Split}
Let $M$ and $N$ be   major or minor Seifert  manifolds such that $\widehat{\pi_1M}\cong \widehat{\pi_1N}$. Then, any isomorphism $f:\widehat{\pi_1M} \ttt \widehat{\pi_1N}$ splits as: 
\begin{equation*}
\begin{tikzcd}
1 \arrow[r] & \widehat{K_M}\cong \widehat{\Z} \arrow[r] \arrow[d, "f_\natural","\cong"'] & \widehat{\pi_1M} \arrow[r] \arrow[d, "f","\cong"'] & \widehat{\pi_1\O_M} \arrow[r] \arrow[d, "f_\flat","\cong"'] & 1 \\
1 \arrow[r] & \widehat{K_N}\cong \widehat{\Z} \arrow[r]                         & \widehat{\pi_1N} \arrow[r]                & \widehat{\pi_1\O_N} \arrow[r]                      & 1
\end{tikzcd}
\end{equation*}
\end{lemma}
\begin{proof}
    In view of (\ref{EQU: Seifert exact sequence}), it suffices to show that $f(\overline{K_M})=\overline{K_N}$. 
    By \autoref{PROP: Virtually central}, either $M$ and $N$ are both minor, or they are both major. For the minor case, this follows from \autoref{LEM: Minor completion}; while in the major case, this follows from \autoref{PROP: Virtually central} (\ref{PROP 8.1-1}).
\end{proof}
\begin{example}\label{EX: Boundary map}
Let $M$ and $N$ be major or minor Seifert manifolds, and let $f:\widehat{\pi_1M}\ttt \widehat{\pi_1N}$ be an isomorphism. Suppose for some boundary components $\partial_iM$ and $\partial_jN$, $f$ sends $\overline{\pi_1\partial_iM}$ to a conjugate of $\overline{\pi_1\partial_jN}$. We may choose a free basis $\{h_M,\epsilon_i\}$ for $\pi_1\partial_iM\cong \Z\times\Z$, so that $h_M$ is a generator for the fiber subgroup $K_M$. Similarly, choose a free basis $\{h_N,\epsilon_j'\}$ for $\pi_1\partial_jN$ so that $h_N$ generates $K_N$. Since $f(\overline{K_M})=\overline{K_N}$, up to composing with a conjugation by $g\in \widehat{\pi_1N}$, $C_g\circ f (h_M)=h_N^\lambda$ for some $\lambda\in \Zx$, and the map $C_g\circ f|_{\overline{\pi_1\partial_iM}}:\overline{\pi_1\partial_iM}\to\overline{\pi_1\partial_jN}$ can be represented by an upper triangle matrix:
\begin{equation*}
C_g\circ f \begin{pmatrix} h_M&\epsilon_i\end{pmatrix}= \begin{pmatrix} h_N&\epsilon_j'\end{pmatrix} \begin{pmatrix} \lambda & \rho\\ 0 &\mu \end{pmatrix},\quad \begin{pmatrix} \lambda & \rho\\ 0 &\mu \end{pmatrix}\in \mathrm{GL}(2,\widehat{\Z}).
\end{equation*}

Indeed, the entries $\lambda$ and $\mu$ are uniquely determined by $f$ up to $\pm$-signs, that is to say they are not affected by different choices of bases. 
In particular, when $M$ and $N$ are minor Seifert manifolds, if we choose $(h,\epsilon)=(b^2,a)$ following (\ref{EQU: Minor fundamental group}), then $C_g\circ f$ is represented by a diagonal matrix $\left(\begin{smallmatrix}
     \lambda & 0\\ 0 & \mu 
\end{smallmatrix}\right)$.
\end{example}

\subsection{Profinite isomorphism respecting peripheral structure}\label{SEC: Seifert peripheral}

For a Seifert manifold $M$ with boundary, $\widehat{\pi_1M}$ marked with the peripheral structure $\{\overline{\pi_1\partial_iM}\}$ entails extra information beyond \autoref{PROP: Virtually central}. For instance, this additionally determines the number of boundary components of the base orbifold, which determines the unique isomorphism type of the base orbifold. Following the setting of \autoref{COR: Seifert Split}, if $f:\widehat{\pi_1M}\to \widehat{\pi_1N}$ respects the peripheral structure, then so does $f_\flat$. In addition, we have the following characterization for $f_\flat$.

\begin{definition}[{\cite[Definition 2.6]{Wil18}}]\label{DEF: Exotic Automorphism}
Let $\O$ be an orientable 2-orbifold with boundary, with fundamental group
$$\pi_1\O= \left< a_1,\ldots, a_r, e_1, \ldots, e_s, u_1, v_1, \ldots, u_g,v_g\mid a_j^{p_j}\right>,$$
where the boundary components of $\O$ are represented by the conjugacy classes of the elements $e_1, \ldots, e_s$ together with
$$ e_0 = \left(a_1 \cdots a_r e_1\cdots e_s [u_1, v_1]\cdots [u_g,v_g]\right)^{-1} .$$
Then, an {\em exotic automorphism of $\O$ of type $\mu\in \Zx$} is an automorphism $\psi\colon \widehat{\pi_1\O}\to\widehat{\pi_1\O}$ such that $\psi(a_j) \sim a_j^\mu$ and $\psi(e_i)\sim e_i^\mu$ for all $1\le j\le r$ and $0\le i\le s$, where $\sim$ denotes conjugacy in $\widehat{\pi_1\O}$.

Similarly, let $\O'$ be a non-orientable 2-orbifold with boundary, with fundamental group 
$$\pi_1\O'= \left< a_1,\ldots, a_r, e_1, \ldots, e_s, u_1,\ldots, u_g\mid a_j^{p_j}\right>,$$
where the boundary components of $\O'$ are represented by the conjugacy classes of the elements $e_1, \ldots, e_s$ together with
$$ e_0 = \left( a_1 \cdots a_r e_1\cdots e_su_1^2\cdots u_g^2\right)^{-1}. $$
Let $o\colon \widehat{\pi_1\O'}\to \{\pm 1\}$ be the orientation homomorphism of $\O'$. 
Then, an {\em exotic automorphism of $\O'$ of type $\mu\in\Zx$\iffalse with signs $\sigma_0, \ldots, \sigma_s$\fi} is an automorphism $\psi\colon\widehat{\pi_1\O'}\to\widehat{\pi_1\O'}$ such that $\psi(a_j) \sim a_j^\mu$ and $\psi(e_i)= g_i\cdot e_i^{o(g_i)\mu}\cdot g_i^{-1}$, where $g_i\in \widehat{\pi_1\O'}$.
\end{definition}

\begin{proposition}[{\cite{Wil17, Wil18}}]
\label{PROP: Hempel pair}
Suppose $M$ and $N$ are major or minor Seifert  manifolds, and $f:\widehat{\pi_1M} \ttt \widehat{\pi_1N}$ is an isomorphism respecting the peripheral structure.
Then:
\begin{enumerate}[leftmargin=*]
\item The base orbifolds $\O_M$ and $\O_N$ are isomorphic.
\item\label{8.5-2} 
By choosing an appropriate isomorphism of orbifolds, $\O_M$ and $\O_N$ can be identified as the same orbifold $\O$, so that the isomorphism $f_\flat$ given by \autoref{COR: Seifert Split} is an exotic automorphism of $\O$ of type $\mu\in \Zx$.
\item\label{8.5-3} By choosing generators $h_M\in K_M$ and $h_N\in K_N$, the isomorphism $f_\natural$ given by \autoref{COR: Seifert Split} sends $h_M$ to $h_N^\lambda$ for some $\lambda\in \Zx$.
\item\label{8.5-4} Denote $\kappa=\mu^{-1}\lambda$. If $M$ has Seifert invariants $(p_j,q_j)$, then by identifying the base orbifolds as in (\ref{8.5-2}) and the fiber subgroups as in (\ref{8.5-3}), $N$ has Seifert invariants $(p_j,q_j')$ where $q_j'\equiv \kappa q_j\pmod {p_j}$.
\end{enumerate}
\end{proposition}
\begin{proof}
    When $M$ and $N$ are major Seifert manifolds, this follows from {\cite[Theorem 2.7]{Wil18}} and {\cite[Theorem 5.8]{Wil17}}. 
    
    When $M$ and $N$ are minor Seifert manifolds, this follows from \autoref{LEM: Minor completion}. In fact, following the group presentation (\ref{EQU: Minor fundamental group}), the boundary component of the base orbifold $D^2_{(2,2)}$ is represented by the image of $a$ in $\pi_1M/\langle b^2\rangle$, and $f(a)=\alpha^\mu$ according to \autoref{LEM: Minor completion}, which proves (\ref{8.5-2}). Meanwhile, (\ref{8.5-4}) automatically holds in the minor case, since $(\Z/2\Z)^{\times}$ has only one element.
\end{proof}

In this case, $(M,N)$ is called a \textit{Hempel pair} of scale factor $\kappa$, and $f$ is called a profinite isomorphism of \textit{scale type} $(\lambda,\mu)$. Note that the parameters $\lambda$ and $\mu$ are uniquely determined up to $\pm$-signs, regardless of the group presentation one chooses.

\autoref{PROP: Hempel pair} can also be described group-theoretically. For instance, when $\O$ is non-orientable, we can express $\pi_1M$ and $\pi_1N$ as
\small
\begin{equation}\label{EQU: Non-orientable}
\begin{aligned}
&\pi_1M=\left<\left.\begin{gathered}a_1,\ldots, a_r, e_0, e_1, \ldots, e_s, \\u_1, \ldots, u_g,h_M\end{gathered}\right.\left|\,\begin{gathered}e_0=\left(a_1 \cdots a_r e_1\cdots e_s u_1^2\cdots u_g^2\right)^{-1}\\ a_{j}^{p_j}h_M^{q_j}=1,  [a_j,h_M]=1,\\  [e_i,h_M]=1, \, u_ih_Mu_i^{-1}=h_M^{-1}\end{gathered}\right.\right>,\\
&\pi_1N=\left<\left.\begin{gathered}a_1',\ldots, a_r', e_0', e_1', \ldots, e_s', \\u_1', \ldots, u_g',h_N\end{gathered}\right.\left|\,\begin{gathered}e_0'=\left(a_1' \cdots a_r' e_1'\cdots e_s' u_1^{\prime 2}\cdots u_g^{\prime 2}\right)^{-1}\\(a_{j}')^{p_j}h_N^{q_j'}=1,\,  [a_j',h_N]=1,\\  [e_i',h_N]=1, \, u_i'h_Nu_i^{\prime -1}=h_N^{-1}\end{gathered}\right.\right>,
\end{aligned}
\end{equation}
\normalsize
where $q_j'\equiv \kappa q_j\pmod{p_j}$, such that $f(h_M)=h_N^{\lambda}$, and $f(e_i)=g_i(e_i')^{o(g_i)\mu} g_i^{-1}h_N^{\rho_i}$ for some $g_i\in \widehat{\pi_1N}$ and $\rho_i\in \widehat{\Z}$. Here, 
$$o:\widehat{\pi_1N}\tto \widehat{\pi_1\O_N}\tto \{\pm1\}$$ denotes the orientation homomorphism.

Similarly, when $\O$ is orientable, there exist presentations for $\pi_1M$ and $\pi_1N$
\small
\begin{equation}\label{EQU: Orientable}
\begin{aligned}
&\pi_1M=\left<\left.\begin{gathered}a_1,\ldots, a_r, e_0, e_1, \ldots, e_s, \\u_1, v_1, \ldots, u_g,v_g,h_M\end{gathered}\right.\left|\,\begin{gathered}e_0=\left(a_1 \cdots a_r e_1\cdots e_s [u_1, v_1]\cdots [u_g,v_g]\right)^{-1}\\ a_{j}^{p_j}h_M^{q_j}=1, \,h_M\text{ central}\end{gathered}\right.\right>,\\
&\pi_1N=\left<\left.\begin{gathered}a_1',\ldots, a_r',e_0', e_1', \ldots, e_s', \\u_1', v_1', \ldots, u_g',v_g',h_N\end{gathered}\right.\left|\,\begin{gathered}e_0'=\left(a_1' \cdots a_r' e_1'\cdots e_s' [u_1', v_1']\cdots [u_g',v_g']\right)^{-1}\\ (a_{j}')^{p_j}h_N^{q_j'}=1, \,h_N\text{ central}\end{gathered}\right. \right>,
\end{aligned}
\end{equation}
\normalsize
where $q_j'\equiv \kappa q_j\pmod{p_j}$, such that $f(h_M)=h_N^{\lambda}$, and $f(e_i)=g_i(e_i')^\mu g_i^{-1}h_N^{\rho_i}$ for some $g_i\in \widehat{\pi_1N}$ and $\rho_i\in \widehat{\Z}$. 
For consistency of notation, when $\O$ is orientable, we define
$$
o:\widehat{\pi_1N}\to \{\pm1\}
$$
to be the trivial homomorphism.

Consequently, when $f$ respects the peripheral structure, regardless whether $\O$ is orientable, $f$ restricting on the peripheral subgroups can be represented by matrices as shown by \autoref{EX: Boundary map}: 
$$ C_{g_i^{-1}}\circ f\begin{pmatrix}
    h_M& e_i
\end{pmatrix}=\begin{pmatrix}
    h_N& e_i'
\end{pmatrix}\begin{pmatrix}
   o(g_i) \lambda & o(g_i)\rho_i\\ 0 & o(g_i) \mu
\end{pmatrix},$$ where the pair of diagonal elements $(\lambda,\mu)$ is exactly the scale type of $f$, and is unified for all the boundary components up to $\pm$-signs.

For brevity of notation, the elements $e_i$ (or $e_i'$) are called {\em peripheral cross-sections} in the following context. Indeed, they can be identified as the boundary components of a cross-section of the base orbifold removing the neighbourhoods of the singular points. 
Note that the peripheral cross-sections are not unique, and they in fact depend on the precise Seifert invariants. 
Indeed, on each component of $\partial M$, the basis $\{h_M,e_i\}$ determines an orientation compatible with each other, and further determines the invariant $\sum \frac{q_j}{p_j}$. 

In particular, when $\lambda=\pm \mu$, a Hempel pair of scale factor $\pm 1$ is a pair of homeomorphic Seifert manifolds, according to the classification of Seifert fibrations. 


In the following, we denote $\Z^{-1}\widehat{\Z}=\left\{ \frac{\alpha}{p}\mid \alpha\in \widehat{\Z},\,p\in \Z\setminus \{0\}\right\}$ as the localization. Since $\widehat{\Z}$ is torsion-free, $\widehat{\Z}$ injects into $\Z^{-1}\widehat{\Z}$. 
The following lemma is based on the calculation of homology.

\begin{lemma}\label{LEM: Sum slope}
Let $(M,N)$ be a Hempel pair of scale factor $\kappa=\mu^{-1}\lambda$, and $f:\widehat{\pi_1M}\ttt\widehat{\pi_1N}$ be a profinite isomorphism respecting the peripheral structure, with scale type $(\lambda,\mu)$. Choose presentations of the fundamental groups as in (\ref{EQU: Non-orientable}) and (\ref{EQU: Orientable}), so that 
\begin{equation}\label{8.9equ1}
f(h_M)=h_N^\lambda\text{ and }f(e_i)=g_i(e_i')^{o(g_i)\mu}g_i^{-1}h_N^{\rho_i}.
\end{equation} 
Then, in $\Z^{-1}\widehat{\Z}$,
\newsavebox{\tempequa}
\begin{lrbox}{\tempequa}
$ \sum\limits_{i=0}^{s}\rho_i=\lambda\sum\limits_{j=1}^{r}\frac{q_j}{p_j}-\mu \sum\limits_{j=1}^{r}\frac{q_j'}{p_j}.$
\end{lrbox}
\begin{equation*}
\usebox{\tempequa}
\end{equation*}
\end{lemma}


\newsavebox{\tempequb}
\begin{lrbox}{\tempequb}
$\prod\limits_{j=1}^r p_j \sum\limits_{i=0}^{s}\left[e_i\right]=-\prod\limits_{j=1}^r p_j \sum\limits_{l=1}^{r}\left[a_l\right]=\sum\limits_{j=1}^{r} (\prod\limits_{k\neq j}p_k  )q_j[h_M].$
\end{lrbox}

\newsavebox{\tempequd}
\begin{lrbox}{\tempequd}
$
2\sum\limits_{i=0}^{s}\rho_i=\lambda\cdot 2\sum\limits_{j=1}^{r}\frac{q_j}{p_j}-\mu \cdot 2\sum\limits_{j=1}^{r}\frac{q_j'}{p_j}.
$
\end{lrbox}

\begin{proof}
We first prove the case when $M$ and $N$ have orientable base orbifolds. In this case, 
in $H_1(M;\Z)$ we have \begin{equation}\label{8.9equ2}
\usebox{\tempequb}
\end{equation}
Similarly, in $H_1(N;\Z)$ we have
\begin{equation}\label{8.9equ3}{\textstyle\prod\limits_{j=1}^r p_j \sum\limits_{i=0}^{s}\left[e_i'\right]=\sum\limits_{j=1}^{r} (\prod\limits_{k\neq j}p_k  )q_j'[h_N].}\end{equation}

Let $f_{\ast}: \tensor\hfree{M}\ttt \tensor \hfree{N}$ be the isomorphism induced by $f$ according to \autoref{PROP: Induce linear}. Then, in $\tensor\hfree{N}$, 

\begin{equation*}
\begin{aligned}
{\textstyle\lambda\cdot  \big(\sum\limits_{j=1}^{r} (\prod\limits_{k\neq j}p_k  )q_j \big)} [h_N]=& {\textstyle f_{\ast}\Big(  (\sum\limits_{j=1}^{r} (\prod\limits_{k\neq j}p_k  )q_j  ) [h_M]\Big)}& (\ref{8.9equ1})\\
=&{\textstyle f_\ast\Big(\prod\limits_{j=1}^r p_j \sum\limits_{i=0}^{s}\left[e_i\right]\Big)}&(\ref{8.9equ2})\\
=&{\textstyle\prod\limits_{j=1}^r p_j  \sum\limits_{i=0}^{s} \left (\mu[e_i']+\rho_i[h_N]\right)}&(\ref{8.9equ1})\\
=&{\textstyle\mu\cdot\prod\limits_{j=1}^r p_j \sum\limits_{i=0}^{s}\left[e_i'\right]+ \prod\limits_{j=1}^r p_j  ( \sum\limits_{i=0}^{s}\rho_i ) [h_N]}&\\
=&{\textstyle\Big(\mu\cdot \sum\limits_{j=1}^{r} (\prod\limits_{k\neq j}p_k  )q_j'+\prod\limits_{j=1}^r p_j  ( \sum\limits_{i=0}^{s}\rho_i )\Big ) [h_N]}&(\ref{8.9equ3})
\end{aligned}
\end{equation*}

\allowdisplaybreaks[0]
\newsavebox{\tempequc}
\begin{lrbox}{\tempequc}
$
\sum\limits_{i=0}^{s}\rho_i=\lambda\sum\limits_{j=1}^{r}\frac{q_j}{p_j}-\mu \sum\limits_{j=1}^{r}\frac{q_j'}{p_j}\in \Z^{-1}\widehat{\Z}.
$
\end{lrbox}

According to \cite[Lemma 2]{OVZ67}, $[h_N]$ is non-torsion in $H_1(N;\Z)$ when the base orbifold $\O_N$ is orientable, i.e. $[h_N]\neq 0$  in the finitely generated free $\Z$-module $ \hfree{N}$. Thus, the above coefficients must be equal:
\begin{equation*}{\textstyle\lambda\cdot \big(\sum\limits_{j=1}^{r} (\prod\limits_{k\neq j}p_k  )q_j\big )=\mu\cdot \sum\limits_{j=1}^{r} (\prod\limits_{k\neq j}p_k  )q_j'+\prod\limits_{j=1}^r p_j  ( \sum\limits_{i=0}^{s}\rho_i )}.\end{equation*}
Quotienting the non-zero integer $\prod_{j=1}^r p_j$ implies 
\begin{equation*}
\usebox{\tempequc}
\end{equation*}

Next, we move on to the case when $M$ and $N$ have non-orientable base orbifolds. Let $M_o$ and $N_o$ be the two-fold covers of $M$ and $N$ corresponding to the centralizer of the fiber subgroups, i.e. $\widehat{\pi_1M_o}\cong \overline{\pi_1M_o}$ is the centralizer of $\overline{K_M}$, and so is $\widehat{\pi_1N_o}$. Since $f(\overline{K_M})=\overline{K_N}$ by \autoref{PROP: Virtually central}, $f$ restricts to an isomorphism $f:\widehat{\pi_1M_o}\ttt \widehat{\pi_1N_o}$ that also respects the peripheral structure. In fact, the scale type of this isomorphism is also $(\lambda,\mu)$.

To be explicit, pick an element $\sigma \in \pi_1M\setminus \pi_1M_o$, and let $d_i=\sigma e_i^{-1}\sigma ^{-1}$. Note that each boundary component as well as each singular fiber of $M$ lifts to two components in $M_o$, and the conjugation by $\sigma$ on $\pi_1M_o$ corresponds to the deck transformation on $M_o$ which is orientation-preserving but fiber-reversing. Thus, $\{h_M^{-1},d_i^{-1}\}=\{\sigma h_M\sigma ^{-1},\sigma e_i\sigma^{-1}\}$ determines an orientation on the boundary component, which is consistent with $\{h_M, e_i\}$. By simultaneously reversing the two generators, the elements $e_0,\cdots, e_s,d_0,\cdots, d_s$ become the (conjugates of) peripheral cross-sections in $\pi_1M_o$. Together with the orientation on the fiber determined by the generator $h_M$, this determines the Seifert invariants of $M_o$ as $(p_1,q_1),(p_1,q_1),\cdots, (p_r,q_r),(p_r,q_r)$, i.e. copying the Seifert invariants of $M$ twice with the same $\pm$-sign.

Similarly, pick $\delta \in \pi_1N\setminus \pi_1N_o$, and let $d_i'=\delta e_i'^{-1} \delta^{-1}$. Then  the Seifert invariants of $N_o$, determined by these peripheral cross-sections and the fiber generator $h_N$, are $(p_1,q_1'),(p_1,q_1'),\cdots, (p_r,q_r'),(p_r,q_r')$.

For each $0\le i \le s$, $f(e_i)=g_i(e_i')^{o(g_i)\mu}g_i^{-1}h_N^{\rho_i}$ where $g_i\in \widehat{\pi_1N}$. If $o(g_i)=1$, i.e. $g_i\in \widehat{\pi_1M_o}$, then it is easy to verify that $f(d_i)=g_i'(d_i')^{\mu} g_i^{\prime -1} h_N^{\rho_i}$, where $g_i'=f(\sigma)g_i\delta^{-1}\in \widehat{\pi_1N_o}$. If $o(g_i)={-1}$, then similarly, $f(e_i)=l_i (d_i')^\mu l_i^{-1} h_N^{\rho_i}$ and $f(d_i)=l_i' (e_i')^\mu l_i^{\prime -1} h_N^{\rho_i}$, where $l_i=g_i\delta^{-1}\in \widehat{\pi_1N_o}$ and $l_i'=f(\sigma)g_i\in \widehat{\pi_1N_o}$.

This implies that $f:\widehat{\pi_1M_o}\ttt \widehat{\pi_1N_o}$ has scale type $(\lambda, \mu)$. The same calculation in the homology shows that
\begin{equation*}
\begin{gathered}
{\textstyle \prod\limits_{j=1}^r p_j \big(\sum\limits_{i=0}^{s}\left[e_i\right]+\sum\limits_{i=0}^s[d_i]\big)=2\sum\limits_{j=1}^{r} (\prod\limits_{k\neq j}p_k  )q_j[h_M]\in H_1(M_o;\Z),}\\
{\textstyle  \prod\limits_{j=1}^r p_j \big(\sum\limits_{i=0}^{s}\left[e_i'\right]+\sum\limits_{i=0}^s[d_i']\big)=2\sum\limits_{j=1}^{r} (\prod\limits_{k\neq j}p_k  )q_j'[h_N]\in H_1(N_o;\Z).}
\end{gathered}
\end{equation*}
Since $M_o$ and $N_o$ have orientable base orbifolds, the calculation in the first case implies  that
\begin{equation*}
\usebox{\tempequd}
\end{equation*}
Quotienting the coefficient $2$ yields the conclusion.
\end{proof}

\subsection{Realizing as homeomorphism}


This subsection serves as a preparation for the proof of \autoref{inthm: Graph new}. In particular, we prove \autoref{inthm: Graph new}~(\ref{di222}) with the graph manifolds replaced by Seifert fibered spaces.

\begin{proposition}\label{PROP: Seifert Zx-regular}
Let $M$ and $N$ be Seifert fibered manifolds with non-empty incompressible boundary. Suppose that $f:\widehat{\pi_1M}\ttt\widehat{\pi_1N}$ is an isomorphism respecting the peripheral structure, 
and that $f$ is peripheral $\widehat{\Z}^{\times}$-regular at some selected boundary components $\partial_0 M, \partial_1 M,\cdots, \partial_k M$ of $M$, where $1\le k+1 \le \# \partial M$.\footnote{The indices start from $0$ just for consistency with the subsequent notations.} 
\begin{enumerate}[leftmargin=*]
\item Then, $M$ and $N$ are homeomorphic.
\item There exists an element  $\lambda\in \widehat{\Z}^{\times}$ and two  homeomorphisms $\Phi^\pm:M\to N$ such that 
for each $0\le i\le k$, $f$ restricting on $\partial_iM$ is induced by the homeomorphism $\Phi^\pm$ with coefficient $\pm\lambda$, where the $\pm$-signs are consistent.
\end{enumerate}
\end{proposition}

It is worth pointing out that a homeomorphism  between the manifolds is nothing more than a group isomorphism between their fundamental groups  respecting the peripheral structure, as is proven by Waldhausen \cite{Wal68}.

\begin{proposition}
\label{PROP: Homeo induce}
Let $M$ and $N$ be  compact, irreducible 3-manifolds with non-empty incompressible boundary. Then any isomorphism $\varphi: \pi_1M\to \pi_1N$  respecting the peripheral structure can be induced, up to conjugation, by  
a homeomorphism $\Phi :M\to N$.
\end{proposition}

Thus, the homeomorphisms in \autoref{PROP: Seifert Zx-regular} can either be constructed in a concrete way, or be obtained from a group theoretical construction. Our proof will be a combination of these two methods.

\begin{lemma}\label{LEM: Dehn twist}
Let $M$ be a major or a  minor Seifert manifold.  Let $e_0,\cdots, e_s$ denote the peripheral cross-sections and let $h_M$ denote a generator of the fiber subgroup. For any $t_0,\cdots, t_s\in \Z$ such that $\sum_{i=0}^ s t_i=0$, there exists a homeomorphism $\Psi:M\to M$ such that $\Psi_{\ast}(h_M)=h_M$ and $\Psi_{\ast}(e_i)=g_ie_ig_i^{-1}h_M^{t_i}$, where every $g_i$ belongs to the kernel of the orientation homomorphism $o$.
\end{lemma}
\begin{proof}
In fact, this can be realized through vertical Dehn twists. 
Let $F$ be a compact, connected, orientable sub-surface of the base orbifold $\O_M$ which contains all the boundary components of $M$, while it contains no singular points. Then the union of fibers over $F$ is a sub-manifold of $M$ containing $\partial M$, which decomposes as $F\times S^1$.

We may choose disjoint arcs $\gamma_1,\cdots, \gamma_s$ in $F$ connecting the boundary components of $F\cap \partial \O_M$ that correspond to $e_0,\cdots, e_s$ one-by-one. Let $\Psi_i$ be the vertical Dehn twist (with appropriate orientation) supported on $N(\gamma_i)\times S^1$, where $N(\gamma_i)$ denotes the neighbourhood of $\gamma_i$ in $F$. Then $\Psi=\Psi_1^{t_0}\circ \Psi_2^{t_1+t_0}\circ\cdots \circ \Psi_s^{t_{s-1}+\cdots+t_0}$ meets the requirement.

Group theoretically, if we fix the presentation of the fundamental group as in (\ref{EQU: Non-orientable}) or (\ref{EQU: Orientable}), then $\Psi_{\ast}$ simply sends $e_i$ to $e_ih_M^{t_i}$ and fixes all the other generators.
\end{proof}

\begin{lemma}\label{LEM: Swap boundary}
Suppose $M$ is a major Seifert manifold with a non-orientable base orbifold. Let $e_0,\cdots, e_s$ denote the peripheral cross-sections, let $h_M$ be a generator of the fiber subgroup, and let $o:\pi_1M\to \pi_1\O_M\to \{\pm1\}$ be the orientation homomorphism. Then, for any $\epsilon_0,\cdots,\epsilon_s\in \{\pm1\}$, there exists a homeomorphism $\Theta :M\to M$ such that $\Theta_{\ast}(h_M)=h_M$ and $\Theta_{\ast}(e_i)=g_ie_i^{\epsilon_i}g_i^{-1}$, where $o(g_i)=\epsilon _i$.
\end{lemma}
\begin{proof}
By composition, it suffices to prove the conclusion when exactly one of $\epsilon_i$ is $-1$. 
We first show that the conclusion holds when  $\epsilon_0=-1$ and $\epsilon_1=\cdots=\epsilon_s=1$.

Fix the presentation  $$\pi_1M=\left<\left.\begin{gathered}a_1,\ldots, a_r, e_0, e_1, \ldots, e_s, \\u_1, \ldots, u_g,h_M\end{gathered}\right.\left|\,\begin{gathered}e_0=\left(a_1 \cdots a_r e_1\cdots e_s u_1^2\cdots u_g^2\right)^{-1}\\ a_{j}^{p_j}h_M^{q_j}=1,\,  [a_j,h_M]=1,\\  [e_i,h_M]=1, \, u_ih_Mu_i^{-1}=h_M^{-1}\end{gathered}\right.\right>.$$
Denote $t=a_1\cdots a_re_1\cdots e_su_1^2\cdots u_{g-1}^2$, and let $\theta:\pi_1M\to \pi_1M$ be defined by: 
\begin{align*}
&\theta(h_M)=h_M,&\\
&\theta(a_j)=a_j,& 1\le j \le r,\\
&\theta(e_i)=e_i,& 1\le i \le s,\\
&\theta(u_k)=u_k,& 1\le k\le g-1,\\
&\theta(u_g)=t^{-1}u_g^{-1}.&
\end{align*}
It is easy to verify that $\theta$ is a group isomorphism and $\theta(e_0)=\theta(u_g^{-2}t^{-1})=u_gtu_g=u_ge_0^{-1}u_g^{-1}$, where $o(u_g)=-1$. In particular, $\theta$ respects the peripheral structure, so it is induced by a homeomorphism $\Theta$ by \autoref{PROP: Homeo induce}, which satisfies our requirement.

Now we assume $\epsilon_j=-1$, and $\epsilon_i=1$ for all $i\neq j$. 
 Note that the boundary components of $M$ are symmetric, i.e. there exists a homeomorphism $\Phi$ on $M$ switching the two  boundary components corresponding to $\epsilon_0$ and $\epsilon_j$, while preserving the cross-section of the entire orbifold with the neighbourhoods of the singular points removed. Then up to composing with a conjugation, $\Phi_{\ast}$ sends $h_M$ to $h_M$, and sends $e_i$ to $\alpha_i e_{\sigma(i)}^{o(\alpha_i)}\alpha_i^{-1}$, where $\sigma(0)=j$ and $\sigma(j)=0$. 
Let $\Theta_j=\Phi^{-1}\circ \Theta\circ \Phi$. One can easily verify that $\Theta_j$ satisfies the requirement.

From another perspective, this is in fact realized by a $\Z/2$-equivariant homeomorphism on the two-fold cover $M_o$ corresponding to the kernel of $o$, which swaps the two components  in the pre-image of each boundary component of $M$ corresponding to $e_i$ with $\epsilon_i=-1$. 
\end{proof}

\begin{lemma}\label{LEM: -1}
Let $M$  be a Seifert fibered 3-manifold with incompressible boundary. Then there exists a homeomorphism $\tau: M\to M$ preserving each boundary component, such that up to composing with a conjugation, $\tau_{\ast}$ is a scalar multiplication by $-1$ on each peripheral subgroup $\pi_1\partial_iM\cong \Z\oplus\Z$.
\end{lemma}
\begin{proof} 
When $M$ is the minor Seifert manifold, this follows from the group theoretical construction as in \autoref{PROP: Homeo induce}, i.e. an automorphism of $\pi_1M=\left\langle a,b\mid bab^{-1}=a^{-1}\right\rangle $ sending $a$ to $a^{-1}$ and $b$ to $b^{-1}$. 

Also note that there is an automorphism of the solid torus $V$, which induces the $-id$--map on its boundary.

Generally, $M$ can be constructed in the following way. Take a compact, orientable surface $F$, and let $X=F\times S^1$. Choose some of the boundary components of $X$, and $M$ is obtained through attaching either a solid torus or a minor Seifert manifold to each of them in an appropriate way.

Let $\sigma$ be an orientation-reversing involution of $F$ preserving each boundary component, and let $r$ denote the reflection on $S^1$. Then the involution $\tau=(\sigma,r)$ on $X$ restricts to a scalar multiplication by $-1$ on each boundary component. Thus, for any of the above chosen boundary components $\partial_iX$, $\tau|_{\partial_iX}$ can be extended as a homeomorphism on the filled-in solid torus or minor Seifert manifold as shown by the above construction. This  yields an automorphism of $M$ satisfying the required condition.
\end{proof}

We are now ready to prove \autoref{PROP: Seifert Zx-regular} through combining the above constructions.

\begin{proof}[Proof of \autoref{PROP: Seifert Zx-regular}]
Note that \autoref{PROP: Seifert Zx-regular} holds apparently for the thickened tori. Thus, we assume that $M$ and $N$ are either major or minor Seifert manifolds in the following proof. 

Suppose the scale type of $f$ is $(\lambda,\mu)$. 
Since $f:\widehat{\pi_1M}\ttt\widehat{\pi_1N}$ is peripheral $\Zx$-regular at some of the boundary components, it follows from \autoref{EX: Boundary map} that $${ \left(\begin{matrix} \lambda &\rho \\ 0 &\mu\end{matrix}\right)= \eta \cdot A}$$  for some $\eta\in \Zx$ and $A\in \mathrm{GL}(2,\Z)$. Thus, $\lambda=\pm \mu$. 
By changing the generator of the fiber subgroup, we may assume $\lambda=\mu$. Then $(M,N)$ is a Hempel pair of scale factor $1$, and hence $M$ and $N$ are homeomorphic.

We follow the presentation of fundamental groups as in (\ref{EQU: Non-orientable})  or (\ref{EQU: Orientable}). Then, ${q_j'\equiv q_j\pmod{p_j}}$. In fact, by replacing $a_j'$ with $a_j'h_N^{c_j}$ for some appropriate $c_j\in \Z$ and adjusting $e_0'$ correspondingly, we may assume $q_j'=q_j$ without loss of generality, which establishes an apparent isomorphism between $\pi_1M$ and $\pi_1N$. This yields an apparent homeomorphism $\mathscr{F}:M\ttt N$ by \autoref{PROP: Homeo induce} which we fix henceforth. 

Based on this presentation, suppose $f(h_M)=h_N^\lambda$ and $f(e_i)=g_i(e_i')^{o(g_i)\lambda}g_i^{-1}h_N^{\rho_i}$ ($0\le i\le s$), where $o:\widehat{\pi_1N}\to \{\pm 1\}$ is the orientation homomorphism. It then  follows from \autoref{LEM: Sum slope} that $\sum_{i=0}^s \rho_i=0$. Fix peripheral subgroups $\pi_1\partial_iM=\left\langle h_M\right\rangle \times \left\langle e_i \right\rangle $ and $\pi_1\partial_iN=\left\langle h_N\right\rangle \times \left\langle e_i' \right\rangle $.
Then $C_{g_i^{-1}}\circ f: \overline{\pi_1\partial_iM}\to \overline{\pi_1\partial_iN}$ is represented by a matrix. 
\begin{equation}\label{8.20equ1}
C_{g_i^{-1}}\circ f \begin{pmatrix}h_M&e_i\end{pmatrix} =\begin{pmatrix} h_N & e_i'\end{pmatrix}\begin{pmatrix} o(g_i)\lambda & o(g_i)\rho_i\\ 0 & o(g_i)\lambda\end{pmatrix}.
\end{equation}

Recall that $f$ is peripheral $\Zx$-regular at $\partial_0M,\cdots,\partial_kM$. 
%
%
Thus, these matrices can be decomposed as $\lambda$ multiplicating a matrix in $\mathrm{GL}(2,\Z)$, regardless of the choice of the conjugators $g_i$, according to \autoref{PROP: peripheral structure}. 
Hence, 
for $0\le i \le k$, 
$\rho_i= \lambda t_i$ for some $t_i\in \Z$.

\newsavebox{\tempeque}
\begin{lrbox}{\tempeque}
$
\begin{tikzcd}
\Psi:\, M \arrow[r,"\cong"',"\Psi_M"] & M \arrow[r,"\cong"',"{\mathscr{F}}"] & N
\end{tikzcd}
$
\end{lrbox}
\newsavebox{\tempequf}
\begin{lrbox}{\tempequf}
$
\begin{tikzcd}
\Phi^+:\, M \arrow[r,"\cong"',"\Theta"] & M \arrow[r,"\cong"',"{\Psi}"] & N.
\end{tikzcd}
$
\end{lrbox}

If $k=s$, then $\sum_{i=0}^s\rho_i=0$ implies $\sum_{i=0}^s t_i=0$. If $k<s$, then we can choose $t_{k+1},\cdots, t_s\in \Z$ so that $\sum_{i=0}^s t_i=0$. Then, by \autoref{LEM: Dehn twist}, there exists a homeomorphism 
\begin{equation*}
\usebox{\tempeque}
\end{equation*}
 such that for each $0\le i\le s$, $\Psi_{\ast}(e_i)=\alpha_i e_i' \alpha_i^{-1} h_N^{t_i}$ for some $\alpha_i \in \ker (o)$.

Let $\epsilon_i=o(g_i)\in \{\pm 1\}$. 
When the base orbifolds of $M$ and $N$ are non-orientable,  let $\Theta: M\to M$ be the homeomorphism constructed in \autoref{LEM: Swap boundary}, 
and let 
\begin{equation*}
\usebox{\tempequf}
\end{equation*}
When the base orbifolds of $M$ and $N$ are orientable, $\epsilon_i=1$, and we simply let $$\Phi^+=\Psi.$$ 
In either case, for each $0\le i\le s$, we have 
\begin{equation}\label{equ820*}
\Phi^+_{\ast}(e_i)=\beta_i (e_i')^{\epsilon_i} \beta_i^{-1} h_N^{t_i},
\end{equation}
 where $\beta_i\in \pi_1N$ and $o(\beta_i)=\epsilon_i$.

\newsavebox{\tempequg}
\begin{lrbox}{\tempequg}
$
C_{g_i^+}\circ f| _{\overline{\pi_1\partial_i M}}=\lambda\otimes (\Phi^+_{\ast}|_{\pi_1\partial_i M}): \overline{\pi_1\partial_iM}\cong \widehat{\Z}\otimes \pi_1\partial_iM \tto \widehat{\Z}\otimes  \pi_1\partial_iN\cong \overline{\pi_1\partial_iN}
$
\end{lrbox}

Now, we take $g_i^+=\beta_ig_i^{-1}$, so  $o(g_i^+)=1$. By (\ref{8.20equ1}), $C_{g_i^+}\circ f$ sends $h_M$ to $h_N^{\lambda}$ and sends $e_i$ to $\beta_i (e_i')^{\epsilon_i\lambda} \beta_i^{-1} h_N^{\rho_i}$. Thus, according to (\ref{equ820*}), since $\rho_i=\lambda t_i$, we derive that  
\begin{equation*}
\scalebox{0.94}{\usebox{\tempequg}}
\end{equation*}
for each $0\le i\le k$. In other words, $f$ restricting on $\partial_iM$ is induced by the homeomorphism $\Phi^+$ with coefficient $\lambda$. 

For the alternate $\pm$-sign, let $\tau: M\to M$ be constructed by \autoref{LEM: -1}, which, up to conjugation, acts as a scalar multiplication of $-1$ on each peripheral subgroup of $\pi_1M$. Then, we  define $\Phi^-=\Phi^+\circ \tau$, and thus $f$ restricting on $\partial_iM$ is induced by the homeomorphism $\Phi^-$ with coefficient $-\lambda$. 
\end{proof}

\section{Graph manifolds}\label{SEC: Graph Mfd}

In our context, a graph manifold denotes a compact, orientable, irreducible 3-manifold with empty or incompressible toral boundary, which is not a closed $Sol$-manifold, and whose JSJ-decomposition consists of only Seifert fibered pieces. For consistency of proof, a single Seifert fibered manifold is also recognized as a graph manifold.

We shall denote a graph manifold as $M=(M_{\bullet},\Gamma)$, where $\Gamma$ is the oriented JSJ-graph of $M$, and $M_v$ is a Seifert fibered space for each $v\in V(\Gamma)$. 

\subsection{Profinite isomorphism between graph manifolds}\label{subSEC: graph1}

In this subsection, we consider the following setting. Let $M=(M_{\bullet},\Gamma)$ and $N=(N_{\bullet},\Gamma')$ be two graph manifolds, and let $f:\widehat{\pi_1M}\ttt\widehat{\pi_1N}$ be an isomorphism. According to \autoref{THM: Profinite Isom up to Conj}, $f$ determines a congruent isomorphism $f_{\bullet}$ between the JSJ-graphs of profinite groups. For brevity of notation, we  identify $\Gamma$ and $\Gamma'$ through $\F$.

\newsavebox{\butterfly}
\begin{lrbox}{\butterfly}
$
\begin{tikzcd}[column sep=1.25em]
\widehat{\pi_1M_u} \arrow[ddd, "f_u"',"\cong"] \arrow[r,symbol=\supseteq] & \overline{\pi_1\partial_iM_u} \arrow[rddd, "C_{g_0}^{-1}\circ f_u "',"\cong"] \arrow[rrrrrr, "\widehat{\psi}","\cong"', bend left=17] &                 &                                               & \widehat{\pi_1T^2} \arrow[rrr, "\widehat{\varphi_1}","\cong"'] \arrow[ddd, "f_e","\cong"'] \arrow[lll, "\widehat{\varphi_0}"',"\cong"] &                    &                     & \overline{\pi_1\partial_kM_v} \arrow[lddd, "C_{g_1}^{-1}\circ f_v ","\cong"'] \arrow[r,symbol=\subseteq] & \widehat{\pi_1M_v} \arrow[ddd, "f_v","\cong"'] \\
                                                 &                                                                                                                                        &                                                    & &            &                                                                                                            &                                         &                                                                                                                 &                                       \\
                                                 &                                                                                                                                        &                                                      & &          &                                                                                                            &                                         &                                                                                                                 &                                       \\
\widehat{\pi_1N_u}                               & \widehat{\pi_1N_u} \arrow[r,symbol=\supseteq] \arrow[l, "C_{g_0}"',"\cong"]                                                                                     & \overline{\pi_1\partial_jN_u} \arrow[rrrr, "\widehat{\psi'}"',"\cong", bend right=24] &  & \widehat{\pi_1T^2} \arrow[rr, "\widehat{\varphi_1'}","\cong"'] \arrow[ll, "\widehat{\varphi_0'}"',"\cong"]         &           & \overline{\pi_1\partial_lN_v} \arrow[r,symbol=\subseteq] &  \widehat{\pi_1N_v} \arrow[r, "C_{g_1}","\cong"']                                                                         & \widehat{\pi_1N_v}                   
\end{tikzcd}
$
\end{lrbox}

For any $e\in E(\Gamma)$, denote the two endpoints as $u=d_0(e)$ and $v=d_1(e)$. Then the congruent isomorphism $f_{\bullet}$ implies the following commutative diagram.
\begin{equation}\label{EQU: JSJ-graph diagram}
\scalebox{0.9}{\usebox{\butterfly}}
\end{equation}

Fix a free $\Z$-basis $\left\{h_{M_u},e_i\right\}$ for the peripheral subgroup $\pi_1\partial_i M_u\cong \Z\oplus\Z$, where $h_{M_u}$ is represented by a regular fiber, and $e_i$ is a peripheral cross-section (see \autoref{SEC: Seifert peripheral}). When $M_u$ is the minor Seifert manifold, we always view it as a Seifert fibration over the disk with two singular points of index $2$ so as to avoid the ambiguity of $h_{M_u}$. Then, $\overline{\pi_1\partial_i M_u}\subseteq \widehat{\pi_1M_u}$ is a free $\widehat{\Z}$-module over the same basis $\left\{h_{M_u},e_i\right\}$ according to \autoref{COR: Peripheral group tensor}. Similarly, we fix $\pi_1\partial_kM_v=\Z\left\{h_{M_v},e_k\right\}$, ${\pi_1\partial_jN_u}={\Z}\left\{h_{N_u},e_j'\right\}$, and ${\pi_1\partial_lN_v}={\Z}\left\{h_{N_v},e_l'\right\}$.

Over these bases, the maps $\widehat{\psi}$, $\widehat{\psi'}$, $C_{g_0}^{-1}\circ f_u$, and $C_{g_1}^{-1}\circ f_v$ can be represented by matrices in $\mathrm{GL}(2,\widehat{\Z})$ as shown by the following diagram.

\begin{equation}\label{EQU: Matrix}
\begin{tikzcd}[ampersand replacement=\&]
{\overline{\pi_1\partial_iM_u}=\widehat{\Z}\left\{h_{M_u},e_i\right\}} \arrow[rrr, "{\left(\begin{smallmatrix} \alpha_e &\beta_e\\   \gamma_e&\delta_e\end{smallmatrix}\right)}","\widehat{\psi}"'] \arrow[ddd, "{\left(\begin{smallmatrix} \lambda_{e,0} &\rho_{e,0}\\ 0&\mu_{e,0}\end{smallmatrix}\right)}"',"C_{g_0}^{-1}\circ f_u"] \& \& \& {\widehat{\Z}\left\{h_{M_v},e_k\right\}=\overline{\pi_1\partial_kM_v}} \arrow[ddd, "{\left(\begin{smallmatrix} \lambda_{e,1} &\rho_{e,1}\\0&\mu_{e,1}\end{smallmatrix}\right)}", "C_{g_1}^{-1}\circ f_v"'] \\
                                                                                                                                                                                                              \& \& \&                                                                                                                              \\
\& \& \& \\
{\overline{\pi_1\partial_jN_u}=\widehat{\Z}\left\{h_{N_u},e_j'\right\}} \arrow[rrr, "{\left(\begin{smallmatrix} \alpha'_e &\beta'_e\\     \gamma'_e&\delta'_e\end{smallmatrix}\right)}"',"\widehat{\psi'}"]                                                                                 \& \& \& {\widehat{\Z}\left\{h_{N_v},e_l'\right\}=\overline{\pi_1\partial_lN_v}}                                                                                    
\end{tikzcd}
\end{equation}

In particular, the upper-triangle matrices $\left(\begin{smallmatrix}\lambda & \rho \\0 & \mu\end{smallmatrix}\right)$ (with indices omitted) follow from \autoref{EX: Boundary map}, and $\lambda,\mu\in \Zx$. Note that  $\widehat{\psi}$ is in fact the profinite completion of the gluing map in $\mathrm{Isom}_{\Z}(\pi_1\partial_iM_u,\pi_1\partial_kM_v)$, so $ \left(\begin{smallmatrix} \alpha_e &\beta_e\\   \gamma_e&\delta_e\end{smallmatrix}\right)\in \mathrm{GL}(2,\Z)$.  
 Similarly, $\left(\begin{smallmatrix} \alpha'_e &\beta'_e\\     \gamma'_e&\delta'_e\end{smallmatrix}\right)\in \mathrm{GL}(2,\Z)$.

\begin{lemma}\label{LEM: Graph mfd calculation}
We follow the notation of (\ref{EQU: Matrix}), and omit the subscript $e$ for brevity. Then,
\begin{enumerate}[leftmargin=*]
\item\label{9.2-1} $\gamma=\pm\gamma'\neq 0$;
\item\label{9.2-2} $\lambda_1=\pm \mu_0$ and $\mu_1=\pm \lambda_0$; 
\item\label{9.2-3}  $\frac{\delta'}{\gamma'}=\lambda_0\mu_0^{-1}\frac{\delta}{\gamma}-\mu_0^{-1}\rho_0$ and $\frac{\alpha'}{\gamma'}=\mu_1^{-1}\lambda_1\frac{\alpha} {\gamma} +\mu_1^{-1}\rho_1$ in $\Z^{-1}\widehat{\Z}$;
\item\label{9.2-4} if $\lambda_0=\pm \mu_0$, then $\rho_0,\rho_1\in \lambda_0\Z$, and consequently, $f_u$ is peripheral $\lambda_0$-regular at $\partial_iM_u$ and $f_v$ is peripheral $\lambda_0$-regular at $\partial_kM_v$.
\end{enumerate}
\end{lemma}
The calculation was mainly done in the proof of \cite[Theorem 10.1]{Wil18}. We shall briefly reproduce this calculation for completeness.
\begin{proof}
Note that $\gamma$ is the intersection number of the regular fibers of $M_u$ and $M_v$ on the common toral boundary. The minimality of the JSJ-system of tori implies that $\gamma\neq 0$, for otherwise, gluing $M_u$ and $M_v$ along this common boundary in a fiber-parallel way still yields a Seifert fibered space. Similarly, $\gamma'\neq 0$.

The commutative diagram (\ref{EQU: Matrix}) implies
\begin{equation*}
\begin{aligned}
\left(\begin{matrix} \lambda_1\alpha+\rho_1\gamma& \lambda_1\beta+\rho_1\delta\\ \mu_1\gamma&\mu_1\delta\end{matrix}\right)=&\left(\begin{matrix} \lambda_1 & \rho_1 \\ 0 & \mu_1\end{matrix}\right ) \left( \begin{matrix} \alpha & \beta \\ \gamma & \delta\end{matrix}\right)=\left(\begin{matrix}\alpha'&\beta'\\ \gamma' & \delta'\end{matrix}\right)\left(\begin{matrix} \lambda_0 & \rho_0\\ 0 &\mu_0\end{matrix}\right )\\=&\left(\begin{matrix} \lambda_0\alpha' & \rho_0\alpha'+\mu_0\beta'\\ \lambda_0\gamma' & \rho_0\gamma'+\mu_0\delta'\end{matrix}\right).
\end{aligned}
\end{equation*}
\normalsize

In particular, the lower-left entry implies $\lambda_0^{-1}\mu_1\gamma=\gamma'$. Since $\gamma$ and $\gamma'$ are non-zero integers, while $\lambda_0^{-1}\mu_1\in \Zx$, \cite[Lemma 2.2]{Wil18} implies that $\gamma=\pm \gamma'$ and $\lambda_0^{-1}\mu_1=\pm1$. Thus, $\mu_1=\pm \lambda_0$. 
In addition, since $\left(\begin{smallmatrix} \alpha&\beta\\   \gamma&\delta\end{smallmatrix}\right)\in \mathrm{GL}(2,\Z)$, $\det\left(\begin{smallmatrix} \alpha&\beta\\   \gamma&\delta\end{smallmatrix}\right) =\pm1$, and similarly  $\det\left(\begin{smallmatrix} \alpha'&\beta'\\   \gamma'&\delta'\end{smallmatrix}\right) =\pm1$. Thus, $\det \left(\begin{smallmatrix} \lambda_0 & \rho_0\\ 0 &\mu_0\end{smallmatrix}\right )= \pm \det \left(\begin{smallmatrix} \lambda_1 & \rho_1\\ 0 &\mu_1\end{smallmatrix}\right )$, i.e. $\lambda_0\mu_0=\pm \lambda_1\mu_1$. Thus, $\mu_1=\pm \lambda_0$ implies $\lambda_1=\pm \mu_0$. This proves (\ref{9.2-1}) and (\ref{9.2-2}).

Moreover, the lower-right entry implies $$\frac{\delta'}{\gamma'}=\frac {\mu_0^{-1}\mu_1 \delta-\mu_0^{-1}\rho_0\gamma'}{\gamma'}=\frac{\mu_0^{-1}\mu_1\delta}{\lambda_0^{-1}\mu_1\gamma}-\mu_0^{-1}\rho_0=\lambda_0\mu_0^{-1}\frac{\delta}{\gamma}-\mu_0^{-1}\rho_0,$$ and the upper-left entry implies $$\frac{\alpha'}{\gamma'}=\frac{\lambda_0^{-1}\lambda_1\alpha+\lambda_0^{-1}\rho_1\gamma}{\lambda_0^{-1}\mu_1\gamma}=\mu_1^{-1}\lambda_1\frac{\alpha} {\gamma} +\mu_1^{-1}\rho_1,$$ which proves (\ref{9.2-3}).

When $\lambda_0=\pm \mu_0$, we have $\mu_0^{-1}\mu_1=\pm1$ and $\lambda_0^{-1}\lambda_1=\pm 1$. The lower-right entry implies $\mu_0^{-1}\rho_0\gamma'=\mu_0^{-1}\mu_1\delta-\delta'\in \Z$. Since $\gamma'\in \Z\setminus\{ 0\}$, \cite[Lemma 2.2]{Wil18} then implies $\mu_0^{-1}\rho_0\in \Z$, i.e. $\rho_0\in \mu_0\Z=\lambda_0\Z$. Similarly, the upper-left entry implies $\lambda_0^{-1}\rho_1\gamma=\alpha'-\lambda_0^{-1}\lambda_1\alpha\in \Z$. Again, $\gamma\in \Z\setminus\{ 0\}$ and \cite[Lemma 2.2]{Wil18} implies $\lambda_0^{-1}\rho_1\in \Z$, i.e. $\rho_1\in \lambda_0\Z$. This finishes the proof of (\ref{9.2-4}).
\end{proof}

\subsection{Profinite almost rigidity of graph manifolds}\label{subsec: graph almost}

\begin{theorem}\label{THM: Graph mfd almost rigid}
The class of graph manifolds is profinitely almost rigid.
\end{theorem}

\begin{remark}
This was proved for the class of closed graph manifolds by Wilkes \cite{Wil18}, for which a finite list of numerical equations was introduced to determine precisely whether two  closed graph manifolds are profinitely isomorphic. 
In particular, Wilkes showed that closed graph manifolds are profinitely almost rigid but not profinitely rigid, with the exceptions constructed from gluing Hempel pairs.
\end{remark}

\newsavebox{\tempequi}
\begin{lrbox}{\tempequi}
$
\begin{aligned}
        &\pi_1M_v=\left<\left.\begin{gathered}h_v, a_1,\ldots, a_r,\\ u_1, v_1, \ldots, u_g,v_g,  \\ e_{0,v}, e_{1,v}, \ldots, e_{s,v}\end{gathered}\,\right|\left.\,\begin{gathered}e_{0,v}=\left(a_1 \cdots a_r e_{1,v}\cdots e_{s,v} [u_1, v_1]\cdots [u_g,v_g]\right)^{-1}\\ a_{j}^{p_j}h^{q_j}=1, \,h\text{ central}\end{gathered}\right.\right>,\\
       \text{or } &\pi_1M_v=\left<\left.\begin{gathered}a_1,\ldots, a_r, u_1, \ldots, u_g, \\e_{0,v}, e_{1,v}, \ldots, e_{s,v},h_v\end{gathered}\right.\left|\,\begin{gathered}e_{0,v}=\left(a_1 \cdots a_r e_{1,v}\cdots e_{s,v} u_1^2\cdots u_g^2\right)^{-1}\\ a_{j}^{p_j}h^{q_j}=1, \,  [a_j,h]=1,\\  [e_{i,v},h]=1, \, u_khu_k^{-1}=h^{-1}\end{gathered}\right.\right>.
    \end{aligned}
$
\end{lrbox}

As a preparation for the proof of \autoref{THM: Graph mfd almost rigid}, we briefly review the numerical invariants for an oriented  graph manifold, and the readers are referred to \cite[Section 11.3]{FM13} for more details.

Fix a finite graph $\Gamma$ as the JSJ-graph of a graph manifold $M$. We assume that $\Gamma$ is not a single vertex, in which case the classification of graph manifolds reduces to the classification of Seifert fibered spaces. For each $v\in  V(\Gamma)$, fix an oriented Seifert fibered 3-manifold $M_v$. 
Then we choose a presentation of $\pi_1M_v$ as in \autoref{SEC: Seifert}:
\begin{equation*}
\scalebox{0.92}{\usebox{\tempequi}}
\end{equation*}
We additionally require that
\begin{enumerate}[leftmargin=*]
    \item $\{e_{i,v},h_v\}$ determines the  orientation on the boundary component induced by the given orientation of $M_v$;  
    \item $0<q_j<p_j$ for each $1\le j\le r$.
\end{enumerate}
Then $\{e_{i,v},h_v\}_{i=0}^{s}$ is called an admissible coordinate system on $\partial M_v$, following \cite[Definition 11.4]{FM13}.

Then, an oriented graph manifold $M$ with underlying JSJ-graph $\Gamma$ and corresponding Seifert pieces $M_v$ is determined by the gluing maps $\psi_e$ at each edge $e\in E(\Gamma)$. 
As in \autoref{subSEC: graph1}, the gluing map at $e\in E(\Gamma)$, 
up to isotopy, is determined by a matrix:
\begin{equation*}
\begin{tikzcd}[ampersand replacement=\&, column sep=large]
\psi_e:\pi_1(\partial_i M_u)=\Z\{h_u,e_{i,u}\} 
\arrow[r, "{\left(\begin{smallmatrix} \alpha_e & \beta_e\\ \gamma_e & \delta_e\end{smallmatrix}\right)}"] 
\& 
\Z\{h_v,e_{k,v}\}=\pi_1(\partial_k M_v)
\end{tikzcd}
\end{equation*}
where $u=d_0(e)$, $v=d_1(e)$, $\left(\begin{smallmatrix} \alpha_e & \beta_e\\ \gamma_e & \delta_e\end{smallmatrix}\right)\in \mathrm{GL}(2,\Z)$, and $\det \left(\begin{smallmatrix} \alpha_e & \beta_e\\ \gamma_e & \delta_e\end{smallmatrix}\right)=-1$ by choice of orientation.

Let the JSJ-torus $T^2_e$ inherit the boundary orientation of $M_v$ on the $d_1$--side, then $\gamma_e=i(h_u,h_v)\neq 0$ is the intersection number of the regular fibers at $T^2_e$.  
Similarly, $\delta_e=i(e_{i,u},h_v)$ and $\alpha_e=i(e_{k,v},h_u)$. And $\beta_e\in \Z$ is uniquely determined by $\alpha_e,\delta_e,\gamma_e$ from  $\det \left(\begin{smallmatrix} \alpha_e & \beta_e\\ \gamma_e & \delta_e\end{smallmatrix}\right)=-1$.

It is possible that different gluing matrices yield the same graph manifold. 
\cite[Theorem 11.2]{FM13} introduces a method to   classify   precisely the homeomorphism type of graph manifolds from the gluing matrices. For simplicity, we shall only need a partial result.

\begin{lemma}\label{LEM: Numerical invariants}
    Let $\overline{M}$ be a graph manifold obtained from the same oriented Seifert fibered pieces through  an alternate choice of  gluing matrices:  
\begin{equation*}
\begin{tikzcd}[ampersand replacement=\&, column sep=large]
\overline{\psi_e}:\pi_1(\partial_i M_u)=\Z\{h_u,e_{i,u}\} 
\arrow[r, "{\left(\begin{smallmatrix} \overline{\alpha_e} & \overline{\beta_e}\\ \overline{\gamma_e} & \overline{\delta_e}\end{smallmatrix}\right)}"] 
\& 
\Z\{h_v,e_{k,v}\}=\pi_1(\partial_k M_v).
\end{tikzcd}
\end{equation*}
Then $\overline{M}\cong M$ 
if the following conditions hold.
\begin{enumerate}[leftmargin=*]
    \item\label{dtcond1} $\overline{\gamma_e}=\gamma_e$, for each $e\in E(\Gamma)$;
    \item\label{dtcond2} $\overline{\delta_e}\equiv\delta_e\pmod{ \gamma_e}$, and $\,\overline{\alpha_e}\equiv\alpha_e\pmod{ \gamma_e}$, for each $e\in E(\Gamma)$;
    \item\label{inv3} for any $v\in V(\Gamma)$ such that $\# \partial M_v=deg_{\Gamma}(v)$  (i.e. $M_v\cap \partial M=\varnothing$), 
    $$\sum_{ e\in E(\Gamma):\,d_0(e)=v }\frac{\delta_e}{\gamma_e}-\sum_{e\in E(\Gamma):\,d_1(e)=v}\frac{\alpha_e}{\gamma_e}=\sum_{e\in E(\Gamma):\,d_0(e)=v}\frac{\overline{\delta_e}}{\overline{\gamma_e}}-\sum_{e\in E(\Gamma):\,d_1(e)=v}\frac{\overline{\alpha_e}}{\overline{\gamma_e}}.$$
\end{enumerate}
\end{lemma}
\begin{proof}
    We define an operation on the gluing matrices as follows.

(DT): 
Choose a vertex $v\in V(\Gamma)$. For each element in $D_v=\{(e,i)\in E(\Gamma)\times \{0,1\} \mid d_i(e)=v\}$,  assign an integer $t_{e,i}\in \Z$. If $\#\partial M_v=deg_{\Gamma}(v)$, we further require that $\sum_{(e,i)\in D_v}t_{e,i}=0$. First, for each $e\in E(\Gamma)$ so that $(e,0)\in D_v$, replace $\left(\begin{smallmatrix}
            \alpha_e&\beta_e\\\gamma_e&\delta_e
        \end{smallmatrix}\right)$ by $\left(\begin{smallmatrix}
            \alpha_e&\beta_e\\\gamma_e&\delta_e
        \end{smallmatrix}\right)\left(\begin{smallmatrix}
            1& t_{e,0}\\ 0 & 1
        \end{smallmatrix}\right)$. And then, for each $e\in E(\Gamma)$ so that $(e,1)\in D_v$, replace $\left(\begin{smallmatrix}
            \alpha_e&\beta_e\\\gamma_e&\delta_e
        \end{smallmatrix}\right)$ by $\left(\begin{smallmatrix}
            1& -t_{e,1}\\ 0 & 1
        \end{smallmatrix}\right)\left(\begin{smallmatrix}
            \alpha_e&\beta_e\\\gamma_e&\delta_e
        \end{smallmatrix}\right)$.
    
    \cite[Theorem 11.2]{FM13} implies that two choices of gluing matrices yield homeomorphic graph manifolds if one of them can be obtained from the other one through finite steps of operation \dt. In fact, {\dt} is realized by a vertical Dehn twist on the Seifert piece $M_v$ as shown in \autoref{LEM: Dehn twist}. The restriction $\sum t_{e,i}=0$ when $\#\partial M_v=deg_{\Gamma}(v)$ is the same restriction in the Dehn twist construction, as illustrated in \autoref{LEM: Dehn twist}; and this restriction is not needed when $\#\partial M_v>deg_{\Gamma}(v)$ since we can adjust the twisting on the extra boundary component in $\partial M\cap M_v$ accordingly.

    It is easy to verify that $\big\{\left(\begin{smallmatrix}
            \overline{\alpha_e}&\overline{\beta_e}\\\overline{\gamma_e}&\overline{\delta_e}
        \end{smallmatrix}\right)\big\}_{e\in E(\Gamma)}$ can be obtained from $\big\{\left(\begin{smallmatrix}
            \alpha_e&\beta_e\\\gamma_e&\delta_e
        \end{smallmatrix}\right)\big \}_{e\in E(\Gamma)}$ through finitely many steps of {\dt}    if and only if the conditions (\ref{dtcond1}), (\ref{dtcond2}), (\ref{inv3}) hold.
    %
\end{proof}

In fact, condition (\ref{inv3}) introduced in \autoref{LEM: Numerical invariants} is equivalently determined by a rational number 
\begin{align*}
    \tau(v)
    =& \sum_{e\in E(\Gamma):\,d_0(e)=v}\frac{\delta_e}{\gamma_e}-\sum_{e\in E(\Gamma):\,d_1(e)=v}\frac{\alpha_e}{\gamma_e}-\sum_{j}\frac{q_j}{p_j}\\
    =&\sum_{e\in E(\Gamma)\text{ adjoining } v}\frac{i(e_{i,v},h_w)}{i(h_v,h_w)}-\sum_{j}\frac{q_j}{p_j}.
\end{align*}
    In  the second line, $e\in E(\Gamma)$ with $d_0(e)=d_1(e)=v$ is count twice, and $w$ denotes the other endpoint of $e$. Note that the fraction $\frac{i(e_{i,v},h_w)}{i(h_v,h_w)}$ is independent with the orientation of the JSJ-torus as well as the orientation of $h_w$, so the fraction is well-defined. 
    
    The Seifert invariants $(p_j,q_j)$ for $M_v$ are fixed beforehand by the choice of fundamental group presentations. Thus, condition (\ref{inv3}) holds if and only if $\tau(v)=\overline{\tau}(v)$. The number $\tau(v)$ is also referred to as the {\em total slope} at $M_v$ 
    in \cite{Wil18}.  
    
Therefore, 
the (orientation-preserving) homeomorphism type of the graph manifold $M$ can be determined by a list of numerical invariants: 
\begin{equation}\label{new invariants}
    \begin{aligned}
        (1) \quad & \gamma_e\in \Z\text{ for each }e\in E(\Gamma);\\
        (2) \quad & \delta_e(\mathrm{mod}\, \gamma_e)\text{ and }\alpha_e(\mathrm{mod}\, \gamma_e)\text{ for each }e\in E(\Gamma);\\
        (3) \quad & \tau(v)\text{ for each }v\in \{v\in V(\Gamma)\mid \#\partial M_v=deg_{\Gamma}(v)\}.
    \end{aligned}
\end{equation}

Let us now return to the profinite isomorphism between graph manifolds.


\begin{lemma}\label{LEM: Total slope}
Let $M$ and $M'$ be oriented graph manifolds over the same JSJ-graph $\Gamma$ with the same oriented Seifert fibered pieces $M_v$. We fix the presentation of $\pi_1M_v$ as in the beginning of this subsection, and $M$ and $M'$ can be described by the gluing matrices.

Suppose $f:\widehat{\pi_1M}\ttt\widehat{\pi_1M'}$ is an isomorphism, so that the  map $\F:\Gamma\to \Gamma$ induced by  \autoref{THM: Profinite Isom up to Conj} is the identity map.  
Suppose $v\in V(\Gamma)$ is a vertex such that $\# \partial M_v= deg_{\Gamma}(v)$. 
Then, the corresponding total slopes are equal up to a $\pm$-sign: $\tau(v)=\pm \tau'(v)$.
\end{lemma}
\begin{proof}
    Let $f_{\bullet}$ denote the induced congruent isomorphism between the JSJ-graphs of the profinite groups. 
    Then, 
    $\# \partial M_v= deg_{\Gamma}(v)$ implies that $f_v$ in fact respects the peripheral structure, as is completely witnessed in the JSJ-graph. 
Therefore, by \autoref{PROP: Hempel pair}, there exist $\lambda,\mu \in \Zx$ and $\rho_i\in \widehat{\Z}$ such that $f_v(h_v)=h_{v}^\lambda$ and $f_v(e_{i,v})=g_ie_{i,v}^{o(g_i)\mu}g_i^{-1}h_v^{\rho_i}$. Here $g_i\in \widehat{\pi_1M_v}$, and $o$ is the orientation homomorphism whose kernel is the centralizer of $h_v$. 
    By composing with a conjugation, we have
    %
    $$
C_{g_i^{-1}}\circ f \begin{pmatrix}h_{v}&e_{i,v}\end{pmatrix} =\begin{pmatrix} h_{v} & e_{i,v}\end{pmatrix}\begin{pmatrix} o(g_i)\lambda & o(g_i)\rho_i\\ 0 & o(g_i)\mu\end{pmatrix}.$$

For any edge $e\in E(\Gamma)$ with $d_1(e)=v$, we have a commutative diagram as in (\ref{EQU: Matrix}).

\begin{equation*}
\begin{tikzcd}[ampersand replacement=\&, row sep=0.5cm,column sep=0.62cm]
{\overline{\pi_1\partial_jM_u}=\widehat{\Z}\left\{h_{u},e_{j,u}\right\}} \arrow[rrr, "{\left(\begin{smallmatrix} \alpha_e &\beta_e\\   \gamma_e&\delta_e\end{smallmatrix}\right)}"] \arrow[dd] \& \& \& {\widehat{\Z}\left\{h_{v},e_{i,v}\right\}=\overline{\pi_1\partial_iM_v}} \arrow[dd, "{\left(\begin{smallmatrix} o(g_i)\lambda & o(g_i)\rho_i\\ 0 & o(g_i)\mu\end{smallmatrix}\right)}", "C_{g_i}^{-1}\circ f_v"'] \\
                                                                                                                                                                                                              \& \& \&                                                                                                                              \\
{\overline{\pi_1\partial_jM_u}=\widehat{\Z}\left\{h_{u},e_{j,u}\right\}} \arrow[rrr, "{\left(\begin{smallmatrix} \alpha'_e &\beta'_e\\     \gamma'_e&\delta'_e\end{smallmatrix}\right)}"']                                                                                 \& \& \& {\widehat{\Z}\left\{h_{v},e_{i,v}\right\}=\overline{\pi_1\partial_iM_v}}                                                                                    
\end{tikzcd}
\end{equation*}
Then, \autoref{LEM: Graph mfd calculation} (\ref{9.2-3}) implies that 
\begin{equation}\label{9.10equ1}
\frac{\alpha'_e}{\gamma'_e}=\mu^{-1}\lambda \frac{\alpha_e}{\gamma_e}+\mu^{-1}\rho_i\in \Z^{-1}\widehat{\Z}.
\end{equation}

Similarly, for $e\in E(\Gamma)$ with $d_0(e)=v$ we have:
\begin{equation*}
\begin{tikzcd}[ampersand replacement=\&, row sep=0.5cm,column sep=0.62cm]
{\overline{\pi_1\partial_iM_v}=\widehat{\Z}\left\{h_{v},e_{i,v}\right\}} \arrow[rrr, "{\left(\begin{smallmatrix} \alpha_e &\beta_e\\   \gamma_e&\delta_e\end{smallmatrix}\right)}"] \arrow[dd,"{\left(\begin{smallmatrix} o(g_i)\lambda & o(g_i)\rho_i\\ 0 & o(g_i)\mu\end{smallmatrix}\right)}"', "C_{g_1}^{-1}\circ f_v"] \& \& \& {\widehat{\Z}\left\{h_{w},e_{k,w}\right\}=\overline{\pi_1\partial_kM_w}} \arrow[dd] \\
                                                                                                                                                                                                              \& \& \&                                                                                                                              \\
{\overline{\pi_1\partial_iM_v}=\widehat{\Z}\left\{h_{v},e_{i,v}\right\}} \arrow[rrr, "{\left(\begin{smallmatrix} \alpha'_e &\beta'_e\\     \gamma'_e&\delta'_e\end{smallmatrix}\right)}"']                                                                                 \& \& \& {\widehat{\Z}\left\{h_{w},e_{k,w}\right\}=\overline{\pi_1\partial_kN_w}}                                                                                    
\end{tikzcd}
\end{equation*}
It again follows from \autoref{LEM: Graph mfd calculation} (\ref{9.2-3}) that
\begin{equation}\label{9.10equ2}
    \frac{\delta_e'}{\gamma_e'}=\mu^{-1}\lambda\frac{\delta_e}{\gamma_e}-\mu^{-1}\rho_i\in \Z^{-1}\widehat{\Z}.
\end{equation}

We denote the Seifert invariants of $M_v$ as $(p_j,q_j)$. 
Since $f_v$ respects the peripheral structure, in $\Z^{-1}\widehat{\Z}$ we obtain:
\allowdisplaybreaks
\begin{align*}
    \tau'(v)=& {\textstyle \sum\limits_{e\in E(\Gamma):\,d_0(e)=v}}\frac{\delta_e'}{\gamma_e'}-{\textstyle\sum\limits_{e\in E(\Gamma):\,d_1(e)=v}}\frac{\alpha_e'}{\gamma_e'}-{\textstyle\sum\limits_{j}}\frac{q_j}{p_j}&\\
    =&\mu^{-1}\lambda \big( {\textstyle\sum\limits_{d_0(e)=v}}\frac{\delta_e}{\gamma_e}-{\textstyle\sum\limits_{d_1(e)=v}}\frac{\alpha_e}{\gamma_e} \big) -\mu^{-1}{\textstyle\sum\limits_{i}}\rho_i -{\textstyle\sum\limits_{j}}\frac{q_j}{p_j} & (\ref{9.10equ1}),\,(\ref{9.10equ2})\\
    =&\mu^{-1}\lambda \big( {\textstyle\sum\limits_{d_0(e)=v}}\frac{\delta_e}{\gamma_e}-{\textstyle\sum\limits_{d_1(e)=v}}\frac{\alpha_e}{\gamma_e} \big ) -\mu^{-1}\big (\lambda {\textstyle \sum\limits_{j} }\frac{q_j}{p_j}- \mu {\textstyle\sum\limits_{j}}\frac{q_j}{p_j}\big ) -{\textstyle\sum\limits_{j}}\frac{q_j}{p_j} & (\text{\autoref{LEM: Sum slope}})\\
    =&\mu^{-1}\lambda \big( {\textstyle\sum\limits_{d_0(e)=v}}\frac{\delta_e}{\gamma_e}-{\textstyle\sum\limits_{d_1(e)=v}}\frac{\alpha_e}{\gamma_e} -{\textstyle\sum\limits_{j}} \frac{q_j}{p_j} \big ) &\\
    =&\mu^{-1}\lambda\cdot  \tau(v) &
\end{align*}
\normalsize
\allowdisplaybreaks[0]

Note that $\tau(v), \tau'(v)\in \mathbb{Q}$ and $\mu^{-1}\lambda\in\Zx$, by multiplicating the denominators, it follows from \cite[Lemma 2.2]{Wil18} that either $\tau(v)=\tau'(v)=0$, or $\mu^{-1}\lambda=\pm 1$ and $\tau(v)=\pm \tau'(v)$. This finishes the proof of this lemma.
\end{proof}

We are  now ready to prove \autoref{THM: Graph mfd almost rigid}.
\begin{proof}[Proof of \autoref{THM: Graph mfd almost rigid}]
    Let $\mathscr{M}_G$ denote the class of  graph manifolds. We show that $\Delta_{\mathscr{M}_{G}}(M_0)$ is finite for any $M_0\in \mathscr{M}_{G}$.
    
    According to \autoref{THM: Profinite Isom up to Conj}, the profinite completion of the fundamental group of a graph manifold determines its JSJ-graph as well as the profinite completion of the fundamental group of the Seifert piece at each vertex. 
    Thus, any manifold in $\Delta_{\mathscr{M}_{G}}(M_0)$ can be viewed as a graph manifold over the same JSJ-graph $\Gamma$.
    We assume that $\Gamma$ is non-trivial, for the trivial case follows directly from \autoref{PROP: Seifert almost rigid}. 

    By \autoref{PROP: Seifert almost rigid}, the profinite completion of the fundamental group of each Seifert piece determines the homeomorphism type of each Seifert piece  to finitely many possibilities. 
    Thus, $\Delta_{\mathscr{M}_{G}}(M_0)$ can be divided into finitely many subsets $\Delta_1,\cdots, \Delta_n$ such that for any $M,M'\in \Delta_i$, we can choose orientations on $M$ and $M'$, which determines the orientation on their Seifert pieces, so that  $M_v\cong M'_v$ by an orientation preserving homeomorphism for each $v\in V(\Gamma)$.

    It suffices to show that each $\Delta_i$ is finite. Within $\Delta_i$, we can fix the presentation of each $\pi_1M_v$ as in the beginning of this subsection, and describe the graph manifolds through the gluing matrices. 
Then, within $\Delta_i$, \autoref{LEM: Graph mfd calculation}~(\ref{9.2-1}) implies that the profinite completion determines the intersection number  of the regular fibers $\gamma_e$ at each $e\in E(\Gamma)$ up to $\pm$-sign. Moreover, \autoref{LEM: Total slope} implies that the profinite completion determines, up to $\pm$-sign, the total slope $\tau(v)$ of each vertex $v$ such that $\#\partial M_v=deg_{\Gamma}(v)$. Note that when the intersection number $\gamma_e\in \Z\setminus\{0\}$ is fixed, there are only finitely many choices of $\delta_e(\mathrm{mod}\, \gamma_e)$ and $\alpha_e(\mathrm{mod}\, \gamma_e)$. Therefore, within $\Delta_i$, there are only finitely many choices for the numerical invariants listed in (\ref{new invariants}). 
    This implies the finiteness of $\Delta_i$, finishing the proof. 
\end{proof}

\subsection{From peripheral $\Zx$-regularity to homeomorphism}\label{subsec: Zx imply homeo}

Although graph manifolds are not profinitely rigid, even among the closed ones \cite{Wil18}, an additional condition of peripheral $\Zx$-regularity will imply certain rigidity. 
This subsection is devoted to the proof of \autoref{inthm: Graph new}~(\ref{di222}), elaborated below as \autoref{THM: Graph Zx-regular implies homeo}. The proof is based on \autoref{PROP: Seifert Zx-regular}, through joining together homeomorphisms of the Seifert fibered pieces according to the gluing lemma (\autoref{inLem1'}).

\begin{theorem}\label{THM: Graph Zx-regular implies homeo}
Let $M=(M_\bullet, \Gamma)$ and $N=(N_\bullet,\Gamma')$ be 
 graph manifolds with non-empty toral boundary. 
Suppose $f:\widehat{\pi_1M}\ttt\widehat{\pi_1N}$ is an isomorphism respecting the peripheral structure, 
and that $f$ is 
peripheral $\widehat{\Z}^{\times}$-regular at some of the boundary components $\partial_1 M,\cdots, \partial_k M$, where $1\le k\le \#\partial M$. Then, 
\begin{enumerate}[leftmargin=*]
\item $M$ and $N$ are homeomorphic;
\item 
there exists an element  $\lambda\in \widehat{\Z}^{\times}$ and two  homeomorphisms $\Phi^\pm:M\to N$ such that 
for each $1\le i\le k$, $f$ restricting on $\partial_iM$ is induced by  homeomorphism $\Phi^\pm$ with coefficient $\pm\lambda$, where the $\pm$-signs are consistent.
\end{enumerate}
\end{theorem}


\begin{proof}
Let $f_\bullet$  be the congruent isomorphism between the JSJ-graphs of profinite groups $(\widehat{\mathcal{G}_M},\Gamma)$ and $(\widehat{\mathcal{G}_N},\Gamma')$ given by \autoref{THM: Profinite Isom up to Conj}. According to \autoref{COR: peripheral preserving}, since $f$ respects the peripheral structure, for each $v\in V(\Gamma)$, $f_v:\widehat{\pi_1M_v}\ttt \widehat{\pi_1N_{\F(v)}}$ also respects the peripheral structure. 
%
Thus, by \autoref{PROP: Hempel pair}, for each $v\in V(\Gamma)$,  we can define the scale type of $f_v$ as $(\lambda_v,\mu_v)\in (\Zx)^2$, up to  $\pm$-signs. 

We claim that there exists a unified $\lambda\in \Zx$, such that for each $v\in V(\Gamma)$, $\lambda_v=\pm \lambda$ and $\mu_v=\pm \lambda$. First, choose a vertex $v_0\in V(\Gamma)$ such that $M_{v_0}$ contains one of  $\partial_1M,\cdots, \partial_kM$. Then $f_{v_0}$ is peripheral $\Zx$-regular at this boundary component, according to \autoref{EX: peripheral regular}. As a consequence of \autoref{EX: Boundary map}, $f_{v_0}$ has scale type $(\pm\lambda,\lambda)$ for some $\lambda\in \Zx$. For any two adjacent vertices $u,v\in V(\Gamma)$ (i.e. connected by an edge), \autoref{LEM: Graph mfd calculation}~(\ref{9.2-2}) implies $\lambda_u=\pm \mu_v$ and $\mu_u=\pm \lambda_v$. Thus, starting from $v_0$, the connectedness of $\Gamma$ implies that $\lambda_v=\pm \lambda$ and $\mu_v=\pm \lambda$ for all $v\in V(\Gamma)$.

\def\especial{selected }

Therefore, \autoref{LEM: Graph mfd calculation} (\ref{9.2-4}) implies that for each $v\in V(\Gamma)$, $f_v$ is peripheral $\lambda$-regular at the boundary components obtained from the JSJ-tori corresponding to the edges adjoining $v$.  
In addition, according to \autoref{EX: peripheral regular}, each $f_v$ is also peripheral $\lambda$-regular at any of the listed boundary components $\partial_1M,\cdots,\partial_kM$ which belongs to $M_v$. Let these two kinds of boundary components be the ``\especial  boundary components'' of $M_v$. 
Then, each $M_v$ contains at least one \especial boundary component, and $f_v$ is peripheral $\lambda$-regular at all the  \especial 
boundary components. 

Recall that $f_v$ also respects the peripheral structure. Thus, \autoref{PROP: Seifert Zx-regular} implies $M_v\cong N_{\F(v)}$ for each $v\in V(\Gamma)$, and there exists a pair of homeomorphisms $\Phi^{\pm}_{v}$ such that   $f_v$ restricting on each \especial boundary component of $M_v$ is induced by the homeomorphism $\Phi^{\pm}_{v}$ with coefficient $\pm\lambda$. 
In particular, concerning those boundary components corresponding to the JSJ-tori, the gluing lemma, \autoref{inLem1'}~(\ref{lemmain-2'}), implies that the homeomorphism $\{\Phi^+_v\}_{v\in V(\Gamma)}$ are compatible with the gluing maps at all JSJ-tori. Thus, combining these homeomorphisms yields a homeomorphism $\Phi^+: M\to N$. 

In addition, for $1\le i\le k$, suppose $\partial_iM$ belongs to a JSJ-piece $M_v$. According to diagram (\ref{6.10dig}), $f$ restricting on $\overline{\pi_1\partial_iM}$ is equivalent to $f_v$ restricting on $\overline{\pi_1\partial_iM}$ composing with a conjugation in $\widehat{\pi_1M}$. Since $f_v$ restricting on $\partial_iM$ is  induced by the homeomorphism $\Phi^+_v=\Phi^+|_{M_v}$ with coefficient $\lambda$, it follows by definition that $f$ restricting on $\partial_iM$ is  induced by the homeomorphism $\Phi^+ $ with coefficient $\lambda$.

Similarly, the homeomorphisms $\{\Phi^-_v\}_{v\in V(\Gamma)}$ are also compatible with the gluing maps at all JSJ-tori. Combining these homeomorphisms yields a homeomorphism $\Phi^-: M\to N$, where $f$ restricting on each $\partial_iM$ ($1\le i\le k$) is  induced by the homeomorphism $\Phi^- $ with coefficient $-\lambda$.
\end{proof}

\subsection{Profinite almost auto-symmetry for graph manifolds}\label{SEC: Hgroup}
Generally, a profinite isomorphism between graph manifolds with boundary may not respect the peripheral structure, as for an isomorphism between their fundamental groups. In this case, \autoref{THM: Graph Zx-regular implies homeo} is no longer applicable. In order to prove the profinite almost rigidity in the bounded case, we show that graph manifolds  satisfy the {\PAA} property.

\begin{theorem}\label{THM: Hgroup finite index}
Let $M=(M_\bullet,\Gamma)$ be a graph manifold with non-empty boundary, and let $\mathcal{P}$ be a collection of some boundary components of $M$. 
Then, $M$ is {\PAA} at $\mathcal{P}$.
\end{theorem}


\begin{proof}
The proof of this theorem consists of the following two cases.

\textbf{Case 1:} 
Suppose that each JSJ-piece containing a boundary component in $\mathcal{P}$ also contains a boundary component in $\partial M\setminus \mathcal{P}$.

In this case, $M$ is not the minor Seifert manifold. Denote $\mathcal{P}=\{\partial_1M,\cdots, \partial_kM\}$, and fix conjugacy representatives of peripheral subgroups $\pi_1\partial_1M,\cdots,\pi_1\partial_kM$. Then \autoref{PROP: peripheral structure} implies that there is a well-defined group homomorphism
\begin{equation*}
\begin{tikzcd}[row sep=tiny]
\mathcal{F}:\quad  \XX_{\Zx}(M,\mathcal{P}) \arrow[r] & \prod\limits_{i=1}^{k} \cfrac{\Aut(\pi_1\partial_iM)}{\{\pm id\}}\cong (\mathrm{PGL}_2(\Z))^k\\
\;\;f \arrow[r, maps to] & ([\varphi_1],\cdots, [\varphi_k])
\end{tikzcd}
\end{equation*}
such that for each $1\le i\le k$, $C_{g_i}\circ f|_{\overline{\pi_1\partial_iM}}=\lambda_i\otimes \varphi_i$ for some $g_i\in \widehat{\pi_1M}$ and $\lambda_i\in \Zx$.


Suppose each $\partial_iM$ belongs to a JSJ-piece $M_{v_i}$. According to diagram (\ref{6.10dig}), $f$ restricting on $\overline{\pi_1\partial_{i}M}$ is equivalent to the induced isomorphism $f_{v_i}$ restricting on $\overline{\pi_1\partial_iM}$ composing with a conjugation in $\widehat{\pi_1M}$.
Let $\{h_i,e_i\}$ be the free $\Z$-basis of each $\pi_1\partial_iM$, where $h_i$ generates the fiber subgroup of $M_{v_i}$, and $e_i$ is a peripheral cross-section. 
Then according to \autoref{EX: Boundary map}, any element in the image of $\mathcal{F}$ 
can be represented by matrices $$\left(\left[ \begin{pmatrix} \pm 1 & t_1 \\ 0 & \pm1 \end{pmatrix} \right],\cdots,\left[ \begin{pmatrix} \pm 1 & t_k \\ 0 & \pm1 \end{pmatrix} \right]\right)$$ over these bases. 
Consequently, there exists a subgroup $K\le \im(\mathcal{F})$ with index at most $2^k$ such that matrices in $K$ have identical $\pm$-signs in the diagonal elements. In other words, any element in $K$ can be represented by matrices $$\left(\left[ \begin{pmatrix}   1 & t_1 \\ 0 &  1 \end{pmatrix} \right],\cdots,\left[ \begin{pmatrix} 1 & t_k \\ 0 &  1 \end{pmatrix} \right]\right).$$

It suffices to show that $\mathcal{F}^{-1}(K)\subseteq \XX_{h}(M,\mathcal{P})$. Indeed, this can be realized through vertical Dehn twists between each  component  in $\mathcal{P}$ with an extra boundary component in $\partial M\setminus \mathcal{P}$ in the corresponding JSJ-piece. 

Let $f\in \mathcal{F}^{-1}(K)$, and $\mathcal{F}(f)=\left(\left[ \left( \begin{smallmatrix}   1 & t_1 \\ 0 &  1 \end{smallmatrix}\right) \right],\cdots,\left[\left( \begin{smallmatrix} 1 & t_k \\ 0 &  1 \end{smallmatrix}\right) \right]\right)$. 
Suppose $M_v$ is a JSJ-piece that contains a  component in $\mathcal{P}$, and denote $\mathcal{P}\cap M_v= \partial_{i_1}M,\cdots, \partial_{i_s}M$. Let $\partial'M_v\in \partial M\setminus \mathcal{P}$ be an extra boundary component in $M_v$ guaranteed by our assumption.  Then \autoref{LEM: Dehn twist} implies that there exists a vertical  Dehn twist $\Phi_v$ along $\partial_{i_1}M, \cdots, \partial_{i_s}M$ and $\partial'M_v$, which induces 
the automorphism represented by $\left(\begin{smallmatrix} 1 & t_{i_j} \\ 0 & 1\end{smallmatrix}\right)$ on $\pi_1\partial_{i_j}M$,  
and $\Phi_v$ is meanwhile the identity map on each boundary component of $M_v$ obtained from the JSJ-tori. For those Seifert pieces $M_u$ containing no distinguished boundary components, simply take $\Phi_u=Id$. Then the homeomorphisms $\{\Phi_w\}_{w\in V(\Gamma)}$ are all identity maps at the JSJ-tori. Hence, they are compatible, and yields a homeomorphism $\Phi$ on $M$, so that $f$ restricting on each $\partial_iM\in \mathcal{P}$ is induced by $\Phi$ with certain coefficient in $\Zx$. In other words, $f\in \XX_{h}(M,\mathcal{P})$.

\textbf{Case 2:} There exists a JSJ-piece $M_{v_0}$, such that $M_{v_0}$ contains a boundary component in $\mathcal{P}$, but  contains no extra boundary components in $\partial M\setminus \mathcal{P}$. 

We  prove this case through induction on $\#V(\Gamma)$.

When $\#V(\Gamma)=1$, it follows from our assumption that $\mathcal{P}=\partial M$. Thus, any $f\in \XX_{\Zx}(M,\mathcal{P})$ respects the peripheral structure, and it follows from \autoref{THM: Graph Zx-regular implies homeo} that $\XX_{\Zx}(M,\mathcal{P})=\XX_{h}(M,\mathcal{P})$. 

Now we assume $\#V(\Gamma)>1$. 
%
Let $\Gamma_{\text{uo}}$ be  the unoriented underlying graph of $\Gamma$. There is a group homomorphism
\begin{equation*}
\begin{tikzcd}[row sep=tiny]
\rho:\quad  \XX_{\Zx}(M,\mathcal{P}) \arrow[r] & \Aut(\Gamma_{\text{uo}})\\
\;\;f \arrow[r, maps to] & \F
\end{tikzcd}
\end{equation*}
where $\F$ is the graph isomorphism within the congruent isomorphism $f_\bullet$ induced by $f$, according to \autoref{THM: Profinite Isom up to Conj}. Let $$\XX^{0}_{\Zx}(M,\mathcal{P})= \ker(\rho),\;  \text{and }   \XX^0_h(M,\mathcal{P})= \XX_h(M,\mathcal{P})\cap \XX^{0}_{\Zx}(M,\mathcal{P}).$$ Note that $\Gamma$ is a finite graph, so $\Aut(\Gamma_{\text{uo}})$ is a finite group, and $\XX^{0}_{\Zx}(M,\mathcal{P})$ is finite-index in $\XX_{\Zx}(M,\mathcal{P})$. Thus, it suffices to show that $ \XX^0_h(M,\mathcal{P})$ is finite-index in $\XX^{0}_{\Zx}(M,\mathcal{P})$.

Let $\Upsilon$ be the full subgraph of $\Gamma$ spanned by the vertex $v_0$. In other words, $V(\Upsilon)=\{v_0\}$, and $E(\Upsilon)=\{e\in E(\Gamma)\mid d_0(e)=d_1(e)=v_0\}$. Let $M_{\Upsilon}$ be the submanifold of $M$ corresponding to the subgraph $\Upsilon$, which is obtained from $M_{v_0}$ by gluing some of its boundary components.   
Let $\Gamma_1,\cdots,\Gamma_r$ be the connected components of $\Gamma\setminus \Upsilon$, and let $M_{\Gamma_j}$ be the submanifold of $M$ corresponding to the subgraph $\Gamma_j$.  
We label the oriented edges connecting $\Upsilon $ and $\Gamma\setminus \Upsilon$ as $\{e_1,\cdots, e_s\}$. For each $1\le k\le s$ and  $\alpha\in \{0,1\}$, let $\delta_\alpha(e_k)=\Upsilon\text{ or }\Gamma_j$ if $d_{\alpha}(e_k)\in \Upsilon\text{ or }\Gamma_j$, and let $\partial_{\alpha,e_k}  M_{\delta_\alpha(e_k)}$ be the corresponding  boundary component of $M_{\delta_\alpha(e_k)}$ obtained from the JSJ-torus. 

For any $f\in \XX^{0}_{\Zx}(M,\mathcal{P})$, let $f_\bullet$ be the congruent isomorphism on $(\widehat{\mathcal{G}_M},\Gamma)$. Then there are profinite isomorphisms
\begin{equation*}
\begin{gathered}
f_{\Upsilon}: \widehat{\pi_1M_{\Upsilon}}\cong \Pi_1(\widehat{\mathcal{G}_M},\Upsilon)\ttt \Pi_1(\widehat{\mathcal{G}_M},\Upsilon)\cong  \widehat{\pi_1M_{\Upsilon}},\quad\text{and}\\
f_{\Gamma_j}: \widehat{\pi_1M_{\Gamma_j}}\cong \Pi_1(\widehat{\mathcal{G}_M},\Gamma_j)\ttt \Pi_1(\widehat{\mathcal{G}_M},\Gamma_j)\cong  \widehat{\pi_1M_{\Gamma_j}}
\end{gathered}
\end{equation*}
induced by $f_\bullet$ on the subgraphs $\Upsilon$ and  $\Gamma_j$ as shown in \autoref{LEM: congruent isom induce isom}. 

Our assumption implies that  $f_{v_0}$ and $ f_{\Upsilon}$ respect  the peripheral structure, as the peripheral subgroups obtained from JSJ-tori are witnessed by the congruent isomorphism $f_\bullet$, and the  peripheral subgroups corresponding to $M_{v_0}\cap \mathcal{P}$ are preserved by $f$, and thus preserved by $f_{v_0}$  and $ f_{\Upsilon}$ according to \autoref{lem: 111111}. Moreover, \autoref{EX: peripheral regular} implies that $f_{v_0}$ is peripheral $\Zx$-regular at each component  of  $M_{v_0}\cap \mathcal{P}$, and $f_{\Gamma_j}$ is peripheral $\Zx$-regular at each component  of  $M_{\Gamma_j}\cap \mathcal{P}$. 
In particular, $f_{v_0}$ has scale type $(\pm\lambda,\lambda)$ for some $\lambda\in \Zx$. Then, \autoref{LEM: Graph mfd calculation}~(\ref{9.2-4}) implies that $f_{\Upsilon}$ and $f_{\Gamma_j}$ are all peripheral $\lambda$-regular at the boundary components given by the JSJ-tori corresponding to $e_1,\cdots, e_s$. 
Let $$\mathcal{P}_j= (\mathcal{P}\cap M_{\Gamma_j})\cup (\partial M_{\Gamma_j}\cap \partial M_{v_0})=(\mathcal{P}\cap M_{\Gamma_j})\cup \left\{\partial_{\alpha,e_k} M_{\Gamma_j}\mid \delta_\alpha(e_j)=\Gamma_j\right\}$$ be a collection of boundary components in $M_{\Gamma_j}$. 
Then, $$f_{\Upsilon}\in \XX_{\Zx}(M_{\Upsilon},\partial M_{\Upsilon}),\; \text{and } f_{\Gamma_j}\in \XX_{\Zx}(M_{\Gamma_j},\mathcal{P}_j).$$

Indeed, $f_{\Upsilon}\in \XX_{h}(M_{\Upsilon},\partial M_{\Upsilon})$ according to \autoref{THM: Graph Zx-regular implies homeo}. On the other hand, note that $M_{\Gamma_j}$ has less JSJ-pieces than $M$, so $(M_{\Gamma_j},\mathcal{P}_j)$ either falls into {Case 1}, or falls into a proven case in {Case 2}. 
Thus, the induction hypothesis implies that $M_{\Gamma_j}$ is {\PAA} at $\mathcal{P}_j$,  
i.e. $\XX_h(M_{\Gamma_j},\mathcal{P}_j)$ is finite index in $\XX_{\Zx}(M_{\Gamma_j},\mathcal{P}_j)$. 

Therefore, we can find finitely many representatives $f^\star_1,\cdots, f^\star_n\in \XX_{\Zx}^0(M,\mathcal{P})$ such that for any $f\in \XX_{\Zx}^0(M,\mathcal{P})$, there exists $1\le i\le n$ such that $((f^\star_i)^{-1}\circ f)_{\Gamma_j}\in \XX_h(M_{\Gamma_j},\mathcal{P}_j)$ for each $1\le j\le r$. 
  Denote $f'=(f^\star_i)^{-1}\circ f$, and there exist  homeomorphisms
$$
\Phi_{\Upsilon}^{f'}: M_{\Upsilon}\ttt M_{\Upsilon}\;\text{and }  \Phi_{\Gamma_j}^{f'}: M_{\Gamma_j}\ttt M_{\Gamma_j}
$$
such that 
\begin{enumerate}[leftmargin=*]
\item for each $\partial_iM\in \mathcal{P}$, $f'_x$ restricting on $\partial_iM$ is induced by the homeomorphism $\Phi_x^{f'}$ with certain coefficient in $\Zx$, where $x\in \{\Upsilon,\Gamma_1,\cdots, \Gamma_r\}$ such that $\partial_iM\subseteq M_x$; 
\item for each $1\le k\le s$ and  $\alpha\in \{0,1\}$, $f_{\delta_\alpha(e_k)}'$ restricting on $\partial_{\alpha,e_k}  M_{\delta_{\alpha}(e_k)}$ is induced by the homeomorphism $\Phi_{\delta_\alpha(e_k)}^{f'}$ with coefficient $\lambda_{\alpha,k}^{f'}\in\Zx$.  
\end{enumerate}
Then according to \autoref{Cor: gluing}~(\ref{lemmain-1''}), $\lambda_{0,k}^{f'}=\pm\lambda_{1,k}^{f'}$ for each $1\le k \le s$.

However, in order to apply the gluing lemma, it is  required that $\lambda_{0,k}^{f'}= \lambda_{1,k}^{f'}$. Indeed, one can choose at most $2^s$ representatives
\begin{equation*}
\left\{g_1,\cdots, g_m\right\}\subseteq \left\{ g\in \XX^0_{\Zx}(M,\mathcal{P})\mid g_{\Gamma_j}\in \XX_{h}(M_{\Gamma_j},\mathcal{P}_j),\,\forall 1\le j\le r \right\},
\end{equation*}
and construct more representatives 
\begin{equation*}
\left\{f_1^{\star\star},\cdots, f_N^{\star\star}\right\}=\left\{f_{i}^\star\circ g_{l}\mid 1\le i \le n,\, 1\le l \le m\right\} \subseteq \XX_{\Zx}^0(M,\mathcal{P})
\end{equation*}
such that for any $f\in \XX_{\Zx}^0(M,\mathcal{P})$, there exists $1\le i\le N$ such that $f''=(f^{\star\star}_i)^{-1}\circ f$ satisfies $f''_{\Gamma_j}\in \XX_h(M_{\Gamma_j},\mathcal{P}_j)$ for each $1\le j\le r$, and $\lambda_{0,k}^{f''}=\lambda_{1,k}^{f''}$ for each $1\le k\le s$. 

In this case, according to \autoref{Cor: gluing}~(\ref{lemmain-2''}), the homeomorphisms $\Phi_{\Upsilon}^{f''}, \Phi_{\Gamma_1}^{f''},\cdots, \Phi_{\Gamma_r}^{f''}$ are compatible at the JSJ-tori, and yield an ambient homeomorphism $\Phi^{f''}: M\to M$. In addition, for each $\partial_iM\in \mathcal{P}$, let $x\in \{\Upsilon,\Gamma_1,\cdots, \Gamma_r\}$ such that $\partial_iM\subseteq M_x$.  According to diagram (\ref{6.10dig}), $f''$ restricting on $\overline{\pi_1\partial_iM}$ is equivalent to $f_x''$ restricting on $\overline{\pi_1\partial_iM}$ composing with a conjugation in $\widehat{\pi_1M}$. Since $f_x''$ restricting on $\partial_iM$ is  induced by the homeomorphism $\Phi_{x}^{f''}=\Phi^{f''}|_{M_x}$ with certain coefficient, it follows  that $f''$ restricting on $\partial_iM$ is  induced by the homeomorphism $\Phi^{f''}$ with certain coefficient in $\Zx$. In other words, $f''\in \XX_{h}(M,\mathcal{P})$. 
Therefore, $[\XX_{\Zx}^{0}(M,\mathcal{P}):\XX_{h}^0(M,\mathcal{P})]\le N<+\infty$, finishing the proof. 
\end{proof}

\section{Profinite almost rigidity for 3-manifolds}\label{SEC: Main}

\subsection{Profinite rigidity in the Seifert part}\label{subsec: Seifert rigid}
Let us first prove \autoref{Mainthm: SF rigid}.

\normalsize

\begin{THMA}
Suppose $M$ and $N$ are mixed 3-manifolds with empty or incompressible toral boundary, and $f:\widehat{\pi_1M}\ttt \widehat{\pi_1N}$ is an isomorphism respecting the peripheral structure. Then the Seifert parts of $M$ and  $N$ are homeomorphic.
\end{THMA}

\begin{proof}
According to \autoref{THM: Profinite Isom up to Conj}, the isomorphism $f:\widehat{\pi_1M}\ttt\widehat{\pi_1N}$ induces a congruent isomorphism $f_\bullet$ between the JSJ-graphs of profinite groups $(\widehat{\mathcal{G}_M},\Gamma_M)$ and $(\widehat{\mathcal{G}_N},\Gamma_N)$. According to \autoref{COR: peripheral preserving}, for every $v\in V(\Gamma_M)$, the isomorphism $f_v:\widehat{\pi_1M_v}\ttt\widehat{\pi_1N_{\F(v)}}$ respects the peripheral structure. In addition,   $\F$ preserves the hyperbolic or Seifert type of the vertices. Thus, $\F$ witnesses an isomorphism between the full subgraphs spanned by the Seifert vertices, and   establishes a one-to-one correspondence between the connected components of the Seifert parts of $M$ and $N$. It suffices to prove that the corresponding components of the Seifert parts are homeomorphic. 

For any connected component $M_0$ of the Seifert part $M$, which corresponds to a connected component $\Gamma_0$ of the Seifert subgraph of $\Gamma_M$, let $N_0$ denote the connected component of the Seiferft part of $N$ corresponding to the subgraph $\F(\Gamma_0)$. Then, by \autoref{LEM: congruent isom induce isom}, $f_\bullet$ establishes an isomorphism $$f_0:\widehat{\pi_1M_0}\cong \Pi_1(\widehat{\mathcal{G}_M},\Gamma_0)\ttt\Pi_1(\widehat{\mathcal{G}_N},\F(\Gamma_0))\cong \widehat{\pi_1N_0}$$ which also respects the  peripheral structure according to \autoref{COR: peripheral preserving}.

Since $M$ and $N$ are mixed manifolds, one can choose an edge $e\in E(\Gamma_M)$ which connects a vertex $v\in V(\Gamma_0)$ with a hyperbolic vertex $u\in V(\Gamma)$. 
The boundary gluing at the JSJ-tori corresponding to $e$ and $\F(e)$ can be described as
\begin{equation*}
\begin{gathered}
  \begin{tikzcd}
M_u\supseteq \partial_iM_u & T^2_e \arrow[l] \arrow[r] & \partial_kM_0\subseteq M_v\subseteq M_0,
\end{tikzcd}\\
  \begin{tikzcd}
N_{\F(u)}\supseteq \partial_jN_{\F(u)} & T^2_{\F(e)}  \arrow[l] \arrow[r] & \partial_lN_0 \subseteq N_{\F(v)}\subseteq N_0.
\end{tikzcd}
\end{gathered}
\end{equation*}

Since $M_u$ and $N_{\F(u)}$ are hyperbolic, \autoref{THM: Hyperbolic peripheral regular} implies that $f_u$ is peripheral $\Zx$-regular at the boundary component $\partial_iM_u$. Thus, according to \autoref{Cor: gluing}~(\ref{lemmain-1''}), $f_0$ is also peripheral $\Zx$-regular at $\partial_kM_0$. Recall that $f_0$ respects the peripheral structure. 
Hence, \autoref{THM: Graph Zx-regular implies homeo} implies that $M_0$ and $N_0$ are homeomorphic.
\end{proof}


\subsection{Lack of rigidity without respecting peripheral structure}\label{SEC: Counter-example}

Given the fundamental group  $G$ of an irreducible 3-manifold with incompressible  toral boundary,  when $G$ is marked with its peripheral structure, it determines the unique homeomorphism type of the manifold as shown by \autoref{Wald}. Without marking of the peripheral structure, the homeomorphism type of 3-manifolds with this fundamental group may not be unique, for instance, when the manifold is not purely hyperbolic. Yet, in this case, the isomorphism type of $G$ is enough to determine the JSJ-decomposition group theoretically by \cite[Theorem 3.1]{SS01}, that is, it determines the JSJ-graph of groups up to congruent isomorphism. In particular, the isomorphism type of $G$ uniquely determines the isomorphism type of the fundamental groups of the Seifert part of the corresponding manifold.

Unfortunately, this fails to hold when $G$ is replaced by merely $\mathcal{C}(G)$. Although $\widehat{G}$ marked with the peripheral structure, or equivalently $\mathcal{C}(G;\pi_1\partial_1M,\cdots, \pi_1\partial_nM)$, determines the homeomorphism type of the Seifert part within mixed manifolds, $\widehat{G}$, or equivalently $\mathcal{C}(G)$, itself may not even determine the fundamental group of the Seifert part as shown by the following construction.

\begin{THMB}
There exists a pair of irreducible, mixed 3-manifolds $M$ and $N$ with incompressible toral boundary, such that $\widehat{\pi_1M}\cong \widehat{\pi_1N}$, while the fundamental groups of the  Seifert part of $M$ and $N$ are not isomorphic.
\end{THMB}
\begin{proof}
We shall construct examples over the JSJ-graph $\Gamma$ shown in \autoref{JSJ-graph}. 

\begin{figure}[ht]
\includegraphics[width=2.9in, alt={Three vertices in a row, labeled u,v,w in order, connected by two edges between them forming a tree.}]{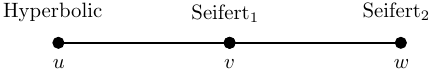}
\caption{{JSJ-graph $\Gamma$ for a non-rigid example}}
\label{JSJ-graph}
\end{figure}

Let $M_u=N_u=X$ be a finite-volume hyperbolic 3-manifold. Let $F$ be a compact, orientable surface with at least three boundary components, and let \({M_v=N_v=F\times S^1=Y}\). 
Let $M_w$ be the Seifert fibered space $(S;(p,q))$ and $N_w$ be the Seifert fibered space $(S;(p,q'))$, such that:
\begin{enumerate}[leftmargin=*]
\item the compact surface $S$ has exactly one boundary component;
\item $\chi(S)<0$, so that $M_w$ and $N_w$ are major Seifert manifolds;
\item $q'\equiv \kappa q\pmod p$ for some $\kappa\in \Zx$, i.e. $(M_w,N_w)$ is a Hempel pair of scale factor~$\kappa$;
\item $q'\not\equiv \pm q\pmod p$, so that $M_w$ and $N_w$ are not homeomorphic.
\end{enumerate}


Let us first clarify the peripheral subgroups of these JSJ-pieces.  For the hyperbolic piece $X$, fix a boundary component $T\subseteq \partial X$ and a conjugacy representative of peripheral subgroup $\pi_1T$. For the intermediate manifold $Y=F\times S^1$, let $e_1,e_2\in \pi_1F$ represent two boundary components of $F$, and let $h$ denote a generator of $\pi_1(S^1)=\Z$. This defines two peripheral subgroups $\pi_1\partial_1Y=\Z\{h,e_1\}$ and $\pi_1\partial_2Y=\Z\{h,e_2\}$ for which we shall construct the gluing. 
For the Seifert fibered space $M_w$, let $h_w$ represent a regular fiber in $\pi_1M_w$, and let $e_w$  denote a peripheral cross-section with appropriate orientation yielding the Seifert invariants. Then, let $\pi_1\partial M_w=\Z\{h_w,e_w\}$ be a conjugacy representative of the peripheral subgroup. Likewise, let $h_w'$ represent a regular fiber in $\pi_1N_w$, let $e'_w$  represent the appropriate   peripheral cross-section, and let $\pi_1\partial N_w=\Z\{h_w',e'_w\}$ be a conjugacy representative of the peripheral subgroup.

Now we define $M$ and $N$ through the gluing maps at these peripheral subgroups. First, fix a gluing map $\psi:\pi_1T\ttt\pi_1 \partial_1Y=\Z\{h,e_1\}$. Second, since $p$ and $q$ (resp. $p$ and $q'$) are coprime, we can choose 
$$A=\begin{pmatrix} -q & \ast  \\ p & \ast \end{pmatrix} \in \mathrm{GL}(2,\Z) \quad\text{and}\quad B= \begin{pmatrix} -q' & \ast  \\ p & \ast \end{pmatrix} \in \mathrm{GL}(2,\Z)$$
 such that $\det A=\det B=-1$.

Let $M$ be defined by the gluing maps:
\normalsize
\begin{equation*}
\begin{tikzcd}[ampersand replacement=\&, column sep=tiny,row sep=1em ]
\pi_1T \arrow[d, symbol=\subseteq] \arrow[rrrr, "\psi"] \&  \&  \&  \& \pi_1\partial_1Y \arrow[d, symbol=\subseteq] \&                                \&   \&  \& \& \&  \&  \&                   \&  \& \\
\pi_1X        \arrow[rrrr,dash,dashed]                           \&  \&  \&  \& \pi_1Y \arrow[d, symbol=\supseteq]          \arrow[rrrrrrrrrr,dash,dashed]    \&                                   \&  \&  \&  \& \&  \& \& \&  \& \pi_1M_w \arrow[d, symbol=\supseteq] \\
                                      \&  \&  \&  \& \pi_1\partial_2Y \arrow[r, symbol=\mathop{=}] \& {\Z\{h,e_2\}} \arrow[rrrrrrrr, "{A=\left(\begin{smallmatrix} -q & \ast  \\ p & \ast \end{smallmatrix}\right)}"'] \&  \& \& \&  \& \&  \&  \& \Z\{h_w,e_w\} \arrow[r, symbol=\mathop{=}] \& \pi_1\partial M_w 
\end{tikzcd}
\end{equation*}
\normalsize
Note that the lower-left entry of $A$ is non-zero, so the gluing between $M_v$ and $M_w$ is not fiber-parallel. Thus, \autoref{JSJ-graph} is indeed the JSJ-decomposition of $M$, by \cite[Proposition 1.6.2]{AFW15}.

Similarly, let $N$ be  defined by the gluing maps:
\normalsize
\begin{equation*}
\begin{tikzcd}[ampersand replacement=\&, column sep=tiny,row sep=1em ]
\pi_1T \arrow[d, symbol=\subseteq] \arrow[rrrr, "\psi"] \&  \&  \&  \& \pi_1\partial_1Y \arrow[d, symbol=\subseteq] \&                                \&   \&  \& \& \&  \&  \&                   \&  \& \\
\pi_1X  \arrow[rrrr,dash,dashed]                              \&  \&  \&  \& \pi_1Y \arrow[d, symbol=\supseteq]   \arrow[rrrrrrrrrr,dash,dashed]        \&                                   \&  \&  \&  \& \&  \& \& \&  \& \pi_1N_w \arrow[d, symbol=\supseteq] \\
                                      \&  \&  \&  \& \pi_1\partial_2Y \arrow[r, symbol=\mathop{=}] \& {\Z\{h,e_2\}} \arrow[rrrrrrrr, "{B=\left(\begin{smallmatrix} -q' & \ast  \\ p & \ast \end{smallmatrix}\right)}"'] \&  \& \& \&  \& \&  \&  \& \Z\{h_w',e_w'\} \arrow[r, symbol=\mathop{=}] \& \pi_1\partial N_w 
\end{tikzcd}
\end{equation*}
\normalsize

We show that $\widehat{\pi_1M}\cong\widehat{\pi_1N}$ by constructing a congruent isomorphism between the JSJ-graphs of profinite groups. 
We first deal with the Hempel pair $(M_w,N_w)$. Let $$\rho=\kappa\frac{q}{p}-\frac{q'}{p}\in \Z^{-1}\widehat{\Z}.$$ In fact, $\rho\in \widehat{\Z}$ since $q'\equiv \kappa q\pmod p$. 
Then, \cite[Theorem 10.7]{Wil18} constructs an isomorphism $f_w:\widehat{\pi_1M_w}\ttt\widehat{\pi_1N_w}$ respecting the peripheral structure, with scale type $(\kappa,1)$, such that $$f_w(h_w)=(h_w^{\prime})^{ \kappa}\quad\text{and}\quad f_w(e_w)=ge'_wg^{-1}{(h_w')}^\rho$$ for some $g\in \ker (o:\widehat{\pi_1N_w}\to \{\pm 1\})$. Thus, restricting on the peripheral subgroup $\overline{\pi_1\partial M_w}$, $$C_g^{-1}\circ f_w \begin{pmatrix} h_w &e_w\end{pmatrix}=\begin{pmatrix} h_w'&e'_w \end{pmatrix} \begin{pmatrix}\kappa & \rho \\ 0 & 1\end{pmatrix}.$$

We claim that 
$$B^{-1}\begin{pmatrix}\kappa & \rho \\ 0 & 1\end{pmatrix}A=\begin{pmatrix} 1 & \beta \\ 0 & \kappa \end{pmatrix},\quad \text{for some }\beta\in \widehat{\Z}.$$
%
In fact, $B^{-1}=\left(\begin{smallmatrix} \ast & \ast \\ p & q'\end{smallmatrix}\right)$
, so $B^{-1}\left(\begin{smallmatrix}\kappa & \rho \\ 0 & 1\end{smallmatrix}\right) A=\left(\begin{smallmatrix}\ast & \ast \\  p\kappa & p\rho+q'\end{smallmatrix}\right)A= \left(\begin{smallmatrix}\ast & \ast \\ p(-q\kappa +p\rho+q') &\ast\end{smallmatrix}\right)$. Recall that  $\rho=\kappa\frac{q}{p}-\frac{q'}{p}$, so $p(-q\kappa +p\rho+q')=0$, i.e. 
$B^{-1}\left(\begin{smallmatrix}\kappa & \rho \\ 0 & 1\end{smallmatrix}\right)A=\left(\begin{smallmatrix} \alpha & \beta \\ 0 & \gamma \end{smallmatrix}\right)$ 
is an upper triangle matrix. In addition, 
$\left(\begin{smallmatrix} -\alpha q' & \ast \\ \ast & \ast\end{smallmatrix}\right)= B\left(\begin{smallmatrix} \alpha & \beta \\ 0 & \gamma \end{smallmatrix}\right)=\left(\begin{smallmatrix}\kappa & \rho \\ 0 & 1\end{smallmatrix}\right)A=\left(\begin{smallmatrix} -\kappa q + \rho p & \ast \\ \ast & \ast\end{smallmatrix}\right)=\left(\begin{smallmatrix} -q' & \ast \\ \ast & \ast\end{smallmatrix}\right)$. Thus, $\alpha =1$ since $q'\in\Z\setminus\{0\}$ and  $\widehat{\Z}$ is torsion-free. Note that $\gamma=\det \left(\begin{smallmatrix} \alpha & \beta \\ 0 & \gamma \end{smallmatrix}\right)=\det(B^{-1}\left(\begin{smallmatrix}\kappa & \rho \\ 0 & 1\end{smallmatrix}\right)A)= \kappa$. Thus, $B^{-1}\left(\begin{smallmatrix}\kappa & \rho \\ 0 & 1\end{smallmatrix}\right)A=\left(\begin{smallmatrix} \alpha & \beta \\ 0 & \gamma \end{smallmatrix}\right)=\left(\begin{smallmatrix} 1 & \beta \\ 0 & \kappa \end{smallmatrix}\right)$.

Then, we deal with the intermediate manifold $Y= F\times S^1$. Indeed, since ${\#\partial F \ge 3}$, the fundamental group $\pi_1F$ can be presented as a free group over the generators $e_1,e_2,b_3,\cdots, b_r$. 
Thus,  
$\pi_1Y=\left\langle e_1,e_2,b_3,\cdots, b_r\right\rangle \times \left\langle h\right\rangle =F_r\times \Z$. 
It follows from (\ref{EQU: product group}) 
that $\widehat{F_r\times \Z}\cong \widehat{F_r}\times \widehat{\Z}$. 
Let $\tau \in \Aut(\widehat{\pi_1Y})\cong \Aut(\widehat{F_r}\times\widehat{\Z})$ be defined by $$\tau(e_1)=e_1,\; \tau(e_2)=e_2^\kappa h^\beta,\; \tau(b_i)=b_i,\; \text{and  }\tau(h)=h.$$  In particular, 
restricting on the peripheral subgroups $\pi_1\partial_1Y$ and $\pi_1\partial_2Y$, we have 
\begin{align*}
&\tau\begin{pmatrix} h& e_1\end{pmatrix} =\begin{pmatrix} h & e_1\end{pmatrix}\begin{pmatrix} 1 & 0 \\ 0 & 1\end{pmatrix};\\
&\tau\begin{pmatrix} h& e_2\end{pmatrix} =\begin{pmatrix} h & e_2\end{pmatrix}\begin{pmatrix} 1 & \beta \\ 0 & \kappa\end{pmatrix}.
\end{align*}
Note that $\tau$ does NOT respect the peripheral structure whenever $\kappa\neq 1$, as shown by \autoref{PROP: Hempel pair}.

Combining these constructions, $(id_{\widehat{\pi_1X}}, \tau, f_w)$ establishes a congruent isomorphism between the JSJ-graphs of profinite groups $(\widehat{\mathcal{G}_M},\Gamma)$ and $(\widehat{\mathcal{G}_N},\Gamma)$, as shown by the following two commutative diagrams. 

\begin{equation*}
\begin{tikzcd}[column sep=tiny, row sep=large]
\widehat{\pi_1X} \arrow[d, "id"'] \arrow[r,symbol=\supseteq] & \overline{\pi_1T} \arrow[d, "id"] \arrow[rrrrr, "\widehat{\psi}"] &  & & &  & \overline{\pi_1\partial_1Y} \arrow[r, symbol=\subseteq] \arrow[d, "id"'] & \widehat{\pi_1Y} \arrow[d, "\tau"] \arrow[r,symbol=\cong] & \widehat{F_r}\times \widehat{\Z} \\
\widehat{\pi_1X} \arrow[r,symbol=\supseteq]                  & \overline{\pi_1T} \arrow[rrrrr, "\widehat{\psi}"']                &  &  & & & \overline{\pi_1\partial_1Y} \arrow[r,symbol=\subseteq]                  & \widehat{\pi_1Y} \arrow[r,symbol=\cong]                   & \widehat{F_r}\times \widehat{\Z}
\end{tikzcd}
\end{equation*}

\begin{equation*}
\begin{tikzcd}[ampersand replacement=\&, column sep=tiny]
\widehat{\pi_1Y} \arrow[dd, "\tau"'] \arrow[r,symbol=\supseteq] \& \overline{\pi_1\partial_2Y} \arrow[r,symbol=\mathop{=}] \& {\widehat{\Z}\{h,e_2\}} \arrow[dd, "{\left(\begin{smallmatrix} 1 & \beta \\ 0 & \kappa \end{smallmatrix}\right)}"] \arrow[rrrrrrrr, "{A=\left(\begin{smallmatrix} -q & \ast \\ p & \ast \end{smallmatrix}\right)}"] \&  \& \&  \& \& \& \&  \& {\widehat{\Z}\{h_w,e_w\}} \arrow[r,symbol=\mathop{=}] \arrow[dd, "{\left(\begin{smallmatrix} \kappa & \rho \\ 0 & 1 \end{smallmatrix}\right)}"'] \& \overline{\pi_1\partial M_w} \arrow[r,symbol=\subseteq] \& \widehat{\pi_1M_w} \arrow[dd, "f_w"] \\
                                               \&                                       \&                                                              \&  \&  \& \&  \&                                           \& \&   \&        \&                                        \&                                      \\
\widehat{\pi_1Y} \arrow[r,symbol=\supseteq]                     \& \overline{\pi_1\partial_2Y} \arrow[r,symbol=\mathop{=}] \& {\widehat{\Z}\{h,e_2\}} \arrow[rrrrrrrr, "{B=\left(\begin{smallmatrix} -q' & \ast \\ p & \ast \end{smallmatrix}\right)}"']               \& \&   \&  \&  \&  \&  \& \& {\widehat{\Z}\{h_w',e_w'\}} \arrow[r,symbol=\mathop{=}]                  \& \overline{\pi_1\partial N_w} \arrow[r,symbol=\subseteq] \& \widehat{\pi_1N_w}                  
\end{tikzcd}
\end{equation*}

Therefore, according to \autoref{LEM: congruent isom induce isom} and \autoref{PROP: JSJ Efficient}, we obtain an isomorphism  
\begin{equation*}
\begin{tikzcd}
\widehat{\pi_1M}\cong \Pi_1(\widehat{\mathcal{G}_M},\Gamma)\arrow[r,"\cong"] &   \Pi_1(\widehat{\mathcal{G}_N},\Gamma)\cong \widehat{\pi_1N}
\end{tikzcd}
\end{equation*}
 which does not respect the peripheral structure.

Denote the Seifert parts of $M$ and $N$ as $M_S$ and $N_S$ respectively, which correspond to the subgraph consisting of $v$, $w$ and the edge adjoining them. 
We now show that 
$\pi_1M_S$ and $\pi_1N_S$ are not isomorphic. 

Indeed, if there is an isomorphism $\varphi:\pi_1M_S\ttt \pi_1N_S$, the group theoretical JSJ-decomposition \cite[Theorem 3.1]{SS01} implies a congruent isomorphism $\varphi_\bullet$ between their JSJ-graph of fundamental groups. Note that $\pi_1M_w\not \cong \pi_1Y= F_r\times \Z$ since $M_w$ contains an exceptional fiber. Therefore, $\varphi$ must induce an isomorphism $\varphi_w:\pi_1M_w\ttt\pi_1N_w$. Since both $M_w$ and $N_w$ have exactly one boundary component, which is adjoint to $Y$, the isomorphism $\varphi_w$ obtained from the congruent isomorphism indeed respects the peripheral structure. Thus, \autoref{Wald} implies that $M_w$ and $N_w$ are homeomorphic, which is a contradiction with our construction.
\end{proof}

In particular, this construction yields $\widehat{\pi_1M}\cong \widehat{\pi_1N}$ while $\pi_1M\not\cong \pi_1N$.

\subsection{Profinite almost rigidity}\label{subsec: almost rigid}
In this subsection, we prove \autoref{Mainthm: Almost rigid}. 

Let $\mathscr{M}$ denote the class of compact, orientable 3-manifolds; let $\mathscr{M}'$ denote the sub-class of $\mathscr{M}$ consisting of   boundary-incompressible 3-manifolds; and let $\mathscr{M}_0$ denote the sub-class of $\mathscr{M}'$ consisting of irreducible 3-manifolds with empty or incompressible toral boundary.  Then  \autoref{Mainthm: Almost rigid} is restated as follows.

\begin{THMC}
Any compact, orientable 3-manifold with empty or  toral boundary is profinitely almost rigid in $\mathscr{M}$.
\end{THMC}

We shall first reduce \autoref{Mainthm: Almost rigid} to the most crucial case.

\begin{lemma}\label{LEM: compression}
Suppose $M\in \mathscr{M}$. Then, there exists $M'\in \mathscr{M}'$ such that $\pi_1M\cong \pi_1M'$. In addition, if $M$ has empty or toral boundary, then so does $M'$. 
\end{lemma}
\begin{proof}
This is a standard application of the loop theorem \cite{Sta60}. Indeed, $M'$ can be constructed starting from $M$ by repeatedly attaching 2-handles to the compressible boundary components of $M$ while preserving the fundamental group, and then capping off the emerging boundary spheres by 3-handles, until there are no more compressible boundary components. This process can also be viewed as attaching suitable compression-bodies to the compressible boundary components of $M$. In particular, if  $M$ has empty or toral boundary, then $M'$ constructed above is a Dehn filling of $M$, so $M'$ also has empty or toral boundary.
%
\end{proof}

\begin{corollary}\label{COR: reduce to incompressible}
\autoref{Mainthm: Almost rigid} holds if any compact, orientable 3-manifold with empty or incompressible toral boundary is profinitely almost rigid in $\mathscr{M}'$.
\end{corollary}
\begin{proof}
Suppose $M\in\mathscr{M}$ has empty or toral boundary. Let $M'\in \mathscr{M}'$ be the manifold constructed in \autoref{LEM: compression}   such that $\pi_1M\cong \pi_1M'$. We claim that 
\begin{equation}\label{EQU:10.2}
\Delta_{\mathscr{M}}(M)=\bigcup_{N\in \Delta_{\mathscr{M}'}(M')}\left\{ M_\star\in \mathscr{M}\mid \pi_1M_\star\cong \pi_1 N\right\}.
\end{equation}
The ``$\supseteq$'' direction of (\ref{EQU:10.2}) is clear. 
On the other hand, for any $M_\star\in \Delta_{\mathscr{M}(M)}$, let $M_\star'\in \mathscr{M}'$ be constructed according to  \autoref{LEM: compression}   such that $ {\pi_1M_\star}\cong  {\pi_1M_\star'}$. Then $\widehat{\pi_1M_\star'}\cong \widehat{\pi_1M_\star}\cong \widehat{\pi_1M}\cong \widehat{\pi_1M'}$, i.e. $M_\star'\in \Delta_{\mathscr{M}'}(M')$. This implies that $M_\star$ belongs to the right-hand side of (\ref{EQU:10.2}).  

Note that $M'$ has empty or incompressible toral boundary. If $M'$ is profinitely almost rigid in $\mathscr{M}'$, then $\Delta_{\mathscr{M}'}(M')$ is finite. In addition, for each $N\in \Delta_{\mathscr{M}'}(M')$, the set $\{ M_\star\in \mathscr{M}\mid \pi_1M_\star\cong \pi_1 N\}$ is finite according to \autoref{Joh}. Thus, the union $\Delta_{\mathscr{M}}(M)$ is finite, finishing the proof.
\end{proof}

Moreover, the profinite detection of prime decomposition (\autoref{THM: Detect Kneser-Milnor}) reduces the proof to the sub-class of irreducible 3-manifolds, and earlier results on $Sol$-manifolds \cite{GPS80}, hyperbolic manifolds \cite{Liu23} and graph manifolds  (\autoref{THM: Graph mfd almost rigid}) further reduces the proof to the profinite  rigidity of mixed 3-manifolds, which is the crucial part.

\begin{lemma}\label{LEM: reduce to mixed}
\autoref{Mainthm: Almost rigid} holds if any mixed 3-manifold is profinitely almost rigid in $\mathscr{M}_0$.
\end{lemma}
\begin{proof}
According to \autoref{COR: reduce to incompressible}, it suffices to show that any $M\in \mathscr{M}'$ with empty or incompressible toral boundary is profinitely almost rigid in $\mathscr{M}'$. In fact, for any $N\in \Delta_{\mathscr{M}'}(M)$, \autoref{COR: Determine boundary} implies that $N$ also has empty or toral boundary. 

In addition, suppose the prime decomposition of $M$ is $M\cong M_1\# \cdots \# M_m\#(S^2\times S^1)^{\# r}$, where $M_i\in \mathscr{M}_0$. Then \autoref{THM: Detect Kneser-Milnor} implies the  prime decomposition  $N\cong N_1\#\cdots N_m\#(S^2\times S^1)^{\#r}$, where  $N_i \in \Delta_{\mathscr{M}_0}(M_i)$ for each $1\le i\le m$. Note that the homeomorphism type of $N$ can be determined by the homeomorphism type of each $N_i$ together with an orientation chosen for each $N_i$. Thus, $\#\Delta_{\mathscr{M}'}(M)<+\infty$ if and only if $\#\Delta_{\mathscr{M}_0}(M_i)<+\infty$ for each $1\le i \le m$.

Therefore, \autoref{Mainthm: Almost rigid} holds if  $\Delta_{\mathscr{M}_0}(M )$ is finite for any $M \in \mathscr{M}_0$.

When $M $ is a closed $Sol$-manifold,  any $N\in  \Delta_{\mathscr{M}_0}(M )$ is also closed according to \autoref{COR: Determine boundary}. \cite[Theorem 8.4]{WZ17} further implies that $N$ is also a $Sol$-manifold. The conclusion follows from the profinite almost rigidity of $Sol$-manifolds deduced from the profinite almost rigidity of polycyclic-by-finite groups \cite{GPS80}.

Similarly, when $M$ is either a finite-volume hyperbolic 3-manifold, a Seifert fibered space, or a  graph manifold, any $N\in  \Delta_{\mathscr{M}_0}(M )$ is also either finite-volume hyperbolic, Seifert fibered, or a  graph manifold according to \autoref{THM: Profinite Isom up to Conj}, and the finiteness of $\Delta_{\mathscr{M}_0}(M)$ follows from \autoref{Liu1}, \autoref{PROP: Seifert almost rigid}, or \autoref{THM: Graph mfd almost rigid} respectively.  

The only remaining case is that $M $ is mixed. Thus, \autoref{Mainthm: Almost rigid} holds if any mixed 3-manifold $M \in \mathscr{M}_0$ is profinitely almost rigid in $\mathscr{M}_0$.
\end{proof}

We are now ready to complete the proof of \autoref{Mainthm: Almost rigid}.

\begin{proof}[Proof of \autoref{Mainthm: Almost rigid}]
Based on \autoref{LEM: reduce to mixed}, it suffices to show that any mixed 3-manifold $M$ is profinitely almost rigid in $\mathscr{M}_0$.
According to \autoref{THM: Profinite Isom up to Conj}, all manifolds in $ \Delta_{\mathscr{M}_0}(M)$ can be identified over the same JSJ-graph $\Gamma$ with each vertex fixed as either hyperbolic type or Seifert type, so that for each $N\in \Delta_{\mathscr{M}_0}(M)$, there exists a congruent isomorphism $f_\bullet:(\widehat{\mathcal{G}_M},\Gamma)\to(\widehat{\mathcal{G}_N},\Gamma)$ such that $\F=id_{\Gamma}$. 

Let us first list some notations related to the oriented graph $\Gamma$. 
Let $V_{\text{hyp}}(\Gamma)$ denote the collection of hyperbolic vertices in $\Gamma$, and each $v\in V_{\text{hyp}}(\Gamma)$ is also viewed as a subgraph of $\Gamma$. Let $\Lambda_1,\cdots, \Lambda_m$ be the connected components of the full subgraph spanned by the Seifert vertices. 
Let $\Gamma_\ast$ be an oriented graph obtained from $\Gamma$ by collapsing each $\Lambda_i$ to a vertex. 
Then we can identify  $V(\Gamma_\ast)$ as $V_{\text{hyp}}(\Gamma)\cup\{\Lambda_1,\cdots,\Lambda_m\}$, and identify  $E(\Gamma_\ast)$ as the collection of edges in $\Gamma$ adjoint to a hyperbolic vertex.  Let $\delta_{0,1}$ denote the relation map in $\Gamma_\ast$, which is to say for $e\in E(\Gamma_\ast)$ and $\alpha\in \{0,1\}$,  $\delta_{\alpha}(e)=\Upsilon\in V(\Gamma_\ast)$ if and only if $d_{\alpha}(e)$ belongs to the subgraph $\Upsilon$. 

For each $N\in \Delta_{\mathscr{M}_0}(M)$ and $\Upsilon\in V(\Gamma_\ast)$,  denote by $N_{\Upsilon}$  the submanifold of $N$ corresponding to the subgraph $\Upsilon$, which is either a hyperbolic manifold or a graph manifold. For any $e\in E(\Gamma_\ast)$ and $\alpha\in \{0,1\}$, let $\partial_{\alpha,e} N_{\delta_\alpha(e)}$ denote the boundary component of $N_{\delta_\alpha(e)}$ obtained from the $\delta_\alpha$-side of the JSJ-torus corresponding to $e$. We define
$$
\mathcal{P}_{\Upsilon}^N=\left\{ \partial_{\alpha,e} N_{\delta_\alpha(e)} \mid \delta_\alpha(e)=\Upsilon,\; e\in E(\Gamma_\ast)\text{ and }\alpha\in \{0,1\}\right\}
$$
 as a collection of boundary components of $N_\Upsilon$. 

Then there are profinite isomorphisms 
\begin{equation*}
\begin{tikzcd}[row sep=tiny]
\widehat{\pi_1M_{\Upsilon}}\cong \Pi_1(\widehat{\mathcal{G}_M},\Upsilon) \arrow[r,"\cong"] & \Pi_1(\widehat{\mathcal{G}_N},\Upsilon)\cong  \widehat{\pi_1N_{\Upsilon}},\quad \Upsilon\in V(\Gamma_\ast)
\end{tikzcd}
\end{equation*}
induced by the congruent isomorphism $f_\bullet$ on the subgraphs $\Upsilon$ according to  \autoref{LEM: congruent isom induce isom}. 


According to \autoref{Liu1} for the hyperbolic pieces, and  \autoref{THM: Graph mfd almost rigid} for the graph submanifolds, the profinite completion determines the homeomorphism type of $N_v$ ($v\in V_{\text{hyp}}$) and $N_{\Lambda_i}$ to finitely many possibilities. 
To be more precise, we can divide $ \Delta_{\mathscr{M}_0}(M)$ into finitely many subsets $\Delta_1,\cdots,\Delta_n$, such that for any $N$ and $N'$ belonging to the same subset $\Delta_j$, there is a homeomorphism $\Psi_{\Upsilon}: N'_\Upsilon\ttt N_\Upsilon$ for each $\Upsilon \in V(\Gamma_\ast)$ so that $\Psi_{\delta_\alpha(e)}(\partial_{\alpha,e}N'_{\delta_\alpha(e)})=\partial_{\alpha,e} N_{\delta_\alpha(e)}$ for each $e\in E(\Gamma_\ast)$ and $\alpha\in \{0,1\}$. 
In other words, the manifolds in $\Delta_j$ have the same hyperbolic and graph manifold pieces, with the same boundary pairing. 
In order to show the finiteness of each $\Delta_j$, it remains to determine how these pieces are glued up along the JSJ-tori. 

For each $\Delta_j$, we fix a representative $N\in \Delta_j$. Then, for any $N'\in \Delta_j$, we fix the aforementioned homeomorphisms $\Psi_\Upsilon: N'_\Upsilon\to N_\Upsilon$, and identify $N'_\Upsilon$ with $N_\Upsilon$ through these homeomorphisms in the following context. In addition, we fix a congruent isomorphism $f^{N'}_\bullet:(\widehat{\mathcal{G}_{N'}},\Gamma)\to (\widehat{\mathcal{G}_{N}},\Gamma)$ such that $(f^{N'})^\dagger=id_{\Gamma}$. 
Let \begin{equation*}
\begin{tikzcd}
f^{N'}_\Upsilon:\;\widehat{\pi_1N_{\Upsilon}}=\widehat{\pi_1N'_{\Upsilon}} \cong \Pi_1(\widehat{\mathcal{G}_{N'}},\Upsilon)\arrow[r,"\cong"] & \Pi_1(\widehat{\mathcal{G}_N},\Upsilon)\cong  \widehat{\pi_1N_{\Upsilon}},\; \Upsilon\in V(\Gamma_\ast)
\end{tikzcd}
\end{equation*}
be the isomorphisms induced by the congruent isomorphism $f_\bullet$ on the subgraphs $\Upsilon$ according to  \autoref{LEM: congruent isom induce isom}. 
Then, $f^{N'}_{\delta_{\alpha}(e)}$ preserves the conjugacy class of $\overline{\pi_1\partial_{\alpha,e}N_{\delta_\alpha(e)}}$ in $\widehat{\pi_1N_{\delta_\alpha(e)}}$, for any $e\in E(\Gamma_\ast)$ and $\alpha\in \{0,1\}$. 

According to \autoref{THM: Hyperbolic peripheral regular}~(\ref{THM: Hyperbolic peripheral regular (1)}), for each $v\in V_{\text{hyp}}(\Gamma)$, $f^{N'}_v$ is peripheral $\Zx$-regular at all boundary components in $\mathcal{P}^{N'}_v=\mathcal{P}^N_v$. Thus, $f^{N'}_v$ can be viewed as an element in $\XX_{\Zx}(N_v,\mathcal{P}^N_v)$. In addition, any component in $\mathcal{P}^{N'}_{\Lambda_i}=\mathcal{P}^{N}_{\Lambda_i}$ is attached to a hyperbolic piece, i.e. attached to a component in $\mathcal{P}^N_v$ for some $v\in V_{\text{hyp}}(\Gamma)$. Thus, the gluing lemma  (\autoref{Cor: gluing}~(\ref{lemmain-1''})) implies that $f^{N'}_{\Lambda_i}$ is peripheral $\Zx$-regular at all components in $\mathcal{P}^{N}_{\Lambda_i}$. Therefore, $f^{N'}_{\Lambda_i}$ can be viewed as an element $\XX_{\Zx}(N_{\Lambda_i},\mathcal{P}^N_{\Lambda_i})$. 

Moreover, for each $\Upsilon \in V(\Gamma_\ast)$, $N_\Upsilon$ is {\PAA} at $\mathcal{P}^N_\Upsilon$ according to \autoref{COR: Hyperbolic PAA} and \autoref{THM: Hgroup finite index} respectively. 
In other words, $\XX_{h}(N_{\Upsilon},\mathcal{P}^N_{\Upsilon})$ is finite index in $\XX_{\Zx}(N_{\Upsilon},\mathcal{P}^N_{\Upsilon})$. Thus, there exist finitely many representatives $N^{\star}_1,\cdots,N^{\star}_s$ in $\Delta_j$ so that for each $N'\in \Delta_j$, there exists $1\le k\le s$ such that for all $\Upsilon \in V(\Gamma_\ast)$, 
\begin{equation*}
\begin{tikzcd}
((f^{N^\star_k}_\bullet)^{-1}\circ f^{N'}_\bullet)_{\Upsilon}=(f^{N^\star_k}_\Upsilon)^{-1}\circ f^{N'}_\Upsilon: \;\widehat{\pi_1N_{\Upsilon}} \arrow[r,"\cong"] &  \widehat{\pi_1N_{\Upsilon}}
\end{tikzcd}
\end{equation*}
 belongs to $\XX_{h}(N_{\Upsilon},\mathcal{P}^{N}_{\Upsilon})$.

Denote $f'_\bullet = (f^{N^\star_k}_\bullet)^{-1}\circ f^{N'}_\bullet$ as the congruent isomorphism between $(\widehat{\mathcal{G}_{N'}},\Gamma)$ and $(\widehat{\mathcal{G}_{N^\star_k}},\Gamma)$. Then, for each $\Upsilon\in V(\Gamma_\ast)$ there exists a homeomorphism $\Phi_{\Upsilon}: N_{\Upsilon}=N'_{\Upsilon}\ttt (N^\star_k)_{\Upsilon}=N_{\Upsilon}$ such that for each  $\partial_{\alpha,e}N_{\Upsilon}\in \mathcal{P}^N_{\Upsilon}$, $f_{\Upsilon}'$ restricting on $\partial_{\alpha,e}N_{\Upsilon}$ is induced by the  homeomorphism $\Phi_{\Upsilon}$ with coefficient $\lambda_{\alpha,e}\in \Zx$. In particular, for each $e\in E(\Gamma_\ast)$, $\lambda_{0,e}=\pm \lambda_{1,e}$ according to  \autoref{Cor: gluing}~(\ref{lemmain-1''}).

In order to apply the gluing lemma, one actually requires $\lambda_{0,e}= \lambda_{1,e}$. Similar to the proof of \autoref{THM: Hgroup finite index}, one can choose more (a total of at most $2^{\#E(\Gamma_\ast)}s$) representatives $N^{\star\star}_1,\cdots, N^{\star\star}_t$  in $\Delta_j$, such that for each $N'\in \Delta_j$, there exists $1\le l \le t$ such that $f'_\Upsilon= (f^{N^{\star\star}_l}_\Upsilon)^{-1}\circ f^{N'}_\Upsilon\in \XX_{h}(N_{\Upsilon},\mathcal{P}^N_{\Upsilon})$, and $\lambda_{0,e}=\lambda_{1,e}$ for each $e\in E(\Gamma_{\ast})$.  In this case, according to \autoref{Cor: gluing}~(\ref{lemmain-2''}), the homeomorphisms $\{\Phi_{\Upsilon}\}_{\Upsilon\in V(\Gamma_\ast)}$  are compatible with the gluing maps of $N'$ and $N^{\star\star}_l$ at each JSJ-torus corresponding to $e\in E(\Gamma_\ast)$. As a result, $\{\Phi_{\Upsilon}\}_{\Upsilon\in V(\Gamma_\ast)}$  yields an ambient homeomorphism $\Phi: N'\ttt N^{\star\star}_l$.

Therefore, $\Delta_j=\left\{N^{\star\star}_1,\cdots, N^{\star\star}_t\right\}$ is finite, and $ \Delta_{\mathscr{M}_0}(M)=\Delta_1\cup\cdots\cup\Delta_n$ is also finite, which finishes the proof.
\end{proof}

\appendix
\section{Proof of \autoref{THM: finite quotients with marked subgroups}}\label{APPproof}
\autoref{THM: finite quotients with marked subgroups} might be well-known to experts. As we haven't found a documented version of its proof, we supplement this in the appendix.

\newtheorem*{Appproof}{\autoref{THM: finite quotients with marked subgroups}}
\begin{Appproof}
Let $G$ and $G'$ be finitely generated groups. Suppose $H_1,\cdots, H_m$ are subgroups of $G$, and $H_1',\cdots, H_m'$ are subgroups of $G'$.
Then the following are equivalent.
\begin{enumerate}[leftmargin=*]
\item\label{c1} $\mathcal{C} (G;H_1,\cdots, H_m)=\mathcal{C} (G';H_1',\cdots, H_m')$.
\item\label{c3} There exists an isomorphism as profinite groups $f:\widehat{G}\to \widehat{G'}$, such that $f(\overline{H_i})$ is a conjugate of $\overline{H_i'}$ in $\widehat{G'}$ for each $1\le i \le m$.
\item\label{c2} There exists an isomorphism as abstract groups $f:\widehat{G}\to \widehat{G'}$, such that $f(\overline{H_i})$ is a conjugate of $\overline{H_i'}$ in $\widehat{G'}$ for each $1\le i \le m$.
\end{enumerate}
\end{Appproof}

\begin{proof}
(\ref{c1})$\Rightarrow$(\ref{c3}): Let $\mathcal{U}_n=\{U\le G\mid [G:U]\le n\}$, and $\mathcal{V}_n=\{V\le G'\mid [G':V]\le n\}$. Since $G$ and $G'$ are finitely generated, $\mathcal{U}_n$ and $\mathcal{V}_n$ are finite for each $n$. Let $U_n=\bigcap_{U\in\mathcal{U}_n} U$ and $V_n=\bigcap_{V\in\mathcal{V}_n}V$, then $U_n$ and $V_n$ are finite-index normal subgroups of $G$ and $G'$ respectively.

For each $n$, since $\mathcal{C} (G;H_1,\cdots, H_m)=\mathcal{C} (G';H_1',\cdots, H_m')$, there exists $K\lhd_f G$ and an isomorphism $\psi: G/K\ttt G'/V_n$ such that $\psi(H_i/(H_i\cap K))$ is a conjugate of $H_i'/(H_i'\cap V_n)$ for each $1\le i\le m$. 

We claim that $K= U_n$. 

For each $V\in \mathcal{V}_n$, consider the homomorphism 
\begin{equation*}
\begin{centering}
\begin{tikzcd}
f_V:G \arrow[r, two heads] & G/K \arrow[r, "\cong"',"\psi"] & G'/V_n \arrow[r, "\pi_V", two heads] & G'/V.
\end{tikzcd}
\end{centering}
\end{equation*}
Since $\bigcap_{V\in \mathcal{V}_n} \mathrm{ker}(\pi_V)=1$, it follows that $\bigcap_{V\in \mathcal{V}_n} \mathrm{ker}(f_V)=K$. On the other hand, $\abs{G/\mathrm{ker}(f_V)}=\abs{G'/V}\le n$, which implies that $\mathrm{ker}(f_V)\in \mathcal U_n$. Therefore, $K=\bigcap\mathrm{ker}(f_V) \supseteq U_n$. 
In particular, this shows that $\abs{G/U_n}\ge \abs{G/K}=\abs{G'/V_n}$. By switching $G$ and $G'$, we can symmetrically prove that $\abs{G'/V_n}\ge \abs{G/U_n}$. Thus, $\abs{G/U_n}= \abs{G'/V_n}=\abs{G/K}$. This implies that $K=U_n$, i.e. $G/U_n\cong G'/V_n$.

\newsavebox{\xxxyyy}
\begin{lrbox}{\xxxyyy}
$
X_n=\left\{\psi:G/U_n\ttt G'/V_n\mid \psi(H_i/(H_i\cap U_n))\sim H_i'/(H_i'\cap V_n),\,\forall 1\le i\le m\right\},
$
\end{lrbox}

For each $n$, denote 
\begin{equation*}
\scalebox{0.97}{\usebox{\xxxyyy}}
\end{equation*}
where $\sim$ means conjugation in $G'/V_n$. The above construction shows that $X_n$ is non-empty, and $X_n$ is also finite due to the finiteness of $G/U_n$ and $G'/V_n$. 

Notice that $U_{n-1}/U_n$ is the intersection of all subgroups in $G/U_n$ with index no more than $n-1$, 
and so is $V_{n-1}/V_{n}$ in $G'/V_n$. Thus, any element $\psi\in X_n$ sends $U_{n-1}/U_n$ to $V_{n-1}/V_n$. Consequently, $\psi$ induces an isomorphism 
\begin{equation*}
\begin{tikzcd}
\widetilde{\psi}:G/U_{n-1} \arrow[r,"\cong"] &  G'/V_{n-1},
\end{tikzcd}
\end{equation*}
which sends $$\frac{H_i/(H_i\cap U_n)}{(H_i/(H_i\cap U_n) )\cap( U_{n-1}/U_n)}=\frac{H_i}{H_i\cap U_{n-1}}$$ to a conjugate of $$\frac{H_i'/(H_i'\cap V_n)}{(H_i'/(H_i'\cap V_n) )\cap(V_{n-1}/V_n)}=\frac{H_i'}{H_i'\cap V_{n-1}}.$$  Therefore, we have constructed a map $X_n\to X_{n-1}$ sending $\psi\in X_n$ to $\widetilde{\psi}\in X_{n-1}$, which makes  $\left\{X_n\right\}$ an inverse system of non-empty finite sets over a directed partially ordered set. It follows from \cite[Proposition 1.1.4]{RZ10} that $\limi_n X_n$ is non-empty.

Choose $(\psi_n)_{n\ge 0}\in \limi X_n$. Indeed, we obtain a series of compatible isomorphisms $\psi_n:G/U_n\stackrel{\cong}{\longrightarrow}G'/V_n$. As $U_n$ and $V_n$ are cofinal systems of finite-index normal subgroups in $G_1$ and $G_2$, by taking inverse limits, we derive an isomorphism of profinite groups $$\hat{\psi}: \widehat{G}=\limi G/U_n\stackrel{\cong}{\longrightarrow}\limi G'/V_n=\widehat{G'}.$$

Denote $$Y_{n,i}=\left\{ g_{n,i}\in G'/V_n\mid \psi_n(H_i/(H_i\cap U_n))=g_{n,i} (H_i'/(H_i'\cap V_n)) g_{n,i}^{-1}\right\}.$$ Then $Y_{n,i}$ is a non-empty finite set. Note that the maps $(\psi_n)$ are compatible, and thus the quotient map $G'/V_n \twoheadrightarrow G'/V_{n-1}$ induces a map $Y_{n,i} \to Y_{n-1,i}$, making $\{Y_{n,i}\}$ an inverse system. Again, \cite[Proposition 1.1.4]{RZ10} implies that $\limi_n Y_{n,i}$ is non-empty. 
Note that $$\overline{H_i}=\limi H_i/(H_i\cap U_n) \subseteq \widehat{G},\; \text{and }\quad\overline{H_i'}=\limi H_i'/(H_i'\cap V_n) \subseteq \widehat{G'}.$$  
Thus, by choosing $\hat{g}_i\in \limi Y_{n,i} \subseteq \limi G'/V_n =\widehat{G'}$,  the above construction shows that $\hat{\psi}(\overline{H_i})=\hat{g}_i \cdot \overline{H_i'} \cdot \hat{g}_i^{-1}$.

(\ref{c3})$\Rightarrow$(\ref{c2}) is trivial.

(\ref{c2})$\Rightarrow$(\ref{c1}): Condition (\ref{c2}) actually implies that $$\mathcal{C} (\widehat{G};\overline{H_1},\cdots, \overline{H_m})=\mathcal{C} (\widehat{G'};\overline{H_1'},\cdots, \overline{H_m'}).$$ Thus, it suffices to show that 
\begin{equation}\label{cplEQU1}
\mathcal{C} (G; H_1,\cdots, H_m)=\mathcal{C} (\widehat{G};\overline{H_1},\cdots, \overline{H_m}),\end{equation}
 and 
\begin{equation}\label{cplEQU2}
\mathcal{C} (G'; H_1',\cdots, H_m')=\mathcal{C} (\widehat{G'};\overline{H_1'},\cdots, \overline{H_m'}).\end{equation}


Since $G$ is finitely generated, it follows from a deep theorem of Nikolov-Segal \cite{NS03} that every finite-index subgroup in $\widehat{G}$ is open.  
Thus, \autoref{THM: correspondence of subgroup} implies that there is a one-to-one correspondence between the finite-index normal subgroups of $G$ and the finite-index normal subgroups of $\widehat{G}$, by sending $U\lhd_fG$ to $\overline{U}\lhd_oG$. We claim that 
\begin{equation}\label{cplEQU3}
\scalebox{.89}{
$\left[\quo{G}{U};\,\quo{H_1}{(H_1\cap U)},\cdots,\quo{H_m}{(H_m\cap U)}\right]$}=\scalebox{.8}{$\left[\quo{\widehat{G}}{\overline{U}};\,\quo{\overline{H_1}}{(\overline{H_1}\cap \overline U)},\cdots,\quo{\overline {H_m}}{(\overline {H_m}\cap \overline U)}\right].$
}
\end{equation}
\normalsize

In fact, let $\iota: G\to \widehat{G}$ be the canonical homomorphism. Consider the homomorphism 
\begin{equation*}
\begin{tikzcd}
\phi:\;G \arrow[r, "\iota"] & \widehat{G} \arrow[r] & \widehat{G}/\overline{U}.
\end{tikzcd}
\end{equation*}
 Note that $\phi(G)$ is dense in $\widehat{G}/\overline{U}$, since $\iota(G)$ is dense in $\widehat{G}$. In addition, since $\overline{U}\lhd_o \widehat{G}$, the quotient group $\widehat{G}/\overline{U}$ is a finite discrete group. Therefore, $\phi$ is surjective. On the other hand, $\ker(\phi)=\iota^{-1}(\overline{U})=U$. Thus, $\phi$ induces an isomorphism 
\begin{equation*}
\begin{tikzcd}
\phi_0: G/U\arrow[r,"\cong"] &  \widehat{G}/\overline{U}
\end{tikzcd}
\end{equation*}
between finite groups.

By construction, for each $1\le i \le m$, $\phi_0(H_i/(H_i\cap U))$ is a dense subgroup of $\overline{H_i}/(\overline{H_i}\cap \overline{U})$. Since $\overline{H_i}/(\overline{H_i}\cap \overline{U})\le \widehat{G}/\overline{U}$ is finite and discrete, we conclude that $\phi_0(H_i/(H_i\cap U))=\overline{H_i}/(\overline{H_i}\cap \overline{U})$. 

Therefore, we obtain (\ref{cplEQU3}) through $\phi_0$, which implies (\ref{cplEQU1}). Similarly, (\ref{cplEQU2}) holds, which finishes the proof of this theorem.
%
\end{proof}

\bibliographystyle{amsplain}

\providecommand{\bysame}{\leavevmode\hbox to3em{\hrulefill}\thinspace}
\providecommand{\MR}{\relax\ifhmode\unskip\space\fi MR }
\providecommand{\MRhref}[2]{%
  \href{http://www.ams.org/mathscinet-getitem?mr=#1}{#2}
}
\providecommand{\href}[2]{#2}

\end{sloppypar}
\end{document}